\documentclass[12pt]{article}

\usepackage{epic,eepic,euscript,verbatim,amsmath,amssymb,amsthm,amsfonts,afterpage,float,bm,stmaryrd,paralist,bm,mathrsfs}
\usepackage{color}
\usepackage{multirow}
\usepackage{booktabs}
\usepackage{csvsimple}
\usepackage{gensymb} 
\usepackage{colordvi}
\usepackage{graphicx}

\usepackage{caption}
\usepackage{wrapfig}
\usepackage{algpseudocode}
\usepackage{authblk}
\usepackage{lipsum}

\usepackage[titletoc]{appendix}
\usepackage{setspace}

\makeatletter
\let\c@lofdepth\relax
\let\c@lotdepth\relax
\makeatother
\usepackage{subfigure}

\usepackage[subfigure]{tocloft}

\newtheorem{theorem}{Theorem}[section]
\newtheorem{lemma}{Lemma}[section]

\newtheorem{corollary}{Corollary}[section]
\newtheorem{definition}{Definition}[section]
\newtheorem{remark}{Remark}[section]

\numberwithin{equation}{section}
\numberwithin{figure}{section}

\def\Xint#1{\mathchoice
{\XXint\displaystyle\textstyle{#1}}%
{\XXint\textstyle\scriptstyle{#1}}%
{\XXint\scriptstyle\scriptscriptstyle{#1}}%
{\XXint\scriptscriptstyle\scriptscriptstyle{#1}}%
\!\int}
\def\XXint#1#2#3{{\setbox0=\hbox{$#1{#2#3}{\int}$}
\vcenter{\hbox{$#2#3$}}\kern-.51\wd0}}

\def\dashint{\Xint-}

\newcommand{\reff}[1]{{\rm (\ref{#1})}}

\newcommand{\R}{\mathbb{R}}            

\newcommand{\Z}{\mathbb{Z}}            
\newcommand{\C}{\mathbb{C}}
\newcommand{\N}{\mathbb{N}}
\newcommand{\M}{\mathbb{M}}
\newcommand{\sA}{\mathscr{A}}

\newcommand{\sP}{\mathscr{P}}
\newcommand{\sQ}{\mathscr{Q}}

\newcommand{\bA}{\mathbf A}  
\newcommand{\bD}{\mathbf D}            
\newcommand{\bF}{\mathbf F}
\newcommand{\bv}{\mathbf v}
\newcommand{\br}{\mathbf r}            
\newcommand{\ve}{\varepsilon}          
\newcommand{\p}{\mathcal{P}}

\newcommand{\zhou}[1] {{\color{red} #1}}

\newcommand{\BL}[1]{{\color{blue}[#1 {-BL}]}}

\allowdisplaybreaks

\begin{document}


\title{Finite-Difference Approximations and Local Algorithms for the 
Poisson and Poisson--Boltzmann Electrostatics}


\author[1]{Bo Li}
\author[2]{Qian Yin}
\author[3]{Shenggao Zhou}

\affil[1]{\small{Department of Mathematics, 
University of California, San Diego, La Jolla, California, USA.

Email: bli@math.ucsd.edu
}}

\affil[2]{\small{
Department of Applied Mathematics, The Hong Kong Polytechnic University, 
Hong Kong, China,

and School of Mathematical Sciences, 
Shanghai Jiao Tong University, Shanghai, 
China.

Email: qqian.yin@polyu.edu.hk
}}

\affil[3]{\small{School of Mathematical Sciences, MOE-LSC, CMA-Shanghai, and Shanghai Center for Applied Mathematics,  Shanghai Jiao Tong University, 
Shanghai, China.

Email: sgzhou@sjtu.edu.cn
}
}


\vspace{-6 mm}

\date{\today}

\maketitle

\begin{abstract}


We study finite-difference approximations of both Poisson and 
Poisson--Boltzmann (PB) electrostatic energy functionals for periodic structures 
constrained by Gauss' law  and a class of local algorithms
for minimizing the finite-difference discretization of such functionals. 
The variable of Poisson energy is the vector field of electric displacement
and that for the PB energy consists of 
an electric displacement and ionic concentrations.
The displacement is discretized at midpoints of edges of grid boxes
while the concentrations are discretize at grid points. 
The local algorithm is an iteration over all the grid boxes that locally 
minimizes the energy on each grid box, keeping Gauss' law satisfied.  
We prove that the energy functionals admit unique minimizers that
are solutions to the corresponding Poisson's and charge-conserved PB equation, respectively.
Local equilibrium conditions are identified to characterize
the finite-difference minimizers of the discretized energy functionals. 
These conditions are the curl free for the Poisson case and the discrete
Boltzmann distributions for the PB case, respectively. 
Next, we obtain the uniform bound with respect to the grid size $h$
and $O(h^2)$-error estimates in maximum norm for the finite-difference minimizers. 
The local algorithms are detailed, and a new local algorithm with 
shift is proposed to treat the general case of a variable coefficient
for the Poisson energy. 
We prove the convergence of all these local algorithms,  
using the characterization of the finite-difference minimizers. 
Finally, we present numerical tests to demonstrate the results of our analysis. 

\medskip


\medskip

\noindent
{\bf AMS Subject Class:} 49M20, 65N06, 65Z05.

\end{abstract}

\tableofcontents





\section{Introduction}
\label{s:Introduction}
%

We consider the following variational problems of minimizing 
the non-dimensionalized Poisson \cite{Landau_Book93,Jackson99} and Poisson--Boltzmann (PB) 
\cite{Chapman1913,DebyeHuckel1923,Gouy1910,AndelmanHandbook95,
Fixman_JCP1979,Zhou_JCP94,CDLM_JPCB08,Li_SIMA09}
electrostatic energy functionals constrained by Gauss' law for periodic structures: 
\begin{align*}
& 
\left\{ 
\begin{aligned}
& \mbox{Miminize} & &  F[D]:= \int_{\Omega} \frac{1}{2 \ve} |D|^2 dx
& & \mbox{(Poisson energy)}, 
\\
& \mbox{subject to} & & \nabla \cdot D = \rho \qquad \mbox{in } \Omega
& & \mbox{(Gauss' law)}; 
\end{aligned}
\right.
\\
&\left\{
\begin{aligned}
& \mbox{Minimize} & & \hat{F}[c, D] :=  \int_{\Omega} \left( \frac{1}{2 \ve} |D|^2 
+ \sum_{s=1}^M c_s \log c_s \right) dx
& & \mbox{(PB energy)}, 
\\
& \mbox{subject to} & & 
\nabla \cdot D = \rho + \sum_{s=1}^M q_s c_s \qquad \mbox{in } \Omega 
& & \mbox{(Gauss' law)}, 
\\
& \ & &  \int_\Omega c_s \, dx = N_s,  \quad s = 1, \dots, M
   &  & \mbox{(Conservation of mass)}. 
\end{aligned}
\right.
\end{align*}
Here, $\Omega \subset \R^d$ $(d = 2 \mbox{ or } 3)$ is a cube, 
$\ve > 0$ and $\rho$ are given $\overline{\Omega}$-periodic functions
representing the dielectric coefficient and a fixed charge density, respectively, 
and $D$ is an $\overline{\Omega}$-periodic vector filed of electric displacement. 
For the PB case, $c = (c_1, \dots, c_M)$ and each $c_s \ge 0$ is 
the local concentration of ions of $s$th species, a total of 
$M$ species is assumed.  For each $s$, $q_s$ 
is the charge for an ion in species $s$ and $N_s$ is the total amount
of concentration of such ions. All $M$, $q_s$, and $N_s$ are given constants. 
Here and below $\log$ denotes the natural logarithm and 
$u \log u = 0$ if $u = 0.$

To discretize the energy functionals and Gauss' law, let us consider
the three-dimensional case to be specific and cover 
$\overline{\Omega}$ with a finite-difference grid of size $h$ with 
the grid point $(i, j, k)$ corresponding to the spatial point 
$(x_i, y_j, z_k)$.  We approximate the displacement at half-grid points
by $D_{i+1/2,j+1/2,k+1/2} = (u_{i+1/2, j, k},  v_{i, j+1/2, k}, w_{i, j,k+1/2})$ 
and concentrations $c = (c_s, \dots, c_s)$ 
at grid points by $c_{s, i,j,k} \ge 0$ for all $s, i, j, k.$
The PB energy and the corresponding Gauss' law at all the grid points are 
then discretized as
\begin{align*}
& \hat{F}_h [c,D] :=  \frac{h^3}{2} 
\sum_{i, j, k} \left( \frac{u_{i+1/2,j,k}^2 }{ \ve_{i+1/2, j, k}}
+  \frac{v_{i,j+1/2,k}^2 }{ \ve_{i, j+1/2, k}} 
+  \frac{w_{i,j,k+1/2}^2 }{ \ve_{i, j, k+1/2}} \right)
+ h^3 \sum_{s}\sum_{i,j, k} c_{s, i, j, k} \log c_{s, i, j, k},
\\
&
u_{i+1/2, j, k}-u_{i-1/2, j, k} + v_{i, j+1/2, k} - v_{i, j-1/2, k } 
+ w_{i, j,k+1/2}-w_{i, j, k-1/2}
= h \left( \rho_{i,j,k} + \sum_{s} q_s c_{s,i,j, k}\right), 
\end{align*}
respectively, where $\ve_{i+1/2, j, k} = ( \ve(x_i, y_j, z_k) + \ve(x_{i+1}, y_j, z_k)/2$
and $\ve_{i,j+1/2, k}$ and $\ve_{i,j,k+1/2}$ are similarly defined, 
and $\rho_{i,j,k}$ is an approximation of $\rho(x_i,y_j,z_k)$. 
The  mass conservation can be discretized similarly. 
The finite-difference discretization $F_h[D]$ of the  Poisson energy and that of
the corresponding Gauss' law are similar. 
Note that the discretization of displacement is 
a classical scheme for Maxwell's equation for isotropic media \cite{Yee_1966}
(cf.\ also \cite{MonkSuli94,LiShields16}).
If the displacement is given by $-\ve \nabla \phi$ 
with an electrostatic potential $\phi$, 
then the resulting scheme for $\phi$ is a commonly used, 
second-order central differencing scheme; cf.\ e.g., 
\cite{Pruitt_NumerMath2014,Pruitt_M2AN2015}. 

We are interested in a class of local algorithms for electrostatics 
\cite{MaggsRossetto_PRL02,M:JCP:2002,DuenwegPRE09,RM:PRL:2004,eint}
that are based on the above formulation of the constrained
energy minimization and the corresponding finite-difference discretization. 
The key idea of such algorithms
is to keep Gauss' law satisfied at each grid point while 
locally updating the discretized displacement or ionic concentrations 
one grid at a time, cycling through all the grid points iteratively. 
For instance, given a finite-difference 
displacement $D = (u, v, w)$ and a grid box
$(i, j, k)+[0, 1]^3$, one updates locally the components of 
$D$ on the edges of the three faces of the grid box  
sharing the vertex $(i, j,k)$ to decrease the Poisson energy $F_h[D]$. 
Let us fix such a face to be the square
with vertices $(i, j, k)$, $(i+1, j, k)$, $(i, j+1, k)$, and $(i+1, j+1, k).$  
To satisfy Gauss' law at these vertices, we update
\begin{align*}
&{u}_{i+1/2,j, k} \leftarrow u_{i+1/2,j, k}+\eta, \quad \mbox{and} \quad 
{u}_{i+1/2,j+1, k} \leftarrow u_{i+1/2,j+1,k}-\eta,
\\
&{v}_{i,j+1/2, k} \leftarrow v_{i,j+1/2, k}-\eta, \quad \mbox{and} \quad 
{v}_{i+1,j+1/2, k} \leftarrow v_{i+1,j+1/2, k}+\eta,
\end{align*}
with a single parameter $\eta$ that can be readily computed to minimize 
the perturbed Poisson energy; cf.\ section~\ref{ss:LcoalAlg4Poisson} 
for more details. For the PB energy, the concentration $c_s$ and the displacement $D$
are locally updated at neighboring grids, e.g., $(i, j, k)$ and $(i+1, j, k)$, 
and at the edge connecting them, respectively, by 
\begin{align*}
c_{s, i, j, k} \leftarrow c_{s, i, j, k} - \zeta, 
\quad 
c_{s, i+1, j, k} \leftarrow c_{s, i+1, j, k} + \zeta, 
\quad \mbox{and} \quad 
u_{i+1/2, j, k} \leftarrow u_{i+1/2, j, k} - h q_s \zeta, 
\end{align*}
with a single parameter $\zeta$ that can be computed to minimize the perturbed PB energy. 
The special forms of these perturbations are determined by the mass conservation and 
Gauss' law; cf.\ section~\ref{ss:LocalAlg4PB} for more details. 

Let us now briefly describe and discuss our main results. 

(1) {\it Existence, uniqueness, characterization, and bounds of minimizers.}
The constrained Poisson energy $F$ is uniquely minimized by 
$D_{\rm min} = -\ve \nabla \phi_{\rm min}$, where
$\phi_{\rm min}$ is the unique solution to Poisson's equation 
$\nabla \cdot \ve \nabla \phi = -\rho$; 
cf.\ Theorem~\ref{t:PoissonEnergy}. 

Similarly, the unique minimizer $(\hat{c}_{\rm min}, \hat{D}_{\rm min})$
of the constrained PB energy $\hat{F}$ is given by 
$\hat{D}_{\rm min} = -\ve \nabla \hat{\phi}_{\rm min}$
and the Boltzmann distributions $\hat{c}_{{\rm min}, s} \propto 
e^{-q_s \hat{\phi}_{\rm min}}$ for all $s$,  
where the electrostatic potential $\hat{\phi}_{\rm min}$ 
is the unique solution to the charge-conserved PB equation (CCPBE)
\begin{equation*}
\nabla \cdot \ve \nabla \phi +\sum_{s=1}^M N_s q_s 
\left( \int_\Omega e^{-q_s \phi} dx \right)^{-1} e^{-q_s \phi}  = -\rho.  
\end{equation*}
Moreover, a variational analysis of the CCPBE using a comparison argument 
\cite{LiChengZhang_SIAP11} shows that 
$\hat{\phi}_{\rm min}$ is bounded function. This leads to the uniform positive bounds 
\[
 0 < \theta_1 \le \hat{c}_{{\rm min}, s}(x) \le \theta_2  
\qquad \mbox{for all } x, s, 
\]
where $\theta_1$ and $\theta_2$ are constants; 
cf.\ Theorem~\ref{t:CCPBenergy} and Theorem~\ref{t:PBEnergy}. 

(2) {\it Characterization and uniform bounds of finite-difference minimizers.}
The unique minimizer $D_{\rm min}^h$ of the discretized
constrained Poisson energy $F_h$
is given by $D_{\rm min}^h = - \ve \nabla \phi_{\rm min}^h$, where 
$\phi_{\rm min}^h$ is the unique solution to the discretized
Poisson's equation. Moreover, $D_{\rm min}^h$ is characterized by 
the local equilibrium condition and the global constraint 
\[
\nabla_h \times \left( \frac{D_{\rm min}^h}{\ve} 
\right)_{i+1/2,j+1/2,k+1/2}  = 0
\quad \forall i, j, k \quad \mbox{and} \quad 
\sum_{i, j, k} \left( \frac{D_{\rm min}^h}{\ve}\right)_{i+1/2,j+1/2, k+1/2} = 0,  
\]
respectively, where $\nabla_h \times $ is the discrete curl operator; 
cf.\ Theorem~\ref{t:DiscreteEnergy}. These are analogous to the vanishing
of curl and integral of gradient of a smooth and periodic function.

The unique finite-difference solution $\hat{\phi}^h_{\rm min}$ to 
the discretized CCPBE is uniformly bounded in the maximum norm 
with respect to the grid size $h.$ This is proved using a similar comparison argument. 
The unique minimizer $(\hat{c}_{\rm min}^h, \hat{D}_{\rm min}^h)$
of the discretized constrained PB energy 
$\hat{F}_h$ is then given by the discrete Boltzmann distributions and
$\hat{D}_{\rm min}^h = - \ve \nabla_h \hat{\phi}_{\rm min}^h$, where 
$\nabla_h$ is the discrete gradient. 
These, together with the uniform positive bounds 
\[
 0 < C_1 \le \hat{c}^h_{{\rm min}, s} \le C_2 \quad 
\mbox{on all the grid points,}
\]
with $C_1 $ and $C_2 $ constants independent of $h$,  
characterize the discrete minimizer for the PB energy;  
cf.\ Theorem~\ref{t:hCCPB} and Theorem~\ref{t:DiscretePBEnergy}.


(3) {\it Error estimates.} 
We obtain the $L^\infty$-error estimate 
for the finite-difference approximation $D^h_{\rm min}$ of the Poisson energy minimizer  
$D_{\rm min}$
\[
\| \sP_h D_{\rm min} - D^h_{\rm min} \|_\infty \le Ch^2, 
\]
where $(\sP_h D)_{i, j, k} = (u(x_{i+1/2, j, k}), v(y_{i,j+1/2,k}), w(z_{i,j,k+1/2})) $
for any continuous displacement $D=(u, v, w)$ and all $i, j, k$, and 
$C$ denotes a generic constant independent of $h$.
This follows from the $L^\infty$ and $W^{1,\infty}$
stability of the inverse of the finite-difference operator 
for the Poisson equation \cite{Pruitt_M2AN2015,Pruitt_NumerMath2014,Beale_SINUM2009}.
By a simple averaging from $D^h_{\rm min}$, we obtain an
approximation $ m_h [D^h_{\rm min}]$, a vector-valued grid function,
and the superconvergence estimate
\[
\left\| \frac{m_h [-D^h_{\rm min}]}{\ve} -  \nabla \phi_{\rm min}\right\|_\infty \le Ch^2,
\] 
improving the existing $L^2$-superconvergence estimate \cite{LiShields16}; 
cf.\ Theorem~\ref{t:L2Poisson} and Corollary~\ref{c:D4GradphiP}.

For the PB case, we first prove the $O(h^2)$ $L^2$-error estimates
for both the displacement and concentrations, relying on 
the uniform bounds on the discrete concentrations. 
Such estimates are then used to prove the 
$L^\infty$-error estimate
\[
\| \hat{c}_{\rm min} - \hat{c}^h_{\rm min} \|_\infty + 
\| \sP_h \hat{D}_{\rm min} - \hat{D}^h_{\rm min} \|_\infty \le Ch^2; 
\]
cf.\ Theorem~\ref{t:pbL2error}.

(4) {\it A new local algorithm with shift for variable dielectric coefficient.}
Note that each local update in the local algorithm for relaxing the 
discrete Poisson energy does not change 
$\sum_{i, j,k} D_{i+1/2,j+1/2,k+1/2}$ 
but will change $\sum_{i, j,k} (D/\ve)_{i+1/2,j+1/2,k+1/2}$ if $\ve$ is not 
a constant. Therefore, the local algorithm for Poisson may not 
converge to the correct limit in this case, as the minimizer 
$D^h_{\rm min}$ should satisfy the global constraint 
$\sum_{i, j,k} (D^h_{\rm min}/\ve)_{i+1/2,j+1/2,k+1/2} = 0.$
To resolve this issue, we propose a new local algorithm with shift:
after a few cycles of local update of the displacement $D$, we shift
it by adding a constant vector $(\hat{a}, \hat{b}, \hat{c})$ to 
$D$ so that the shifted new displacement will satisfy the required global constraint; 
cf.\ section~\ref{ss:LcoalAlg4Poisson}. 

(5) {\it Convergence of all the local algorithms.}
The proof relies crucially on the characterization of the finite-difference 
minimizers $D_{\rm min}^h$ and $(\hat{c}_{\rm min}^h, \hat{D}_{\rm min}^h)$
of the discrete Poisson and PB energy functionals, respectively.
If $\delta^{(k)}$ is the energy difference after the $k$th local update, 
then $0 \le \delta^{(k)} \to 0$ as $k \to \infty$. Moreover, the amount of 
local change of the displacement or concentration in a local update is controlled
by the energy difference. Therefore, the sequence of such local changes 
converge to a local equilibrium that satisfies the conditions characterizing
the finite-difference minimizer; 
cf.\ Theorem~\ref{t:ConvP}, Theorem~\ref{t:ConvLocalGlobal}, and 
Theorem~\ref{t:ConvPB}. 


(6) {\it Numerical tests.}
We present numerical tests to demonstrate the results of our analysis on 
the error estimates and the convergence of local algorithms; 
cf.\ section~\ref{s:NumericalTests}.


We remark that the PB equation 
\cite{Chapman1913,DebyeHuckel1923,Gouy1910, AndelmanHandbook95, Fixman_JCP1979,Zhou_JCP94,
CDLM_JPCB08,Li_SIMA09}, 
with different kinds of boundary conditions, 
is a widely used continuum model of electrostatics
for ionic solutions with many applications, particularly in molecular biology
\cite{HonigScience95,SharpHonig90,DavisMcCammon_ChemRev90,GrochowskiTrylska08, 
Hille_Book2001,Baker_QRBiophys2012,FogolariBrigoMolinari02,BSJHM_PNAS01,VISMPB_JCTC14}. 
The periodic boundary conditions for Poisson's and PB equations are commonly used 
for simulations of electrostatics not only for periodic charged structures
such as ionic crystals but also in molecular dynamics simulations of charged molecules
\cite{RauchScott_SIAP2021,RauchScott_SIMA2021,deLeeuW_ElecPBCI80,deLeeuW_ElecPBCII80,
FrenkelSmit_1996,DardenPME_JCP93,Ewald_1921}.

The local algorithms were initially proposed for Monte Carlo and molecular dynamics
simulations of 
electrostatics and electromagnetics 
\cite{MaggsRossetto_PRL02,M:JCP:2002, RM:PRL:2004,DuenwegPRE09, eint}.  
Such algorithms scale linearly with system sizes and are simple to implement. 
The Gauss' law constrained energy minimization model for electrostatics
that is the basis for the local algorithms  
has been extended to model ionic size effects with nonuniform ionic sizes
\cite{ZhouWangLi_PRE11,LiLiuXuZhou_Nonlinearity2013,Li_Nonlinearity09}.
Recently, the local algorithms have been incorporated into numerical methods for 
Poisson--Nernst--Planck equations~\cite{Qiao2022NumANP,Qiao2022ANP,Qiao2024SISC}.  
The linear complexity and locality of the local algorithms make it appealing
to combine them with the recently developed binary level-set method
for large-scale molecular simulations using the variational implicit solvent model
\cite{ZhangRicci_JCTC202021,LiuZhang_2023,ZhangCheng_SISC23,VISMPB_JCTC14}.

The rest of this paper is organized as follows: 
In section~\ref{s:EnergyMin}, we first set up the variational problems of minimizing
the Poisson and PB electrostatic energy functionals constrained by 
Gauss' law. We then obtain the existence, 
uniqueness, and bounds in maximum norm of the energy minimizers
through the corresponding electrostatic potentials that are the 
periodic solutions to Poisson's equation and the CCPBE,   respectively. 
In section~\ref{s:Discretization}, we define finite-difference 
approximations of the Poisson and PB energy functionals, 
identify sufficient and necessary conditions for 
the finite-difference energy minimizers, and 
obtain their uniform bounds in maximum norm independent of the grid size $h$. 
In section~\ref{s:ErrorEstimates}, we prove the error estimates for the
finite-difference energy minimizers. 
In section~\ref{s:LocalAlg}, we describe the local algorithms for minimizing
the finite-difference functionals, and a new local algorithm with shift
for minimizing the Poisson energy with a variable dielectric coefficient. 
We also prove the convergence of all these algorithms.  
In section~\ref{s:NumericalTests}, we report numerical tests to 
demonstrate the results of our analysis. 
Finally, in Appendix, we prove some properties of the finite-difference operators.

\section{Energy Minimization}
\label{s:EnergyMin}

Let $L > 0$ and $\Omega = (0, L)^d$ with $d = 2$ or $3$. We denote by 
$C_{\rm per}(\overline{\Omega})$
and $C^k_{\rm per}(\overline{\Omega})$ ($k \in \N $) the spaces of $\overline{\Omega}$-periodic continuous functions and $\overline{\Omega}$-periodic $C^k$-functions on $\R^d$, respectively. 
Let $1 \le p \le \infty$ and $k \in \N.$  We denote by 
$L^p_{\rm per}(\Omega)$ and $W^{k, p}_{\rm per}(\Omega)$ 
the spaces of all $\overline{\Omega}$-periodic functions 
on $\R^d$ such that their restrictions onto $\Omega$ are
in the Lebesgue space $L^p(\Omega)$ and the Sobolev
space $W^{k, p}(\Omega)$, respectively \cite{GilbargTrudinger98,Adams75,EvansBook2010}.
Note that any $\phi \in L^p(\Omega)$ can be extended 
$\overline{\Omega}$-periodically to $\R^d$ after the values of 
$\phi$ on a set of zero Lebesgue measure are modified if necessary. 
As usual, two functions in $L^p_{\rm per}(\Omega)$ or 
$W^{k, p}_{\rm per}(\Omega) $
are the same if and only if they equal to each other almost everywhere with respect
to the Lebesgue measure. 
We define
\begin{align*}
& \mathring{L}^{p}_{\rm per}(\Omega) = 
\left\{ \phi \in L^p_{\rm per}(\Omega): \sA_\Omega (\phi) = 0 \right\}, 
\\
& \mathring{W}^{k, p}_{\rm per}(\Omega) = 
\left\{ \phi \in W_{\rm per}^{k, p}(\Omega): \sA_\Omega (\phi) = 0 \right\}, 
\end{align*}
where for a Lebesgue measurable function $u$ defined on
a Lebesgue measurable set $A \subset \R^d$ of finite measure $|A| > 0$, 
\begin{equation}
\label{AveAu}
\sA_A(u):=
\dashint_A u \, dx := \frac{1}{| A | } \int_A u \, dx.
\end{equation}
We denote $H^k_{\rm per}(\Omega) = W^{k, 2}_{\rm per}(\Omega)$ and 
$ \mathring{H}^{k}_{\rm per}(\Omega) = \mathring{W}^{k, 2}_{\rm per}(\Omega)$. 
By Poincar{\' e}'s inequality, $\phi \mapsto \|\nabla \phi\|_{L^2(\Omega)}$ is
a norm of $\mathring{H}^1_{\rm per}(\Omega)$, equivalent to the $H^1$-norm. 
We further define
\begin{align*}
&H(\mbox{div}, \Omega) = \{ D \in L^2(\Omega, \R^d): 
\nabla \cdot D \in L^2(\Omega)\},
\\
& H_{\rm per}(\mbox{div}, \Omega) = 
\mbox{the $H(\mbox{div}, \Omega)$--closure of }
 C_{\rm per}^1(\overline{\Omega}, \R^d)\mbox{--functions restricted 
to } \Omega. 
\end{align*}
The divergence $\nabla \cdot D$ 
is understood in the weak sense.
The space $H(\mbox{div}, \Omega)$ is a Hilbert space 
with the corresponding norm  $\|D\|_{H({\rm div}, \Omega)}
= \|D \|_{L^2(\Omega)}+ \| \nabla \cdot D \|_{L^2(\Omega)}$ 
\cite{Temam84}.

\subsection{The Poisson energy}
\label{ss:minPoissonEnergyCont}

We consider the Poisson electrostatic energy with a given charge density
$\rho \in L_{\rm per}^2(\Omega)$. Denote
\begin{align}
\label{Srho}
&S_\rho = \{ D \in H_{\rm per}(\mbox{div}, \Omega): 
\nabla \cdot D = \rho \mbox{ in } \Omega \},        
\\
\label{S0}
&S_0 = \{ D \in H_{\rm per}(\mbox{div}, \Omega): 
\nabla \cdot D = 0 \mbox{ in } \Omega \}. 
\end{align}
By the periodic boundary condition and the divergence theorem, 
$S_{\rho} \ne \emptyset$ if and only if 
$\sA_\Omega(\rho) = 0.$ Clearly $S_0 \ne \emptyset.$
Let $\ve\in L_{\rm per}^\infty(\Omega)$. 
Assume there exist
$\ve_{\rm min}, \, \ve_{\rm max} \in \R$ such that
\begin{equation}
\label{epsilon}
0 < \ve_{\rm min} \le \ve(x) \le \ve_{\rm max}  \qquad 
\forall x \in \R^d.  
\end{equation}
We define 
\begin{align}
\label{defineI}
&I[\phi] = \int_{\Omega} \left( \frac{\ve}{2} |\nabla \phi|^2 - \rho \phi\right) dx
\qquad \forall \phi \in H^1_{\rm per}(\Omega), 
\\
    \label{PoissonEnergy}
&F[D] = \int_\Omega \frac{1}{2 \ve} |D|^2 dx \qquad \forall D \in S_\rho.
\end{align}

\begin{theorem}
\label{t:PoissonEnergy}
Let $\ve\in L^\infty_{\rm per}(\Omega)$ satisfy \reff{epsilon} 
and $\rho \in \mathring{L}^2_{\rm per}(\Omega)$.
\begin{compactenum}
\item[{\rm (1)}] 
There exists a unique $\phi_{\rm min}\in \mathring{H}^1_{\rm per}(\Omega)$ such that 
${I}[{\phi}_{\rm min}] = 
\min_{\phi \in \mathring{H}^1_{\rm per}(\Omega)} I[\phi].$
Moreover, $\phi_{\rm min}$ is the
unique weak solution  in $\mathring{H}_{\rm per}(\Omega)$ to Poisson's equation 
$
\nabla \cdot \ve \nabla \phi_{\rm min} = - \rho, 
$
defined by 
\begin{equation}
\label{weakphimin}
\int_\Omega \ve \nabla \phi_{\rm min} \cdot  \nabla \xi \, dx 
= \int_\Omega \rho \, \xi \, dx \qquad \forall \xi \in \mathring{H}_{\rm per}^1(\Omega).
\end{equation}

\item[{\rm (2)}] 
There exists a unique $D_{\rm min} \in S_\rho$ such that 
$ F [D_{\rm min}] = \min_{D \in S_\rho} F[D].$ Moreover, the minimizer
$D_{\rm min}$ is characterized by $D_{\rm min} \in S_\rho$ and 
\begin{equation}
\label{weakDmin}
\int_\Omega \frac{1}{\ve} D_{\rm min} \cdot 
\tilde{D}\, dx = 0 \qquad \forall \tilde D \in S_0.
\end{equation}

\item[{\rm (3)}] 
We have $D_{\rm min} = - \ve \nabla \phi_{\rm min}.$ 
\end{compactenum}
\end{theorem}

\begin{proof}
(1) These are standard; cf.\ e.g., \cite{EvansBook2010,GilbargTrudinger98}. 

(2) 
The existence and uniqueness of a minimizer $D_{\rm min}$ of 
$F: S_\rho \to \R$ and \reff{weakDmin} are standard. 
Suppose $D \in S_\rho$ satisfies \reff{weakDmin} with $D$ replacing $D_{\rm min}.$ 
 Since $ D - D_{\rm min} \in S_0$, 
\[
\int_\Omega \frac{1}{\ve} D \cdot (D - D_{\rm min}) \, dx = 0.
\]
Thus, by the Cauchy--Schwarz inequality, 
\[
\int_\Omega \frac{1}{2 \ve} |D|^2 dx 
= \int_\Omega \frac{1}{2 \ve} D \cdot D_{\rm min} \, dx 
\le \left( \int_\Omega \frac{1}{2 \ve} |D|^2 dx \right)^{1/2}
\left( \int_\Omega \frac{1}{2 \ve} |D_{\rm min}|^2 dx \right)^{1/2}. 
\]
This leads to $F[D] \le F[D_{\rm min}]$ and hence $D$ is the minimizer. 

(3) By Part (1), $D:  = - \ve \nabla \phi_{\rm min} \in S_\rho.$ Thus, 
\reff{weakDmin} follows from integration by parts. 
Hence $D = D_{\rm min} = - \ve \nabla \phi_{\rm min}.$ 
\end{proof}

\subsection{The charge-conserved Poisson--Boltzmann equation}
\label{ss:CCPBE}


Let $M \ge 1$ be an integer,  $q_1, \dots, q_M$ nonzero real numbers, 
$N_1, \dots, N_M$ positive numbers, $\ve\in L^\infty_{\rm per}(\Omega)$ satisfy
\reff{epsilon}, and $\rho\in L_{\rm per}^2(\Omega).$ 
We shall assume the following: 
\begin{equation}
    \label{neutrality}
    \mbox{Charge neutrality:} \qquad \qquad  \qquad 
    \sum_{s=1}^M q_s N_s + \int_\Omega \rho \, dx = 0.
    \qquad \qquad \qquad \qquad \qquad 
\end{equation}
Let us define
$\hat{I}: H_{\rm per}^1(\Omega)\to \R \cup \{ + \infty \}$ by  \cite{Lee_JMathPhys2014}
\begin{equation}
    \label{Iphi}
\hat{I}[\phi] = \int_\Omega \left( \frac{\ve}{2} | \nabla \phi |^2 - \rho \phi \right) dx
+ \sum_{s=1}^M N_s \log \left( \sA_{\Omega} (e^{-q_s \phi})
\right)
\qquad \forall \phi \in H_{\rm per}^1(\Omega). 
\end{equation} 

 \begin{lemma}
     \label{l:Iphi}
        Let $\ve \in L_{\rm per}^\infty(\Omega) $ 
satisfy \reff{epsilon} and $\rho \in L_{\rm per}^2(\Omega)$ satisfy \reff{neutrality}. Then the following hold true: 
        
     {\rm (1)} 
     $\hat{I}[\phi] = \hat{I}[\phi + a]$ for any $\phi \in H_{\rm per}^1(\Omega)$ and 
any constant  $a \in \R;$

     {\rm (2)} The functional $\hat{I}: \mathring{H}^1_{\rm per}(\Omega) \to \R \cup \{ + \infty \}$
     is strictly convex;
     
     {\rm (3)} There exist $K_1 > 0 $ and $K_2 \in \R$ such that 
     $ \hat{I}[\phi] \ge K_1 \| \phi \|_{H^1(\Omega)}^2 + K_2$ for all 
     $\phi \in \mathring{H}^1_{\rm per}(\Omega).$
 \end{lemma}

 \begin{proof} 
 (1) This follows from the charge neutrality \reff{neutrality}. 

(2) The integral part of the functional $\hat{I}$ is strictly convex as 
$\phi \mapsto \| \nabla \phi \|_{L^2(\Omega)}$ is a norm on $\mathring{H}^1_{\rm per}(\Omega).$
The convexity of 
the non-integral part of the functional $\hat{I}$ follows from an application of 
Holder's inequality and the fact that $u \mapsto \log u$ is an increasing function on 
$(0, \infty).$

 (3) This follows from Jensen's inequality applied to $ u \mapsto - \log u$ and 
 Poincar{\' e}'s inequality applied to $\phi \in \mathring{H}^1_{\rm per}(\Omega).$
 \end{proof} 

By formal calculations, the Euler--Lagrange equation for the functional $\hat{I}$
defined in 
 \reff{Iphi} is the charge-conserved Poisson--Boltzmann equation (CCPBE)
 \begin{equation}
 \label{CCPBE}
       \nabla \cdot \ve \nabla \phi +\sum_{s=1}^M 
     N_s q_s \left( \int_\Omega e^{-q_s \phi} dx \right)^{-1} e^{-q_s \phi}  = -\rho.
 \end{equation}

 \begin{definition}
     A function $ \phi \in \mathring{H}^1_{\rm per}(\Omega)$ is a weak solution to 
     the CCPBE \reff{CCPBE} if 
     $e^{-q_s \phi} \in L^2(\Omega)$ for each $s \in \{ 1, \dots, M\}$ and 
     \begin{equation}
         \label{weakCCPBE}
         \int_\Omega \ve \nabla \phi \cdot \nabla \xi \, dx 
         - \sum_{s=1}^M N_s q_s \left( \int_\Omega e^{-q_s \phi }dx \right)^{-1}
         \int_\Omega e^{-q_s \phi} \xi \, dx = \int_\Omega \rho \, \xi \, dx
         \qquad \forall \xi \in \mathring{H}^1_{\rm per}(\Omega). 
     \end{equation}
 \end{definition}

 \begin{theorem}
     \label{t:CCPBenergy}
     Let $\ve\in L_{\rm per}^\infty(\Omega)$ satisfy \reff{epsilon} and 
 $\rho \in L_{\rm per}^2(\Omega)$ satisfy \reff{neutrality}. 
There exists a unique $\hat{\phi}_{\rm min} \in \mathring{H}^1_{\rm per}(\Omega)$ such   that 
     $
     \hat{I}[\hat{\phi}_{\rm min}] = \min_{\phi \in \mathring{H}^1_{\rm per}(\Omega)} \hat{I}[\phi],
     $
     which is finite.
      If in addition $\ve \in C_{\rm per}^1(\overline{\Omega})$,  
      then $\hat{\phi}_{\rm min} \in L_{\rm per}^\infty(\Omega)
      \cap H_{\rm per}^2(\Omega)$, it is the unique weak solution
     to the CCPBE with the periodic boundary condition, and it satisfies
     \reff{CCPBE} a.e.\ in $\Omega$. 
 \end{theorem}


\begin{remark}
    These results are generally known for the case that $q_s > 0$ for some $s$ and 
    $q_s < 0$ for some other $s$
    \cite{Lee_JMathPhys2014}. Here we include the case that all $q_s > 0$ or all $q_s < 0$. 
    Moreover, we present a proof with a key difference.
    We obtain the $L^\infty(\Omega)$-bound
    of the minimizer by a comparison argument; cf.\ \cite{LiChengZhang_SIAP11}.
    The bound  allows us to apply the Lebesgue Dominated Convergence
    Theorem to show that the minimizer is a weak solution to the CCPBE.
    The comparison method used in obtaining the $L^\infty$ bound will also be used in 
    section~\ref{ss:FDCCPBE} to obtain
    a uniform bound for finite-difference approximations of the solution to CCPBE. 
\end{remark}



 \begin{proof}[Proof of Theorem~\ref{t:CCPBenergy}]
The existence  of a minimizer 
$\hat{\phi}_{\rm min} \in \mathring{H}^1_{\rm per}(\Omega)$ follows from Lemma~\ref{l:Iphi} and a standard argument by direct methods in the calculus of variations; 
cf.\ e.g., \cite{Lee_JMathPhys2014}. 
The uniqueness of a minimizer follows from the strict convexity of the functional 
$\hat{I}.$

We now assume in addition that $\ve \in C_{\rm per}^1(\overline{\Omega})$ and 
prove that $\hat{\phi}_{\rm min} \in L_{\rm per}^\infty(\Omega).$
Let $\phi_0\in \mathring{H}^1_{\rm per}(\Omega)$ be the unique 
weak solution to Poisson's equation 
$\nabla \cdot \ve \nabla \phi_0 = - \rho - (1/|\Omega|) \sum_{s=1}^M q_s N_s$
with the periodic boundary condition, defined by 
\[
\int_\Omega \ve \nabla \phi_0 \cdot \nabla \xi \, dx 
= \int_\Omega \rho \xi \, dx + 
\left( \sum_{s=1}^M q_s N_s \right) \dashint_\Omega \xi\, dx 
= \int_\Omega \rho \xi \, dx
\qquad \forall \xi \in \mathring{H}_{\rm per}^1(\Omega);
\]
cf.\ Theorem~\ref{t:PoissonEnergy}.
By the regularity theory, $\phi_0 \in L_{\rm per}^\infty(\Omega)$
 \cite{GilbargTrudinger98}.
We define 
\begin{equation}
    \label{hatIpsi}
J[\psi] = \int_\Omega \frac{\ve}{2} | \nabla \psi |^2 dx
+ \sum_{s=1}^M N_s \log \left( \sA_\Omega ( e^{-q_s (\phi_0+\psi)})\right)
\qquad \forall \psi \in {H}_{\rm per}^1(\Omega).
\end{equation}
Let $\psi \in H^1_{\rm per}(\Omega)$ and set $\bar{\psi} = \sA_\Omega (\psi)$; 
cf.\ \reff{AveAu}. 
We verify directly that
\begin{equation}
    \label{IpsiIphi}
    {J}[\psi] = {J}[\psi - \bar{\psi}] - \bar{\psi} \sum_{s=1}^M q_s N_s
    = \hat{I}[\phi] + \int_\Omega \frac{\ve}{2} | \nabla \phi_0 |^2 dx
    - \overline{\psi}\sum_{s=1}^M q_s N_s,
\end{equation}
where $\phi := \psi - \bar{\psi} + \phi_0 \in \mathring{H}^1_{\rm per}(\Omega). $ 
If $\psi = \phi - \phi_0 \in \mathring{H}_{\rm per}^1(\Omega)$ 
with $\phi \in \mathring{H}^1_{\rm per}(\Omega)$, then 
\[
{J}[\psi] = \hat{I}[\phi] +\int_\Omega \frac{\ve}{2} | \nabla \phi_0 |^2 dx.
\]
Thus, ${\psi}_{\rm min} := \hat{\phi}_{\rm min}-\phi_0\in \mathring{H}^1_{\rm per}(\Omega)$ 
is the unique minimizer
of ${J}: \mathring{H}^1_{\rm per}(\Omega) \to \R \cup \{ \infty \},$
and ${J}[\psi_{\rm min}]$ is finite since $\hat{I}[\phi_{\rm min}]$ is. 

We show that $\psi: = \psi_{\rm min} \in L_{\rm per}^\infty(\Omega)$
which implies $\hat{\phi}_{\rm min} \in L_{\rm per}^\infty(\Omega).$ 
We consider three cases. 

Case 1: there exist $s^\prime, s^{\prime\prime} \in \{ 1, \dots, M \}$
such that $q_{s^\prime} > 0$ and $q_{s^{\prime\prime}} < 0$. Let $\lambda > 0$ and define 
\begin{align}
\label{psihatlambda}
\hat{\psi}_\lambda = \left\{
\begin{aligned}
    &\psi  \quad & &\mbox{if } |\psi | \le \lambda, &
    \\
    &\lambda \quad & &\mbox{if } \psi > \lambda, &
    \\
    &-\lambda \quad & &\mbox{if } \psi < -\lambda, &
\end{aligned}
\right.
\qquad \mbox{and} \qquad 
\psi_\lambda = \hat{\psi}_\lambda - \sA_\Omega(\hat{\psi}_\lambda).
\end{align}
Clearly, 
$\hat{\psi}_\lambda\in H_{\rm per}^1(\Omega)$ and 
$\psi_\lambda \in \mathring{H}^1_{\rm per}(\Omega).$ 
Since $\psi  = \psi_{\rm min}$, we have 
${J} [{\psi_\lambda} ] \ge {J}[\psi]$. Therefore, it
follows from \reff{IpsiIphi}, \reff{psihatlambda}, and Jensen's inequality
applied to $u\mapsto -\log u$ that 
\begin{align}
\label{case1main}
0 & \ge - \int_{\{ | \psi | > \lambda \} }\frac{\ve}{2} | \nabla \psi |^2 dx
\nonumber \\
& = \int_{\Omega} \frac{\ve}{2} \left( |\nabla \hat{\psi}_\lambda|^2
- | \nabla \psi |^2 \right) dx 
\nonumber \\
& = {J}[\hat{\psi}_\lambda] - {J}[\psi] + 
\sum_{s=1}^M N_s \left[ \log \left( \dashint_\Omega e^{-q_s (\phi_0+\psi) }  dx \right)
-  \log \left( \dashint_\Omega e^{-q_s (\phi_0+\hat{\psi}_\lambda) } dx 
\right)\right]
\nonumber \\
& = {J}[{\psi}_\lambda] - {J}[\psi] - 
\sA_\Omega ( \hat{\psi}_\lambda) \sum_{s=1}^M q_s N_s 
\nonumber \\
&\qquad + 
\sum_{s=1}^M N_s \left[ \log \left( \int_\Omega e^{-q_s (\phi_0+\psi) }  dx \right)
-  \log \left( \int_\Omega e^{-q_s (\phi_0+\hat{\psi}_\lambda) } dx 
\right)\right]
\nonumber \\
& \ge \int_\Omega \left[ B (\phi_0 + \psi) -  B (\phi_0 + \hat{\psi}_\lambda ) \right] dx 
- \sA_\Omega ( \hat{\psi}_\lambda)\sum_{s=1}^M q_s N_s, 
\end{align}
where 
\begin{align}
\label{alphasBu}
B(u) = \sum_{s=1}^M \frac{N_s} {\alpha_s} e^{-q_s u}
\qquad \mbox{and} \qquad 
\alpha_s = \int_\Omega e^{ -q_s (\phi_0 + \psi)} dx.
\end{align}
Note that $ \alpha_s > 0$ for each $s$. Since 
 $J[\psi]$ is finite, we also have $\alpha_s < \infty$  for each $s$. 
Denoting $a = (1/|\Omega|) \sum_{s=1}^M q_s N_s$, 
we have by \reff{psihatlambda} and the fact that 
$\psi \in \mathring{H}^1_{\rm per}(\Omega)$ that 
\begin{align}
\label{a}
- \left(\sum_{s=1}^M q_s N_s \right) \sA_\Omega (\hat{\psi}_\lambda)
 &= a \int_\Omega ( \psi - \hat{\psi}_\lambda)\, dx 
 \nonumber \\
&= a \int_{\{\psi > \lambda \}} ( \psi - \lambda) \, dx + a \int_{\{\psi < - \lambda\}}
(\psi + \lambda) \, dx. 
\end{align}
We can verify directly that $B$ is convex.
Moreover, since $q_{s'} > 0$ and $q_{s''} < 0$, 
$B'(-\infty) = -\infty$ and $B'(+\infty) = +\infty$.
Thus, since $\phi_0 \in L_{\rm per}^\infty(\Omega)$, 
$B'(\phi_0 + \lambda) + a \ge 1 $ and 
$B'(\phi_0 - \lambda) + a \le -1 $ a.e.\ $\Omega$, 
if $\lambda > 0$ is large enough.
Consequently, it follows from 
\reff{case1main}, \reff{a}, and an application
of Jensen's inequality that
\begin{align*}
    0 &\ge  
    \int_{\{ \psi > \lambda\}} \left[ B (\phi_0 + \psi) - B (\phi_0 + \lambda ) \right] dx 
+ \int_{\{ \psi < -\lambda\} } \left[ B(\phi_0 + \psi) - B(\phi_0 - \lambda)\right] dx 
\\
&\qquad 
+ a \int_{\{\psi > \lambda \}} ( \psi - \lambda) \, dx + a \int_{\{\psi < - \lambda\}}
(\psi + \lambda) \, dx
\\
&\ge 
\int_{\{ \psi > \lambda \} }
     \left[ B' (\phi_0 + \lambda)+a\right] ( \psi - \lambda ) \, dx 
     + \int_{ \{ \psi < -\lambda \} } 
     \left[ B' (\phi_0 - \lambda) + a \right] ( \psi + \lambda ) \, dx
\\
&\ge \int_{\{ | \psi | > \lambda \}} \left| \, | \psi | - \lambda \, \right| \, dx.
\end{align*}
Hence, $| \{ | \psi | > \lambda \} | = 0$, i.e., 
 $|\psi| \le \lambda$ a.e.\ $\Omega$. Thus, $\psi \in L_{\rm per}^\infty(\Omega).$

Case 2: all $q_s < 0$ $( 1 \le s \le M).$ In this case, $B = B(u)$ defined 
in \reff{alphasBu} is convex and $B'(+\infty) = +\infty.$ 
For any $\lambda > 0$, we define now
$\hat{\psi}_\lambda = \psi$ if $\psi \le \lambda$ and 
$\hat{\psi}_\lambda = \lambda$ if $\psi > \lambda, $ and 
$\psi_\lambda = \hat{\psi}_\lambda - \sA_\Omega ( \hat{\psi}_\lambda).$
Clearly, $\hat{\psi}_\lambda\in H^1_{\rm per}(\Omega)$ and 
$\psi_\lambda \in \mathring{H}^1_{\rm per}(\Omega)$. 
Carrying out the same calculations as above with $ \{ \psi > \lambda\} $ replacing
 $\{ | \psi | > \lambda\}$, we get for $\lambda  > 0$ large enough that 
\[
0 \ge 
\int_{\{ \psi > \lambda \}} \left[ B'(\phi_0 + \lambda) + a\right]
(\psi - \lambda)\, dx \ge \int_{\{ \psi > \lambda \}}
(\psi - \lambda)\, dx \ge 0,
\]
where $a$ is the same as in \reff{a}. 
Thus, $\psi \le \lambda $ a.e.\ $\Omega.$ Since $\phi_0 \in L_{\rm per}^\infty(\Omega)$
and all $q_s < 0$,
$e^{-q_s (\phi_0 + \psi)}\in L_{\rm per}^\infty(\Omega)$ for each $s$ $(1 \le s \le M).$
Since $\psi $ is the minimizer of $J$ defined in \reff{hatIpsi} over
$\mathring{H}_{\rm per}^1(\Omega)$, we now have by direct calculations that 
\[
\int_\Omega \ve \nabla \psi \cdot \nabla \xi \, dx 
- \sum_{s=1}^M N_s q_s \left( \int_\Omega e^{-q_s (\phi_0 + \psi) }dx \right)^{-1}
         \int_\Omega e^{-q_s (\phi_0 + \psi)} \xi \, dx = 0
         \qquad \forall \xi \in \mathring{H}^1_{\rm per}(\Omega). 
\]
Since $q_s < 0$ and $\psi $ is bounded above, $e^{-q_s (\phi_0+\psi)}\in 
L^\infty_{\rm per}(\Omega)$ for each $s$. Thus, 
$\nabla \cdot \ve \nabla \psi \in L^\infty_{\rm per}(\Omega)$ weakly. 
Consequently, $\Delta \psi = (\nabla \ve \cdot \nabla \psi - \nabla \cdot \ve
\nabla \psi)/\ve \in L^2_{\rm per}(\Omega)$ weakly. 
Hence, $\psi \in H^2_{\rm per}(\Omega)$ and further
$\psi \in L_{\rm per}^\infty(\Omega).$

Case 3: all $q_s > 0$ $ (s = 1, \dots, M).$ This is similar to Case 2. 

Finally, since $\phi := \hat{\phi}_{\rm min} \in \mathring{H}_{\rm per}^1(\Omega) \cap 
L_{\rm per}^\infty(\Omega)$ 
is the unique minimizer of $\hat{I}: \mathring{H}_{\rm per}^1(\Omega)
\to \R \cup \{+ \infty \},$ we obtain by routine calculations 
the equation in \reff{weakCCPBE} with $\xi \in C^1_{\rm per}(\overline{\Omega}).$
By approximations, \reff{weakCCPBE} is true. Thus, 
$\phi$ is a weak solution to the CCPBE with the periodic boundary condition.
This also implies that $\nabla \cdot \ve \nabla \phi \in L^2(\Omega)$
in weak sense. The regularity theory then implies that 
$\phi \in H_{\rm per}^2(\Omega)$ and finally \reff{CCPBE} holds true a.e.\ in $\Omega.$


Assume $\phi_1, \phi_2 \in \mathring{H}_{\rm per}^1(\Omega)$ are two 
weak solutions of the CCPBE. Denote 
\begin{align*}
 \hat{B}_i(u) = \sum_{s=1}^M \frac{N_s}{a_{i, s} } e^{-q_s u} \quad \mbox{with} \quad 
 a_{i, s} = \int_\Omega e^{-q_s \phi_i} dx, \qquad i = 1, 2.  
\end{align*}
Each $\hat{B}_i: \R \to \R$ $(i = 1, 2)$ is a convex function. Thus, 
\begin{align*}
&\hat{B}_1'(\phi_1) (\phi_1 - \phi_2) \ge \hat{B}_1(\phi_1 ) - \hat{B}_1 (\phi_2)
\quad \mbox{a.e.} \ \Omega, 
\\
&\hat{B}_2'(\phi_2) (\phi_1 - \phi_2) \le \hat{B}_2(\phi_1 ) - \hat{B}_2 (\phi_2)
\quad \mbox{a.e.} \ \Omega. 
\end{align*}
Consequently, it follows from
\reff{weakCCPBE} with $\phi = \phi_i$ $(i = 1, 2)$ and $\xi = \phi_1 - \phi_2$ that
\begin{align*}
0 & = \int_\Omega \ve | \nabla (\phi_1 - \phi_2) |^2 dx 
+\int_\Omega [ \hat{B}_1'(\phi_1) (\phi_1 - \phi_2) - 
\hat{B}_2'(\phi_2) (\phi_1 - \phi_2)]\, dx
\\
&\ge 
\int_\Omega \left[  \left(\hat{B}_1(\phi_1) - \hat{B}_1(\phi_2) \right)
- \left( \hat{B}_2(\phi_1) - \hat{B}_2 (\phi_2)\right) \right]  dx
\\
& \ge 
\sum_{s=1}^M \frac{N_s}{a_{1,s} a_{2,s} } 
\left[ \int_\Omega
\left( e^{-q_s \phi_1} - e^{-q_s \phi_2} \right) dx \right]^2
\\
& \ge 0.
\end{align*}
Hence, $\phi_1 = \phi_2$ in $\mathring{H}_{\rm per}^1(\Omega)$ and the weak solution 
is unique. 
\end{proof}

\subsection{The Poisson--Boltzmann energy}
\label{ss:ContPBenergy}

Let $\rho \in L_{\rm per}^2(\Omega) $ satisfy \reff{neutrality}. 
We consider now ionic concentrations $c = (c_1, \dots, c_M)\in 
L_{\rm per}^2(\Omega, \R^M)$ and the electric displacements
$D \in H_{\rm per}(\mbox{div}, \Omega)$ that satisfy the following: 
\begin{align}
\label{positivity}
    &\mbox{Nonnegativity:}  & & c_s \ge 0 \qquad \mbox{ a.e. } \Omega,  \quad s = 1,\dots, M;  &
    \\
    \label{mass}
    &\mbox{Mass conservation:} & & \int_\Omega c_s \, dx = N_s,  \qquad s = 1, \dots, M; &
    \\
    \label{Gauss4Ions}
    &\mbox{Gauss' law:} & & \nabla \cdot D = \rho + \sum_{s=1}^M q_s c_s 
    \quad \mbox{ in } \Omega. &
\end{align}
We define 
\begin{align}
\label{Xrho}
&X_\rho  = \biggl\{ (c, D)
\in L_{\rm per}^2(\Omega, \R^M)\times H_{\rm per}(\mbox{div}, \Omega): 
\mbox{ \reff{positivity}--\reff{Gauss4Ions} hold true.}\biggr\}, 
\\
\label{tildeX0}
&\tilde{X}_0 = \biggl\{ (\tilde{c}, \tilde{D}) 
= (\tilde{c}_1, \dots, \tilde{c}_M; \tilde{D}) \in L_{\rm per}^\infty(\Omega, \R^M)\times 
H_{\rm per}(\mbox{div}, \Omega): 
\nonumber \\
& \qquad \qquad \qquad \qquad \qquad \qquad 
\int_\Omega \tilde{c}_s \, dx = 0 \ (s = 1, \dots, M) \mbox{ and }
\nabla \cdot \tilde{D} = \sum_{s=1}^M q_s \tilde{c}_s \biggr\}.
\end{align}

\begin{lemma}
\label{l:Xrho} Let $\rho \in L_{\rm per}^2(\Omega).$ 
Then, $X_\rho \ne \emptyset$ if and only if \reff{neutrality} holds true.
\end{lemma}

\begin{proof}
If $X_\rho \ne \emptyset$ and $(c, D) \in X_\rho$, then by integrating both sides of \reff{Gauss4Ions} and using
\reff{mass}, we obtain \reff{neutrality}. Conversely, 
    let $c_s = N_s/ |\Omega| $ in $\Omega$ for all $s = 1, \dots, M$ and
$\rho_{\rm ion} = \sum_{s=1}^M q_s c_s$. By \reff{neutrality}, 
$\sA_\Omega (\rho+\rho_{\rm ion}) = 0.$ Thus, 
$S_{\rho+\rho_{\rm ion}} \ne \emptyset;$ cf.\ \reff{Srho}. 
If $D \in S_{\rho+\rho_{\rm ion}}$ and 
$c = (c_1, \dots, c_M)$, then
$(c, D) \in X_\rho.$ Hence, $X_\rho \ne \emptyset.$
\end{proof}

Let $\ve \in L^\infty_{\rm per}(\Omega)$ satisfy \reff{epsilon}. 
We define $\hat{F}: X_\rho \to \R \cup \{ +\infty \}$ by  
\begin{equation}
    \label{PBEnergy}
\hat{F}[c, D] = \int_\Omega \left( \frac{|D|^2}{2 \ve} 
+ \sum_{s = 1}^M c_s \log c_s \right) dx.  
\end{equation}

\begin{theorem}
\label{t:PBEnergy}
Let $\ve \in C_{\rm per}^1(\overline{\Omega})$ satisfy \reff{epsilon}
and $\rho \in L_{\rm per}^2(\Omega)$ satisfy \reff{neutrality}. 

{\rm (1)}  Let $(\hat{c}_{\rm min}, \hat{D}_{\rm min}) = (\hat{c}_{{\rm min}, 1}, 
\cdots, \hat{c}_{{\rm min}, M}, \hat{D}_{\rm min}) $ be given by 
\begin{align}
\label{definecmins}
&\hat{c}_{{\rm min}, s} = N_s 
\left( \int_\Omega e^{- q_s \hat{\phi}_{\rm min}} dx \right)^{-1}
e^{- q_s \hat{\phi}_{\rm min}}\qquad \mbox{in } \R^d,  \ s = 1, \dots, M,
\\
\label{defineDmin}
&
\hat{D}_{\rm min} = - \ve \nabla \hat{\phi}_{\rm min} \qquad \mbox{in } \R^d, 
\end{align}
where $\hat{\phi}_{\rm min} \in \mathring{H}_{\rm per}^1(\Omega)$ is the unique weak solution 
to the CCPBE as given in Theorem~\ref{t:CCPBenergy}.
Then $(\hat{c}_{\rm min}, \hat{D}_{\rm min}) \in X_\rho$ is the unique minimizer of
$\hat{F}: X_\rho \to \R \cup \{+ \infty \}.$


{\rm (2)} Let $(c, D) = (c_1, \dots, c_M, D)
\in  X_\rho$. Then $(c, D) = (\hat{c}_{\rm min}, \hat{D}_{\rm min})$ 
if and only if the following conditions are satisfied: 
\begin{compactenum}
    \item[{\rm (i)}]
    {\em Positive bounds:} 
    There exist $\theta_1, \theta_2 \in \R$ such that 
    $
    0 < \theta_1 \le c_s(x) \le \theta_2 $ for a.e.\ $ x \in \Omega$ and 
    all $
    s = 1,\dots, M; 
    $
    \item[{\rm (ii)}] 
    {\em Global equilibrium:}
\begin{equation}
\label{tildeDc}
\int_\Omega \left( \frac{1}{\ve} D \cdot \tilde{D} \, 
+ \sum_{s=1}^M \tilde{c}_s \log c_s \right) dx = 0 \qquad 
\forall (\tilde{c}, \tilde{D}) \in \tilde{X}_0. 
\end{equation}
\end{compactenum}
\end{theorem}

\begin{proof}
(1) Since $\hat{\phi}_{\rm min} \in L_{\rm per}^\infty(\Omega) $ 
by Theorem~\ref{t:CCPBenergy}, we verify
that $(\hat{c}_{\rm min}, \hat{D}_{\rm min})\in X_\rho.$ Let $(c, D) \in X_\rho$ and denote 
$\tilde{c} = c-\hat{c}_{\rm min}$ and $\tilde{D} = D - \hat{D}_{\rm min}$. 
By the divergence theorem and the periodic boundary condition, 
the convexity of the function $u \mapsto u \log u $ $(u \ge 0)$, 
and \reff{definecmins} and \reff{defineDmin}, 
we obtain
\begin{align*}
    & \hat{F}[c, D]-\hat{F}[\hat{c}_{\rm min}, \hat{D}_{\rm min}] 
    \\
    &\quad = \int_\Omega  \frac{1}{2\ve} \left( | \hat{D}_{\rm min} + \tilde{D}|^2 
    - |\hat{D}_{\rm min}|^2 \right) dx 
    \\
    &\quad \quad + \sum_{s=1}^M \int_\Omega \left[ 
    (\hat{c}_{{\rm min}, s} + \tilde{c}_s) \log ( \hat{c}_{{\rm min}, s} + \tilde{c}_s ) 
    -  \hat{c}_{{\rm min}, s} \log \hat{c}_{{\rm min}, s} \right] 
    \\
    &\quad 
    \ge - \int_\Omega \nabla \hat{\phi}_{\rm min} \cdot ( D - \hat{D}_{\rm min})  \, dx 
    + \sum_{s=1}^M \int_\Omega \tilde{c}_s ( 1 + \log \hat{c}_{{\rm min}, s})  \, dx
    \\
    & \quad = \sum_{s = 1}^M \int_\Omega q_s \hat{\phi}_{\rm min}
    ( c_s - \hat{c}_{{\rm min}, s}) \, dx
    \qquad \mbox{[by integration by parts and Gauss' law \reff{Gauss4Ions}]}
    \\
    & \quad \quad 
    + \sum_{s=1}^M \int_\Omega( c_s - \hat{c}_{{\rm min},s}) 
    \left[ 1 + \log N_s - \log \left( \int_\Omega e^{-q_s \hat{\phi}_{\rm min}(y)} dy \right)
    - q_s \hat{\phi}_{\rm min} \right] dx
    \qquad \mbox{[by \reff{definecmins}]}
    \\
    & \quad = \sum_{s = 1}^M \left[ 1 + \log N_s - \log \left( 
    \int_\Omega e^{-q_s \hat{\phi}_{\rm min}(y) } dy\right) \right] 
    \int_\Omega (c_s - \hat{c}_{{\rm min}, s})\, dx
    \\
    & \quad = 0. \qquad \mbox{[by mass conservation \reff{mass}]}
\end{align*}
Hence $(\hat{c}_{\rm min}, \hat{D}_{\rm min})$ is a minimizer of $\hat{F}: X_\rho \to \R \cup \{ + \infty \}.$
The uniqueness follows from the strict convexity of the functional $\hat{F}.$

(2) Since $\hat{\phi}_{\rm min} \in L_{\rm per}^\infty(\Omega)$ 
(cf.\  Theorem~\ref{t:CCPBenergy}), the minimizer 
$(\hat{c}_{\rm min}, \hat{D}_{\rm min})$
satisfies (i). If $(\tilde{c}, \tilde{D}) \in \tilde{X}_0$, then 
$(\hat{c}_{\rm min}+t \tilde{c}, \hat{D}_{\rm min}
+t \tilde{D}) \in X_\rho$ and  
$
\hat{F}[\hat{c}_{\rm min}, \hat{D}_{\rm min}]\le \hat{F}[ 
\hat{c}+t \tilde{c}, \hat{D}_{\rm min}+t \tilde{D} ],
$
 if $|t| $ is small enough, and hence
$
(d/dt)|_{t=0} \hat{F}[ \hat{c}_{\rm min}+t \tilde{c}, \hat{D}_{\rm min}+t \tilde{D} ] =0.
$
This leads to 
\reff{tildeDc}.
Suppose $(c, D)\in X_\rho$ satisfies
(i) and (ii). Let $(\tilde{c}, \tilde{D}) = (\hat{c}_{\rm min}-c, \hat{D}_{\rm min}-D)\in \tilde{X}_0$.
Then we have 
\begin{align*}
     & \hat{F}[\hat{c}_{\rm min}, \hat{D}_{\rm min}]-\hat{F}[c, D] 
    \\
    &\quad = \int_\Omega  \frac{1}{2\ve} \left( | D + \tilde{D}|^2 
    - |D|^2 \right) dx + \sum_{s=1}^M \int_\Omega \left[ 
    (c_{ s} + \tilde{c}_s) \log ( c_{s} + \tilde{c}_s ) 
    -  c_{ s} \log c_{s} \right] 
    \\
    &\quad 
    \ge \int_\Omega  \frac{1}{\ve} D \cdot \tilde{D}  \, dx 
    + \sum_{s=1}^M \int_\Omega \tilde{c}_s ( 1 + \log c_{s})  \, dx
    \qquad \mbox{[by the convexity of $u \mapsto u\log u$]}
    \\
    & \quad = \int_\Omega  \frac{1}{\ve} D \cdot \tilde{D}  \, dx 
     + \sum_{s=1}^M \int_\Omega \tilde{c}_s \log c_{s}  \, dx
     \qquad \mbox{[by mass conservation \reff{mass} for $\hat{c}_{\rm min}$ and $c$]}
     \\
    & \quad = 0. \qquad \mbox{[by \reff{tildeDc}]}
\end{align*}
Hence, $(c, D)$ is also a minimizer and $(c, D) = (\hat{c}_{\rm min}, \hat{D}_{\rm min})$, since the minimizer is unique. 
\end{proof}

\section{Finite-Difference Approximations}
\label{s:Discretization}



We shall focus on the dimension $d = 3$ from now on. 
The case that the dimension $d = 2$ is similar and simpler. 
Moreover, since we focus on the local algorithms and their convergence, 
we consider for the simplicity of presentation only uniform 
finite-difference grids. 

\subsection{Finite-difference operators}
\label{ss:FDnotations}

Let $N \ge 1 $ be an integer. We cover $\overline{\Omega} 
= [0, L]^3$ with a uniform finite-difference grid of size $h = L/N.$ 
Denote $h\Z^3
= \{ (ih, jh, kh): i, j, k \in \Z \}$. 
For any (complex-valued) grid function 
$\phi: h\Z^3 \to \C$ and any $i, j, k \in \Z$, we denote 
$ \phi_{i,j,k} = \phi(ih, jh, kh)$ and 
\[
\partial^h_1 \phi_{i,j,k} = \frac{ \phi_{i+1, j, k} - 
\phi_{i, j, k}}{h}, \quad
\partial^h_2 \phi_{i,j,k} = \frac{ \phi_{i, j+1, k} - 
\phi_{i, j, k}}{h}, \quad
\partial^h_3 \phi_{i,j,k} = \frac{\phi_{i, j, k+1} - 
\phi_{i, j, k}}{h}. 
\]
We define the discrete forward gradient 
$\nabla_h \phi = (\partial^h_1\phi, \partial^h_2 \phi, \partial^h_3 \phi)$ on $h\Z^3$
and the discrete backward gradient $\nabla_{-h} \phi$ by 
$
\nabla_{-h} \phi_{i, j, k} = (\partial^h_1 \phi_{i-1, j,k}, 
\partial^h_2 \phi_{i, j-1, k}, \partial^h_3 \phi_{i, j, k-1})
$ for all $i, j, k \in \Z.$
The discrete Laplacian $\Delta_h \phi: h\Z^3 \to \C$ is defined to be 
$\Delta_h \phi = \nabla_{-h} \cdot \nabla_h \phi = \nabla_h \cdot \nabla_{-h} \phi,$ 
with the standard seven-point stencil. 
Given $\Phi = (u, v, w): h\Z^3 \to \C^3$, we define the discrete forward 
and backward divergence $\nabla_h \cdot \Phi \to \C$ and 
$\nabla_{-h} \cdot \Phi \to \C,$ respectively, by
\begin{align*}
&(\nabla_h \cdot \Phi)_{i, j, k} = (\partial^h_1 u)_{i,jk} +
(\partial^h_2 v)_{i, j, k} + (\partial^h_3 w)_{i, j,k}, 
\\
&(\nabla_{-h} \cdot \Phi)_{i, j, k} = (\partial^h_1 u)_{i-1,j,k} +
(\partial^h_2 v)_{i, j-1, k} + (\partial^h_3 w)_{i, j,k-1}.
\end{align*}

A grid function $\phi: h\Z^3 \to \C$ is $\overline{\Omega}$-periodic, 
if $\phi_{i+N, j, k} = \phi_{i, j+N, k} = 
\phi_{i, j, k+N} = \phi_{i,j,k} $ for all $i, j, k \in \Z.$ 
Given two $\overline{\Omega}$-periodic grid functions $\phi, \, \psi: h\Z^3 \to \C$, we define
\begin{align}
\label{phipsih}
&\langle \phi, \psi \rangle_h  =
   h^3 \sum_{i,j,k=0}^{N-1} \phi_{i,j,k} \overline{\psi_{i,j,k}} \qquad \mbox{and} \qquad 
   \| \phi \|_h = \sqrt{ \langle \phi, \phi \rangle}_h, 
   \\
   \label{dphidpsih}
    &\langle \nabla_h \phi, \nabla_h \psi \rangle_h  =
    h^3 \sum_{i,j,k=0}^{N-1} (\nabla_h \phi)_{i,j,k}
   \cdot \overline{(\nabla_h\psi)_{i,j,k}} \quad \mbox{and} \quad 
   \| \nabla_h \phi \|_h = \sqrt{ \langle \nabla_h \phi, \nabla_h \phi \rangle}_h,
\end{align}
where an over line denotes the complex conjugate. 
For any $\overline{\Omega}$-periodic grid function $\phi: h\Z^3 \to \C$, we define the discrete average
\begin{equation}
\label{Aveh}
    \sA_h (\phi) = \frac{1}{N^3} \sum_{i,j,k=0}^{N-1} \phi_{i, j, k}
    = \left( \frac{h}{L} \right)^3\sum_{i,j,k=0}^{N-1} \phi_{i, j, k}.
\end{equation}

The proof of the following lemma is given in Appendix: 

\begin{lemma}
    \label{l:FourierBasis} 
Let $\phi, \, \psi: h\Z^3 \to \C$ and $\Phi: h\Z^3 \to \C^3$ be 
$\overline{\Omega}$-periodic. The following hold true: 
\begin{compactenum}
    \item[{\rm (1)}]
    {\em The first discrete Green's identity:}
    $  \langle \nabla_{ \pm h} \cdot \Phi, \phi\rangle_{h} = - \langle \Phi, 
 \nabla_{\mp h} \phi \rangle_h;$

\item[{\rm (2)}]
{\em The second discrete Green's identity:}
$\langle \nabla_h \phi, \nabla_h \psi \rangle_h = 
- \langle \Delta_h \phi, \psi \rangle_h$. 


\item[{\rm (3)}]
{\em The discrete Poincar{\'e}'s inequality:} 
$\| \phi \|_h \le (L/4\sqrt{3}) \| \nabla_h \phi \|_h $ if $\sA_h(\phi) = 0.$
\qed
\end{compactenum}
\end{lemma}



In what follows, we shall consider real-valued grid functions. We define
\begin{align}
\label{Vh}
    &V_h = \{ \mbox{all $\overline{\Omega}$-periodic grid functions } \phi: h\Z^3\to \R \},
    \\
    \label{Vh0}
    &\mathring{V}_h = \left\{ \phi \in V_h: \sA_h(\phi) = 0 \right\}.
   \end{align}
The restriction of any $\phi \in C_{\rm per}(\overline{\Omega})$ onto $h \Z^3$, 
still denoted $\phi$, is in $V_h.$ 
Let $\ve \in C_{\rm per}(\overline{\Omega})$ satisfy
\reff{epsilon}. 
We define a new function on half grid points 
$(i+1/2, j, k), $ $(i, j+1/2, k)$, and $(i, j, k+1/2)$, also denoted $\ve$, by 
\begin{equation}
\label{halfve}
\ve_{i+1/2, j, k } = \frac{\ve_{i, j, k} + \ve_{i+1, j, k}}{2}, 
\ \ 
\ve_{i, j+1/2, k} = \frac{ \ve_{i, j, k} + \ve_{i, j+1, k}}{2}, 
\ \ 
\ve_{i, j, k+1/2} = \frac{\ve_{i, j, k} + \ve_{i, j, k+1}}{2}
\end{equation} 
for all $i, j, k \in \Z.$ For any $\phi \in V_h$, we define
$A_h^\ve [\phi ] \in V_h $ by 
\begin{align}\label{Ahvephi}
A_h^\ve [\phi]_{i, j, k} = \partial^h_1(\ve_{i-1/2, j, k}\partial^h_1 \phi_{i-1,j,k}) 
+\partial^h_2(\ve_{i, j-1/2, k}\partial^h_2 \phi_{i,j-1,k})
+\partial^h_3(\ve_{i, j, k-1/2}\partial^h_3 \phi_{i,j,k-1})
\end{align}
for all $i, j, k \in \Z.$
Clearly, $A_h^\ve: V_h  \to V_h$ is a linear 
operator.  If $\ve = 1$ identically, then $A_h^\ve = \Delta_h,$ which is
the discrete Laplacian. 
We denote for any $\phi, \psi \in V_h$ that 
\begin{align*}
&\langle \nabla_h \phi, \nabla_h \psi\rangle_{\ve, h}
= h^3 \sum_{i,j,k=0}^{N-1}
(\ve_{i+1/2, j, k}\partial_1^h\phi_{i,j,k} \partial_1^h
\psi_{i,j,k}
+  \ve_{i, j+1/2, k} \partial_2^h\phi_{i, j, k} \partial_2^h \psi_{i,j,k} 
\\
& \qquad \qquad \qquad \qquad \qquad \quad 
+  \ve_{i, j, k+1/2} \partial_3^h\phi_{i, j, k}\partial_3^h \psi_{i,j,k} ), 
\\
& \| \nabla_h \phi \|_{\ve, h} = \sqrt{ \langle \nabla_h \phi, \nabla_h \phi \rangle_{\ve, h}}. 
\end{align*}
The discrete Poincar{\' e}'s inequality implies that 
$\langle\cdot, \cdot \rangle_{\ve, h}$ is an inner product and 
$ \| \cdot \|_{\ve, h}$ the corresponding norm
of $\mathring{V}_h.$ If $\ve = 1$ then these are the same
as defined in \reff{dphidpsih}. 




Let $\ve \in C_{\rm per}(\overline{\Omega})$ satisfy \reff{epsilon}
and let ${\rho}^h \in \mathring{V}_h.$
Define
\[
I_h[\phi] = \frac{1}{2} \| \nabla_h \phi \|^2_{\ve, h} - \langle 
{\rho}^h, \phi \rangle_h \qquad \forall \phi \in \mathring{V}_h.
\]
As usual, we denote by $ \| \cdot \|_\infty$ the maximum-norm on 
$V_h$.
We use the notation $\sup_{h}$ to denote the supremum 
over $h = L/N$ for all $N \in \N.$

\begin{lemma}
\label{l:Ihmin}
\begin{compactenum}
    \item[{\rm (1)}]
There exists a unique minimizer $\phi^h_{{\rm min}} $ of $I_h: \mathring{V}_h \to \R. $  
 \item[{\rm (2)}]
 If ${\phi} \in \mathring{V}_h$ then the following are equivalent: 
 {\rm (i)} $\phi = \phi^h_{\rm min}$; 
 {\rm (ii)}   $\langle \nabla_h \phi, \nabla_h \xi \rangle_{ \ve, h }
    = \langle {\rho}^h, \xi \rangle_h $ for all $\xi \in \mathring{V}_h$; and 
 {\rm (iii)}  $A_h^\ve [\phi ] = -{\rho}^h $ on $h \Z^3.$
\item[{\rm (3)}]
{\rm (Uniform discrete $L^\infty$ and $W^{1,\infty}$ stability \cite{Pruitt_M2AN2015})}
   The linear operator $A_h^\ve: \mathring{V}_h \to \mathring{V}_h$ is 
   invertible  and 
$ \| (A_h^\ve)^{-1} \|_\infty + \max_{m=1, 2, 3} 
   \| \partial^h_m (A_h^\ve)^{-1} \|_\infty \le C $
with $C>0$ independent of $h.$
If $\sup_{h} \| {\rho}^{h} \|_\infty < \infty$, then 
$\|\phi^h_{\rm min} \|_\infty+ \| \nabla_h \phi^h_{\rm min} \|_\infty
\le C$ with $C >0$ independent of $h.$
\end{compactenum}
\end{lemma}

\begin{proof}
Parts (1) and (2) are standard.  Part (3) is 
proved by Pruitt \cite{Pruitt_M2AN2015,Pruitt_NumerMath2014} (cf.\ also
\cite{Beale_SINUM2009}).
\end{proof}

We define a discretized electric displacement as a vector-valued function
$D = (u, v, w): h(\Z+1/2)^3\to \R^3$ with 
\begin{equation}
    \label{Dijk}
D_{i+1/2, j+1/2, k+1/2} = (u_{i+1/2, j, k}, v_{i, j+1/2, k}, w_{i, j, k+1/2}) 
\qquad \forall i, j, k \in \Z.
\end{equation}
Here, $u_{i+1/2, j, k}$, $v_{i, j+1/2, k}$, and $w_{i,j,k+1/2}$ are
approximations of 
the first, second, and third components of a 
displacement at $((i+1/2)h, jh, kh)$, $(ih, (j+1/2)h, kh)$, and 
$(ih, jh, (k+1/2)h)$, the midpoints of the corresponding edges of the grid box, 
respectively. 
We denote
\begin{align}
\label{Yh}
Y_h = \{ \mbox{$\overline{\Omega}$-periodic functions } 
D = (u, v, w): h(\Z+1/2)^3\to \R^3 \mbox{ in the form \reff{Dijk}} \},
\end{align}
where $D: h(\Z+1/2)^3 \to \R^3$ is $\overline{\Omega}$-periodic
if $D(\xi + hN e) = D(\xi) $ for any $\xi 
\in h(\Z+1/2)^3$ and $e \in \{ (1, 0, 0), (0, 1, 0), (0, 0, 1)\}.$
Given $D = (u, v, w) \in Y_h,$ we denote 
\begin{align*}
 \sA_h(D) &= ( \sA_h(u), \sA_h(v), \sA_h (w))
=\frac{1}{N^3} \sum_{i, j, k=0}^{N-1} 
(u_{i+1/2, j, k}, v_{i,j+1/2, k}, w_{i, j, k+1/2}). 
\end{align*}
We also define the discrete divergence
$\nabla_h \cdot D: h\Z^3 \to \R $ and the discrete curl
$\nabla_h \times D: h(\Z+1/2)^3 \to \R^3$, respectively, by
\begin{align*}
&(\nabla_h \cdot D)_{i,j,k}= \frac{1}{h} \left(  u_{i+1/2,j,k}-u_{i-1/2,j, k}+v_{i,j+1/2, k}-v_{i,j-1/2,k} +w_{i,j,k+1/2}-w_{i,j,k-1/2} \right), 
\\
&(\nabla_h \times D)_{i+{1}/{2},j+{1}/{2},k+{1}/{2}}= \frac{1}{h}
\begin{pmatrix}
w_{i,j+1,k+1/2}-w_{i,j,k+1/2} -v_{i,j+1/2,k+1} + v_{i,j+1/2,k} \\
u_{i+1/2,j,k+1}-u_{i+1/2,j,k} -w_{i+1,j,k+1/2} +w_{i,j,k+1/2} \\
v_{i+1,j+1/2,k}-v_{i,j+1/2,k} -u_{i+1/2,j+1,k}+ u_{i+1/2,j,k}
\end{pmatrix}.
\end{align*} 
Note that the discrete curl at $(i+1/2, j+1/2, k+1/2)$ is defined through 
the three grid faces of the grid box $(i,j,k)+[0, 1]^3$ sharing
the same grid $(i, j, k).$ Each component of the vector 
represents the total electric displacement, an algebraic sum of the corresponding 
components of $D$, through the four edges of such a face. 
For instance, the last component of the curl is the algebraic sum of  
$u_{i+1/2,j,k}$, $u_{i+1/2, j+1, k}$, $v_{i, j+1/2, k},$ and $v_{i+1,j+1/2,k}$
corresponding to the edges of the face on the plane $z = kh$ which 
is the square with vertices $(i, j, k)$, $(i+1, j,k)$, $(i+1,j+1,k),$
and $(i, j+1, k)$. The signs of the $u$ and $v$ values 
in the sum are determined by circulation directions; cf.\ Figure~3.1. 
Note also that the components of the  discrete curl are
$ \partial_2^h w_{i, j, k+1/2} - \partial_3^h v_{i, j+1/2, k},$
$\partial_3^h u_{i+1/2, j, k} - \partial_1^h w_{i, j, k+1/2}, $
and $ \partial_1^h v_{i, j+1/2, k} - \partial_2^h u_{i+1/2, j, k},$
respectively, approximating those of the curl of a differentiable vector field.

\medskip

\begin{minipage}[c]{2 in}

\hspace{-16 pt}
\includegraphics[scale=0.44]{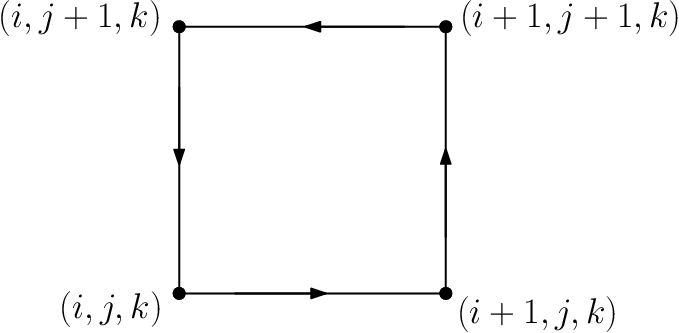} $ \ $ 
\label{f:curl}
\end{minipage}
\begin{minipage}{4.2 in} 
{\small {Figure~3.1.
The face of the grid box $(i, j, k)+[0, 1]^3$ sharing the vertex $(i, j,k)$
on which the last component of the curl 
 $(\nabla_h \times D)_{i+1/2, j+1/2, k+1/2}$ is defined. 
The counterclockwise direction of the displacement circulation along 
the edges determines the sign of the displacement components, 
positive (or negative) if the arrow points to a positive (or negative)
coordinate direction.}}
\end{minipage}

\medskip

Let $D = (u, v, w) \in Y_h$ and $\ve \in C_{\rm per}(\overline{\Omega})$ satisfy
\reff{epsilon}. We define $D/\ve \in Y_h$ by 
\begin{equation}
    \label{Doverve}
\left(\frac{D}{\ve}\right)_{i+1/2, j+1/2, k+1/2}
= \left( \frac{u_{i+1/2, j, k}}{ \ve_{i+1/2, j, k}}, 
\frac{v_{i, j+1/2, k}}{\ve_{i, j+1/2, k}}, 
\frac{w_{i, j, k+1/2}}{\ve_{i,j,k+1/2}}\right) \qquad \forall i, j, k \in \Z.
\end{equation}
If $\phi \in V_h$, we also define $ D_h^\ve[\phi] = (u, v, w) \in Y_h $ by 
\begin{align}
\label{uvephi}
&u_{ i+1/2, j, k} = - \ve_{i+1/2, j, k} \partial_1^h \phi_{i, j, k},
\ \
v_{ i,j+1/2, k} = - \ve_{i,j+1/2, k} \partial_2^h \phi_{i, j, k},
\ \ 
w_{ i, j, k+1/2} = - \ve_{i,j,k+1/2} \partial_3^h \phi_{i, j, k}.
\end{align}
It follows 
from the definition of $ A_h^\ve$ (cf.\ \reff{Ahvephi}) that 
\begin{equation}
    \label{AhDh}
A_h^\ve[\phi] =  -\nabla_h \cdot D_h^\ve [\phi] \qquad \forall \phi \in V_h.
\end{equation}

\begin{lemma}
\label{l:curlfree}
    If $D = (u, v, w) \in Y_h$ satisfies $\nabla_h \times D = 0$ on
    $h(\Z+1/2)^3$ and 
    $\sA_h(D) = 0$ in $\R^3,$
    then there exists a unique $\phi \in \mathring{V}_h$ such that 
   $D = D_h^\ve [\phi]$ with $\ve = 1$ identically. 
\end{lemma}

\begin{proof}
If $\phi_1, \phi_2 \in \mathring{V}_h$ and 
$\nabla_h \phi_1 = \nabla_h \phi_2$, then $\nabla_h (\phi_1 - \phi_2) = 0$. 
Thus $\phi_1 - \phi_2 $ is a constant on $h\Z^3$. Since $\phi_1 - \phi_2 
\in \mathring{V}_h$,  this constant must be $0$ and hence $\phi_1 = \phi_2.$ 
This is the uniqueness. 

Let $\rho^h = \nabla_h \cdot D \in V_h.$ The periodicity of $D$
implies that $\rho^h \in \mathring{V}_h.$ 
By Lemma~\ref{l:Ihmin} with $\ve = 1$, there exists a unique 
$ \phi \in \mathring{V}_h$ that minimizes $I_h: \mathring{V}_h\to \R$. 
Moreover, $A_h^\ve[\phi] = - \rho^h$ on $h\Z^3$ with $\ve = 1.$  
We define $\hat{D}= (\hat{u}, \hat{v}, \hat{w})\in Y_h$
by $\hat{D}  = D_h^\ve[\phi]$ with $\ve = 1$, i.e., by \reff{uvephi}
with $\hat{u}, $ $\hat{v}$, and $\hat{w}$ replacing
$u$, $v$, and $w$, respectively, and with $\ve = 1$ identically. 
Since $\ve = 1$, $\sA_h(\hat{D}) = 0$. 
By \reff{AhDh}, 
 $\nabla_h \cdot\hat{D} = - \nabla_h \cdot D^\ve_h[\phi]=
- A_h^\ve[\phi] = \rho^h$ on $h\Z^3.$ 
By the definition of discrete curl operator
and direct calculations using \reff{uvephi}
with $\hat{u}, $ $\hat{v}$, and $\hat{w}$ replacing
$u$, $v$, and $w$, respectively, we have 
$\nabla_h \times \hat{D} = 0$ on $h (\Z+1/2)^3.$
Denoting $\tilde{D} = (\tilde{u}, \tilde{v}, \tilde{w}) := D - \hat{D} \in Y_h$, 
we have $\nabla_h \cdot \tilde{D} = 0$ on  $h\Z^3$, 
$\nabla_h \times \tilde{D} = 0$ on $h (\Z+1/2)^3$, and 
$\sA_h(\tilde{D}) = 0$ in $\R^3.$
We shall show that $\tilde{D} = 0$ identically which will imply
that $D = \hat{D} = D_h^\ve[\phi]=-\nabla_h \phi$, the desired existence.

We first claim that 
each component of $\tilde{D} = (\tilde{u}, \tilde{v}, \tilde{w})$ satisfies
a discrete mean-value property, or equivalently, is a discrete harmonic function. 
Let us fix $i, j, k\in \Z$. We consider the two adjacent grid points labeled by 
$A = (i, j, k)$ and $B = (i+1, j, k)$, and also the four faces of grid boxes that 
share the common edge $AB$ connecting these two grid points; cf.\ Figure~3.2.
Since $-(\nabla_h \cdot \tilde{D})_{i,j,k} = 0$ and 
$(\nabla_h \cdot \tilde{D})_{i+1, j, k}=0$, we have 
\begin{align}
    \label{DDi}
&\tilde{u}_{i-1/2,j,k}-\tilde{u}_{i+1/2,j, k}+\tilde{v}_{i,j-1/2, k}
-\tilde{v}_{i,j+1/2,k} +\tilde{w}_{i,j,k-1/2}-\tilde{w}_{i,j,k+1/2}
= 0,
\\
\label{DDiplus1}
&\tilde{u}_{i+3/2,j,k}-\tilde{u}_{i+1/2,j, k}+\tilde{v}_{i+1,j+1/2, k}
-\tilde{v}_{i+1,j-1/2,k} 
+\tilde{w}_{i+1,j,k+1/2}-\tilde{w}_{i+1,j,k-1/2} = 0.
\end{align}
Two of the four faces sharing the edge $AB$ are on the plane
$y = jh$, one with the vertices $A$, $B$, $(i, j, k-1)$, and $(i+1, j, k-1)$,
and the other $A, $ $B, $ $(i, j, k+1)$, and $(i+1, j, k+1)$,
respectively.  The other two 
are on the coordinate plane $z = kh$, with vertices 
$A$, $B,$ $(i, j-1, k),$ and $(i+1, j-1, k),$ 
and $A$, $B,$ $(i, j+1, k)$, and $(i+1, j+1, k)$, respectively. 
 Since $\nabla_h \times \tilde{D} = 0$, we have, by
keeping the term $u_{i+1/2,j,k}$ with a negative sign, 
the four circulation-free equations on these four faces (cf.\ Figure~3.2)
\begin{align}
    \label{0C1}
    & \tilde{u}_{i+1/2, j, k-1}- \tilde{u}_{i+1/2,j,k} + 
    \tilde{w}_{i+1,j,k+1/2}- \tilde{w}_{i,j,k+1/2} = 0,
\\
\label{0C2}
& \tilde{u}_{i+1/2, j, k+1} - \tilde{u}_{i+1/2,j,k}  
+ \tilde{w}_{i,j,k+1/2} - \tilde{w}_{i+1,j,k+1/2} = 0,
\\
\label{0C3}
&\tilde{u}_{i+1/2, j-1, k} - \tilde{u}_{i+1/2, j, k} 
+ \tilde{v}_{i+1, j-1/2,k} -\tilde{v}_{i,j-1/2,k} = 0,
    \\
\label{0C4}
& \tilde{u}_{i+1/2,j+1, k}-\tilde{u}_{i+1/2, j, k} 
+ \tilde{v}_{i, j+1/2, k} - \tilde{v}_{i+1, j+1/2,k} = 0.
\end{align}
Consequently, by adding the same sides of all \reff{DDi}--\reff{0C4}, we obtain that
\begin{align}
    \label{meanvalue}
&\tilde{u}_{i+3/2, j, k} + \tilde{u}_{i-1/2, j, k} 
+ \tilde{u}_{i+1/2, j, k-1} + \tilde{u}_{i+1/2, j+1, k}
+ \tilde{u}_{i+1/2, j, k-1} + \tilde{u}_{i+1/2, j+1, k}
\nonumber \\
&\qquad 
-6 \tilde{u}_{i+1/2, j, k} = 0. 
\end{align}
Since $i, j, k\in \Z$ are arbitrary, $\tilde{u}$
satisfies the discrete mean-value property, i.e., 
$\tilde{u}$ is a discrete harmonic function. 
Similarly,  $\tilde{v}$ and $\tilde{w}$ are discrete harmonic functions.

\medskip

\begin{minipage}[c]{1.6 in}
$ \ $
\hspace{-20 pt}
\includegraphics[scale=0.42]{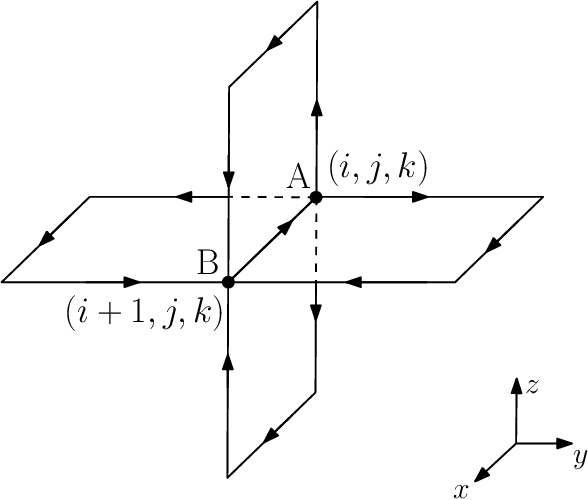}
	\label{f:harmo}
\end{minipage}
\begin{minipage}{4.6 in} 
   {\small {Figure~3.2.
The divergence-free of the displacement $\tilde{D}$ at the two 
vertices $A$ and $B$ (cf.\ \reff{DDi} and \reff{DDiplus1})
 and the zero circulation along the four edges
 of each of the four faces sharing the edge $AB$ that result from 
   the curl-free of $\tilde{D}$ (cf.\ \reff{0C1}--\reff{0C4}) 
   lead to the discrete harmonicity
   of the $\tilde{u}$-component of $\tilde{D}$ at the midpoint of the edge
   $AB$ (cf.\ \reff{meanvalue}).  
An arrow indicates the sign of a component
of $\tilde{D}$, positive (negative) if the arrow points in the 
positive (negative) coordinate direction.  
Note that the current from $B$ to $A$ is counted six times.}}
\end{minipage}

\medskip

To show finally that $\tilde{D} = 0$, 
it suffices to 
show $\tilde{u} = 0$ identically as we can similarly show that $\tilde{v} = 0$ 
and $\tilde{w} = 0$ identically. Let $p, q, r \in \Z$ be such that 
$\tilde{u}_{p+1/2, q, r} = \max_{i, j, k\in \Z} \tilde{u}_{i+1/2, j, k}.$
Then, it follows from the mean-value property \reff{meanvalue} with 
$(i, j, k) = (p, q, r)$ that $\tilde{u}$ also achieves its maximum value
at the $6$ neighboring points. 
Applying this argument to these $6$ neighboring points, and to the $6$ points
neighboring each of these $6$ points, and so on, we see that all
$\tilde{u}_{i+1/2, j, k}$ equal the maximum value. Hence $\tilde{u}$ is a constant. 
But, $\sum_{i,j,k=0}^{N-1} \tilde{u}_{i+1/2, j, k} = 0$. 
Hence, $\tilde{u} = 0$ identically. 
\end{proof}

\subsection{Approximation of the Poisson energy}
\label{ss:DiscretePoissonEnergy}


Given $\rho^h \in V_h$, we define (cf.\ \reff{Srho} and \reff{S0})
\begin{align}
\label{Srhoh}
& S_{\rho, h} = \{ D = (u, v, w) \in Y_h: \nabla_h \cdot D = \rho^h \mbox{ on } h\Z^3 \}, 
\\
\label{S0h}
& {S}_{0, h} = \{ D = (u, v, w) \in Y_h: \nabla_h \cdot D = 0 \mbox{ on } h\Z^3 \}.
\end{align}
The notation $S_{\rho, h}$ indicates that $\rho^h$ is 
a discrete approximation of a fixed 
$\rho\in L_{\rm per}^2(\Omega)$; cf.\ section~\ref{s:ErrorEstimates}. 
Clearly, $S_{0, h} \ne \emptyset$ as $D = 0$ is an element in $S_{0, h}.$

\begin{lemma}
\label{l:Srhoh}
    Let $\rho^h \in V_h$. Then $S_{\rho, h} \ne \emptyset$ if and only if 
    $\rho^h \in \mathring{V}_h.$
\end{lemma}

\begin{proof}
    If $S_{\rho, h} \ne \emptyset$ then there exits $D \in Y_h$ such that 
    $\nabla_h \cdot D = \rho^h$ on $h\Z^3$. Thus, 
    $ \sum_{i,j,k=0}^{N-1} \rho^h_{i,j,k} = \sum_{i,j,k=0}^{N-1} 
    (\nabla_h \cdot D)_{i,j,k}  =0,     $
    and hence $\rho^h \in \mathring{V}_h.$ Suppose $\rho^h \in \mathring{V}_h.$
    Let $\phi^h_{\rm min}$ be the minimizer of $I_h: \mathring{V}_h \to \R$ 
    with $\ve = 1$ identically, and hence 
$-\Delta_h \phi^h_{\rm min} = \rho^h$ on $h\Z^3$; cf.\ Lemma~\ref{l:Ihmin}. 
    Let $D = D_h^\ve[\phi^h_{\rm min}]  \in Y_h $  be defined by \reff{uvephi}
    with $\ve = 1$ identically. 
    We thus have $\nabla_h \cdot D = -\Delta_h \phi^h_{\rm min} = \rho^h$
    and hence $D \in S_{\rho, h}.$
\end{proof}

Let $\ve \in C_{\rm per}(\overline{\Omega})$ satisfy \reff{epsilon}. 
Define for any $D = (u, v, w), \tilde{D} = (\tilde{u}, \tilde{v}, \tilde{w}) \in Y_h$ 
\begin{align}
\label{DtildeDveh}
&\langle D, \tilde{D} \rangle_{1/\ve, h} = {h^3} \sum_{i,j,k=0}^{N-1}
\left(  \frac{u_{i+1/2, j, k} \tilde{u}_{i+1/2,j,k}}{\ve_{i+1/2, j, k}}
+ \frac{v_{i, j+1/2, k} \tilde{v}_{i,j+1/2,k}}{\ve_{i, j+1/2, k}}
+ \frac{w_{i,j,k+1/2} \tilde{w}_{i,j,k+1/2}}{\ve_{i,j,k+1/2}} \right), 
\\
\label{DiscreteDnorm}
&
\| D \|_{1/\ve, h} = \sqrt{\langle D, D \rangle_{1/\ve, h} }.
\end{align}
These are an inner product and the corresponding norm of the finite-dimensional
space $Y_h.$
Let $\rho^h \in \mathring{V}_h$. We define $F_h: S_{\rho, h} \to \R$ by
\begin{equation}
\label{Fh}
F_h[D] = \frac{1}{2} \| D \|^2_{1/\ve, h}
\qquad \forall D = (u, v, w) \in Y_h. 
\end{equation}

The following theorem provides some equivalent conditions on a minimizer 
of the functional $F_h: S_{\rho, h} \to \R$ that will be used to prove the
convergence of local algorithms: 


\begin{theorem}
\label{t:DiscreteEnergy} 
There exists a unique minimizer $D^h_{\rm min} 
= (u^h_{\rm min}, v^h_{\rm min}, w^h_{\rm min}) $ of $F_h: S_{\rho, h} \to \R$
given by  $D^h_{\rm min} = D_h^\ve [\phi^h_{\rm min}]$, 
where $ \phi^h_{\rm min}\in \mathring{V}_h$ is
the unique minimizer of $I_h: \mathring{V}_h \to \R$ as in Lemma~\ref{l:Ihmin}.
If $D = (u, v, w) \in S_{\rho, h}$, then the following are equivalent:
\begin{compactenum}
    \item[{\rm (1)}] 
    {\em Minimizer:} $D = D^h_{\rm min};$
    \item[{\rm (2)}] 
{\em Global equilibrium:}
$\langle D, \tilde{D} \rangle_{1/\ve, h} = 0 $ for all
$ \tilde{D} \in S_{0, h};$
\item[{\rm (3)}]
\begin{compactenum}
\item[{\rm (i)}] 
{\em Local equilibrium:}
$D/\ve$ is curl free, i.e., $\nabla_h \times D/\ve = 0$  
on $h(\Z+1/2)^3$; and 
\item[{\rm (ii)}] 
{\em Zero total field:} $\sA_h(D/\ve) = 0$ in $\R^3$. 
\end{compactenum}
\end{compactenum}
\end{theorem}

\begin{proof} 
By Lemma~\ref{l:Srhoh}, $S_{\rho, h} \ne \emptyset.$ 
Note that $Y_h$ is a finite-dimensional inner-product space, $S_{\rho, h}$
is a closed and convex subset of $Y_h$, and
$F_h: S_{\rho, h} \to \R$ is strictly convex. The existence of a unique minimizer, $D^h_{\rm min}
\in S_{\rho, h}$,   of $F_h: S_{\rho, h} 
\to \R$ follows from standard arguments. 

Before proving $D^h_{\rm min} = D_h^\ve [\phi^h_{\rm min}],$ we first
prove that Part (2) implies Part (1). Suppose $D \in S_{\rho, h}$ and
$\langle D, \tilde{D} \rangle_{1/\ve, h} = 0$ 
for all $\tilde{D} \in S_{0, h}$. With $\tilde{D} :=  D^h_{\rm min}-D\in S_{0,h}$, 
it follows
 \begin{align*}
      F_h[D^h_{\rm min}]-F_h[D] &=F_h[D+\tilde{D}]-F_h[D]
=  \frac{1}{2} \| \tilde{D}\|^2_{1/\ve, h} \ge 0.
       \end{align*}
Thus $D$ is also a minimizer of $F_h: S_{\rho, h} \to \R$ and hence 
$D = D^h_{\rm min}.$ Thus Part (2) implies Part (1). 

We now show that $D^h_{\rm min} = D_h^\ve [\phi^h_{\rm min}].$ First, 
it follows from Part (2) of Lemma~\ref{l:Ihmin} and \reff{AhDh} that
$\nabla_h \cdot D_h^\ve[\phi^h_{\rm min}] = - A_h^\ve[\phi^h_{\rm min}] = \rho^h$
on $h\Z^3.$ Thus, 
$D_h^\ve[\phi^h_{\rm min}] \in S_{\rho, h}.$ Since Part (2) implies Part (1), it now
suffices to show $\langle D^\ve_h[\phi^h_{\rm min}], \tilde{D} \rangle_{1/\ve, h} = 0$
for  any $\tilde{D} =(\tilde{u}, \tilde{v}, \tilde{w}) \in S_{0, h}.$
Denote $\phi = \phi^h_{\rm min} \in \mathring{V}_h$ and 
 $D = D_h^\ve[\phi] = (u, v, w)$. Then, 
the components of $D$ are given by \reff{uvephi}.
For fixed $j$ and $k$, we have by \reff{uvephi} and summation by parts that 
\begin{align}
\label{uminMore}
\sum_{i=0}^{N-1} \frac{u_{ i+1/2, j, k} \tilde{u}_{i+1/2, j, k}}{
\ve_{i+1/2, j, k}} 
= \frac{1}{h} 
\sum_{i=0}^{N-1} \phi_{i, j, k} ( \tilde{u}_{i+1/2, j, k}-\tilde{u}_{i-1/2, j, k}).
\end{align}
Similar identities hold true for the $v$ and $w$ components. Summing both sides of 
all these identities, we obtain by the fact that $\nabla_h \cdot \tilde{D} = 0$ and the 
definition \reff{DtildeDveh} that $\langle D, \tilde{D} \rangle_{1/\ve, h} 
= \langle \phi, \nabla_h \cdot \tilde{D} \rangle_h = 0.$ 
Hence, 
$D^h_{\rm min} = D_h^\ve [\phi^h_{\rm min}].$ 

We now prove that all Part (1), Part (2), and Part (3) are equivalent. 
If $D = D^h_{\rm min}$,
then for any $\tilde{D}\in S_{0, h}$, $g(t): = F_h [D + t \tilde{D}]$
$(t \in \R)$
attains its minimum at $t = 0$. Hence, $g'(0) = 0$, leading to 
$\langle D, \tilde{D}  \rangle_{1/\ve, h} = 0.$ 
Thus, Part (1) implies Part (2).  We already proved above that Part (2) implies 
Part (1). 

If $D = D^h_{\rm min}=D^h_{\ve} [\phi^h_{\rm min}]$, then $D := (u, v, w)$ is 
given by \reff{uvephi}
with $\phi^h_{\rm min}$ replacing $\phi$. 
Now by the definition of $D/\ve$ (cf.\ \reff{Doverve}) and that of the 
discrete curl operator, 
we can directly 
verify that $D/\ve$ is curl free. Hence, Part (1) implies (i) in Part (3). 
For any constant $(a, b, c) \in \R^3$, $ D + (a, b, c) \in S_{\rho, h}$.
Since 
$ g(a, b, c):= F_h[ D + (a, b, c)]$ 
$(a, b, c\in \R)$ 
reaches its minimum at $a = b = c = 0$, we have 
$
\partial_a g(0, 0, 0) = \partial_b g(0, 0, 0) = \partial_c g(0, 0, 0) = 0.
$
These imply 
(ii) in Part (3).
Thus, Part (1) implies Part (3). 

Suppose Part (3) is true.  It follows
from Lemma~\ref{l:curlfree}, applied to ${D}/\ve$, that 
${D}/\ve = -\nabla_h {\phi}$ for a unique  ${\phi} \in 
\mathring{V}_h$,  and thus $(D/\ve)_{i+1/2, j+1/2, k+1/2}=-\nabla_h \phi_{i,j, k}$
for all $i, j, k \in \Z.$
Consequently, setting $D = (u, v, w),$ 
we have by the same argument used above (cf.\ \reff{uminMore}) that 
$\langle D, \tilde{D}\rangle_{1/\ve, h} = 0$ 
for any $\tilde{D} = (\tilde{u}, \tilde{v}, \tilde{w}) \in S_{0,h}$. Thus, 
Part (3) implies Part (2). 
\end{proof}

\subsection{The discrete charge-conserved Poisson--Boltzmann equation}
\label{ss:FDCCPBE}

Let $\rho^h \in V_h$ and assume (cf.\ \reff{neutrality})
\begin{equation}
    \label{DiscreteNeutrality}
 \mbox{Discrete charge neutrality:} \qquad \qquad  \sum_{s=1}^M q_s N_s 
 + h^3\sum_{i,j,k=0}^{N-1} \rho^h_{i, j, k} = 0. \qquad \qquad \qquad 
\end{equation}
Let  $\ve \in C_{\rm per}(\overline{\Omega})$ satisfy \reff{epsilon}.
We define (cf.\ \reff{Iphi} and \reff{Aveh})
\begin{equation}
\label{hatIhphi}
\hat{I}_h [\phi] = \frac12 \| \nabla_h \phi \|^2_{\ve, h } 
-\langle \rho^h, \phi \rangle_h  + \sum_{s=1}^M N_s \log ( \sA_h(e^{-q_s \phi}))
\qquad \forall \phi \in V_h. 
\end{equation}
As in section~\ref{ss:CCPBE}, we can verify that $\hat{I}_h [\phi + a ] = \hat{I}_h [\phi]$ for any $\phi \in V_h$ 
and any constant $a \in \R$, the functional $\hat{I}_h: \mathring{V}_h \to \R$ is 
strictly convex, and by the discrete Poincar{\'e} inequality 
(cf.\ Lemma~\ref{l:FourierBasis}), 
there exist constant $K_1 > 0$ and $K_2 \in \R$, independent of $h$, such that 
$\hat{I}_h[\phi] \ge K_1 \| \nabla_h \phi\|_{\ve, h}^2 + K_2$ for all $\phi \in \mathring{V}_h.$

\begin{theorem}
    \label{t:hCCPB}
    There exists a unique $\hat{\phi}^h_{\rm min}\in \mathring{V}_h$ such that
    $ \hat{I}_h [\hat{\phi}^h_{\rm min}] = \min_{\phi\in \mathring{V}_h} \hat{I}_h [\phi].$
    The minimizer $\phi:=\hat{\phi}^h_{\rm min}$ is also the unique solution in $\mathring{V}_h$ to the 
    discrete CCPBE: 
    \begin{equation}
        \label{discreteCCPBE}
    A_h^\ve[\phi] +\sum_{s=1}^M \frac{q_s N_s}{  L^3 \sA_h(e^{-q_s \phi})}
 e^{-q_s \phi} = - \rho^h \qquad \mbox{on } h\Z^3. 
    \end{equation} 
Moreover, if in addition $\sup_{h} \| \rho^h \|_\infty < \infty$, then
$ \sup_{h} \| \hat{\phi}^h_{\rm min} \|_\infty < \infty.$
\end{theorem}

\begin{proof}
The space $\mathring{V}_h$ is finitely dimensional and the 
functional $\hat{I}_h$ on $\mathring{V}_h$ is strictly convex. It then follows 
that there exists a unique minimizer 
$\phi_{\rm min}^h \in \mathring{V}_h$ of $\hat{I}_h: \mathring{V}_h \to \R.$ 
Consequently, $\phi:=\phi_{\rm min}^h$ satisfies
\[
\langle \nabla_h \phi, \nabla_h \xi \rangle_{\ve, h}
- \langle \rho^h, \xi\rangle_h - \sum_{s=1}^M \frac{N_s q_s}{L^3\sA_h (e^{-q_s \phi})}
\langle e^{-q_s \phi}, \xi \rangle_h = 0 \qquad \forall \xi \in \mathring{V}_h. 
\]
Since $\rho^h + \sum_{s=1}^M q_s N_s (L^3 \sA_h ( e^{-q_s \phi} ) )^{-1} e^{-q_s \phi}
\in \mathring{V}_h$ by \reff{DiscreteNeutrality} and 
$\langle \nabla_h \phi, \nabla_h \xi \rangle_{\ve, h} = \langle
-A_h^\ve[\phi], \xi \rangle_h $ by summation by parts, we obtain \reff{discreteCCPBE}. 

Now assume $\sup_{h} \| \rho^h \|_\infty < \infty$. 
Let $\phi_0^h \in \mathring{V}_h$ be such that 
$\langle \nabla_h \phi^h_0, \nabla_h \xi \rangle_{\ve, h} = 
\langle \rho^h, \xi \rangle_h $ for all $\xi \in \mathring{V}_h; $
cf.\ Lemma~\ref{l:Ihmin}.  By Part (3) of Lemma~\ref{l:Ihmin}, 
there exists a constant $C > 0$, independent of $h$, such that 
\begin{equation}
    \label{maxphi0h}
    | \phi^h_{0, i, j, k} | \le C \qquad \forall i, j, k \in \Z. 
\end{equation}
Define (cf.\ \reff{hatIpsi})
\[
J_h [\psi] = \frac{1}{2} \| \nabla_h \psi \|^2_{\ve, h}
+ \sum_{s=1}^M N_s \log \left( \sA_h(e^{-q_s (\phi_0^h + \psi)}) \right) 
\qquad \forall \psi\in V_h.
\]
Let $\psi \in V_h$ and denote $\bar{\psi} = \sA_h (\psi).$ 
Since $\langle \nabla_h \phi_0^h, \nabla_h \psi\rangle_{\ve, h}
= \langle \rho^h, \psi-\bar{\psi} \rangle_h$ and $ \| \nabla_h \phi_0^h \|^2_{\ve,h}
= \langle \rho^h, \phi_0^h \rangle_h,$
we have by direct calculations that (cf.\ \reff{IpsiIphi})
\begin{align*}
J_h[\psi] &= J_h [\psi - \bar{\psi}] 
- \bar{\psi} \sum_{s=1}^M q_s N_s 
= \hat{I}_h[\psi-\bar{\psi}+\phi_0^h] + \frac12  \| \nabla_h \phi_0^h \|_{\ve, h}^2 
- \bar{\psi} \sum_{s=1}^M q_s N_s. 
\end{align*}
In particular, if $\psi \in \mathring{V}_h$ and 
$\phi = \psi+\phi_0^h \in \mathring{V}_h$, then
$J_h [\psi] = \hat{I}_h[\phi] + (1/2) \| \nabla_h \phi_0^h \|_{\ve, h}^2.$
Thus, $\psi^h_{\rm min} := \hat{\phi}^h_{\rm min} - \phi_0^h \in \mathring{V}_h$
is the unique minimizer of $J_h: \mathring{V}_0 \to \R.$ We show
that $\psi^h_{\rm min}$ is bounded uniformly with respect to $h.$ This will lead to 
the desired bound for $\hat{\phi}^h_{\rm min}.$

For convenience, let us denote $\psi=\psi^h_{\rm min}$ and 
$ \phi_0 = \phi_0^h$. We consider three cases as in the proof of Theorem~\ref{t:CCPBenergy}. 

Case 1: there exist $s', s''\in \{ 1, \dots, M \}$ 
such that $q_{s'} > 0$ and $q_{s''} < 0$.  Let $\lambda > 0$ and define 
\begin{align}
\label{hpsihatlambda}
\hat{\psi}_\lambda = \left\{
\begin{aligned}
    &\psi \quad & &\mbox{if } |\psi | \le \lambda, &
    \\
    &\lambda \quad & &\mbox{if } \psi > \lambda, &
    \\
    &-\lambda \quad & &\mbox{if } \psi < -\lambda, &
\end{aligned}
\right.
\qquad \mbox{and} \qquad 
\psi_\lambda = \hat{\psi}_\lambda - \sA_h (\hat{\psi}_\lambda).
\end{align}
We show that there exists $\lambda > 0$ sufficiently large and independent of $h$
such that for all $h$, 
\begin{equation}
\label{boundbylambda}
| \psi_{i, j, k} | \le \lambda \qquad \forall i, j, k \in \Z.
\end{equation}

It is clear that $\hat{\psi}_\lambda\in V_h$ and 
$\psi_\lambda \in \mathring{V}_h$, and hence $J_h [\psi] \le J_h[\psi_\lambda].$
Consider two neighboring grid points, e.g., $(i, j, k)$ and $(i+1, j, k)$. 
Let $\alpha = \psi_{i, j, k}$ and $\beta= \psi_{i+1,j,k}$, and 
assume $\alpha \le \beta$. (The case that $\beta \ge \alpha$ is similar.)
By checking the following six cases, we obtain
 $| \psi_{i+1, j, k} - \psi_{i, j, k}| \ge 
| \hat{\psi}_{\lambda, i+1, j, k} - \hat{\psi}_{\lambda, i, j, k}|$: 
(1) $\alpha \le \beta \le -\lambda$; 
(2) $\alpha \le - \lambda \le \beta \le \lambda;$
(3) $\alpha \le - \lambda < \lambda \le \beta$; 
(4) $ -\lambda \le \alpha \le \beta \le \lambda;$ 
(5) $-\lambda \le \alpha \le \lambda \le \beta;$ and 
(6) $\lambda \le \alpha \le \beta.$ 
Thus, $| \nabla_h \psi | \ge | \nabla_h \hat{\psi}_\lambda | 
= | \nabla_h \psi_\lambda |$ on $h\Z^3.$
Repeating \reff{case1main} with the summation
 replacing the integral over $\Omega$, we thus have
\begin{align}
\label{AvehAveh}
0 & \ge  \frac{1}{2} \| \nabla_h \hat{\psi}_\lambda \|_{\ve, h}^2
- \frac{1}{2} \| \nabla_h \psi \|_{\ve, h}^2 
\nonumber \\
& = J_h[\hat{\psi}_\lambda] - J_h[\psi] + 
\sum_{s=1}^M N_s \left[ \log \left( \sA_h (  e^{-q_s (\phi_0+\psi)} ) \right)
-  \log \left(  \sA_h ( e^{-q_s (\phi_0+\hat{\psi}_\lambda) } ) \right)\right]
\nonumber \\
& = J_h[{\psi}_\lambda] - J_h[\psi] - 
\sA_h( \hat{\psi}_\lambda) \sum_{s=1}^M q_s N_s 
\nonumber \\
&\qquad + 
\sum_{s=1}^M N_s \left[ \log \left( \sA_h (  e^{-q_s (\phi_0+\psi)} ) \right)
-  \log \left(  \sA_h ( e^{-q_s (\phi_0+\hat{\psi}_\lambda) } ) \right)\right]
\nonumber \\
& \ge \sA_h \left( B_h (\phi_0 + \psi) - 
B_h(\phi_0 + \hat{\psi}_\lambda ) \right)  
- \sA_h ( \hat{\psi}_\lambda)\sum_{s=1}^M q_s N_s, 
\end{align}
where $ B_h(u) = \sum_{s=1}^M (N_s/\alpha_{s, h}) e^{-q_s u} $ and 
$\alpha_{s, h}  = \sA_h ( e^{ -q_s (\phi_0 + \psi)} ).$

We claim that there are positive constants $C_1 $ and $C_2$, independent of 
$h$, such that 
\begin{equation}
\label{boundsalphash}
0 < C_1 \le \alpha_{s, h} \le C_2 \qquad \forall s = 1, \dots, M.
\end{equation}
In fact, by applying Jensen's inequality to $u \mapsto -\log u$ and the fact that 
$\phi_0, \psi \in \mathring{V}_h$,  we obtain that 
$\log \alpha_{s, h} \ge - q_s \sA_h (\phi_0 + \psi) = 0.$ 
Hence, $\alpha_{s, h} \ge 1 =:C_1$. 
Note that 
$
 \sum_{s=1}^M N_s \log (\alpha_{s, h} ) \le J_h [\psi] \le J_h[0] \le C,
$
where $C$ is a constant independent of h; cf.\ \reff{maxphi0h}. Since 
each $\alpha_{s, h} \ge C_1$, we have that each $\alpha_{s, h} \le C_2$
for some constant $C_2$ independent of $h.$ Thus, \reff{boundsalphash} is true. 

Suppose the desired property is not true. Then for any $\lambda > 0$
there is some $h$ such that with $\psi = \psi^h_{\rm min}$ the set
$\{ (i, j, k): \psi_{i, j, k} > \lambda \} \cup 
\{ i, j, k): \psi_{i, j, k} < -\lambda \} \ne \emptyset.$
We may assume both of these subsets of indices are nonempty as the case
that one of them is empty is similar. Set $b = \sum_{s=1}^M q_s N_s.$
It is clear that $B_h$ is a convex function. Thus, by Jensen's 
inequality and the fact that $\sA_h (\psi) = 0,$ we can continue 
from \reff{AvehAveh} to get
\begin{align}
\label{Avehb}
0 & \ge \sA_h  \left( [ B_h'(\phi_0 + \hat{\psi}_\lambda))
+ b ] (\psi - \hat{\psi}_\lambda) \right) 
\nonumber \\
& = h^3 \sum_{i,j, k:\, \psi_{i, j, k} > \lambda} [B_h'(\phi_{0, i, j, k}+ \lambda) + b]
(\psi_{i, j, k} - \lambda) 
\nonumber \\
&\qquad 
+ h^3 \sum_{i,j, k:\, \psi_{i, j, k} < -\lambda} [B_h'(\phi_{0, i, j, k} - \lambda) + b]
(\psi_{i, j, k} + \lambda).
\end{align}
Since $q_{s'} > 0$ and $q_{s''} < 0$, it follows from 
 \reff{boundsalphash} that for any $u \in \R$
\[
B_h'(u) = \sum_{s=1}^M \frac{N_s} {\alpha_{s, h}} (-q_s) e^{-q_s u}
\ge \sum_{s: \, q_s > 0}  \frac{N_s}{C_1} (-q_s) e^{-q_s u} +
\sum_{s: \, q_s < 0}  \frac{N_s}{C_2} (-q_s) e^{-q_s u} =:b_{h}(u). 
\]
The $h$-dependent function $b_{h}(u)$ is an increasing 
function of $u \in \R$. Moreover,  $b_{h}(+\infty) = +\infty $ 
and $b_h (-\infty) = - \infty.$ 
By \reff{maxphi0h}, we can then find $\lambda_+ > 0$ sufficiently 
large and independent of $h$ such that 
\[ 
B_h'(\phi_{0, i, j, k}+ \lambda) + b \ge b_{h}(\phi_{0, i, j, k}+\lambda) 
+ b \ge 1 \qquad \forall  \lambda \ge \lambda_+ \ \forall i, j, k \in \Z. 
\]
Similarly, there exists $\lambda_- > 0$ sufficiently large and independent of $h$ such that
\[
B_h'(\phi_{0, i, j, k} -\lambda) + b \le -1 \qquad \forall \lambda \ge \lambda_- \
\forall i, j, k \in \Z. 
\]
Let $\lambda \ge \max\{ \lambda_+, \lambda_-\}. $ It thus follows from \reff{Avehb} that
\[
0 \ge \sum_{i,j, k:\, \psi_{i, j, k} > \lambda} |\psi_{i, j, k} - \lambda| + 
\sum_{i,j, k:\, \psi_{i, j, k} < - \lambda} | \psi_{i, j, k} + \lambda |.
\]
This is impossible. Thus, \reff{boundbylambda} is true for all $h$. 

Case 2: all $q_s < 0$ $( 1 \le s \le M).$ 
For any $\lambda > 0$, we define now
$\hat{\psi}_\lambda = \psi$ if $\psi \le \lambda$ and 
$\hat{\psi}_\lambda = \lambda$ if $\psi > \lambda, $ and 
$\psi_\lambda = \hat{\psi}_\lambda - \sA_h( \hat{\psi}_\lambda).$
In this case, the function $B_h(u)$ defined above (below \reff{AvehAveh}) is convex
and 
\[
B_h'(u) \ge \sum_{s=1}^M \frac{(-q_s ) N_s}{C_2}  e^{-q_s u} =:{b}_{+, h} (u)
\qquad \forall u \in \R, 
\]
 where $C_2$ is the same as in \reff{boundsalphash}.  Thus, 
${b}_{+, h}(u)$ is an increasing function of $u \in \R$ and ${b}_{+, h}(+\infty)
= +\infty.$ Thus, 
carrying out the same calculations as above with $ \{ \psi > \lambda\} $ replacing
 $\{ | \psi | > \lambda\}$, we get $ \psi  \le \lambda$ on $h\Z^3$ for any $\lambda$
 large enough and independent of $h.$

 Since $\psi = \psi^h_{\rm min}$ is the minimizer of $J_h: \mathring{V}_h \to \R$, it
 is a critical point of $J_h$, which implies
 \[
    A_h^\ve[\psi] +\sum_{s=1}^M \frac{q_s N_s}{L^3 \alpha_{s,h}}
     e^{-q_s (\phi_0 + \psi)} = 0  \quad \mbox{on }  h\Z^3,
     \]
 where $\alpha_{s, h}$ is the same as above (defined below \reff{AvehAveh}). 
 Since $q_s < 0$ for all $s$, $\phi_0=\phi_0^h$ is uniformly bounded, and 
 $\psi $ is uniformly bounded above, we have by \reff{boundsalphash} and 
the uniform $L^\infty$-stability of the inverse of the operator 
$A_h^\ve: \mathring{V}_h \to \mathring{V}_h$
(cf.\ Lemma~\ref{l:Ihmin}) that 
 $\psi$ is also bounded below uniformly with respect to all $h > 0.$ 

Case 3: all $q_s > 0$ $ (s = 1, \dots, M).$ This is similar to Case 2. 
\end{proof}

\subsection{Approximation of the Poisson--Boltzmann energy}
\label{ss:FDapprPBenergy}

Let $\ve \in C_{\rm per}(\overline{\Omega})$ satisfy \reff{epsilon}
and $\rho^h \in V_h$  satisfy \reff{DiscreteNeutrality}. 
We consider discrete ionic concentrations
$c_s \in V_h$ $(s = 1, \dots, M)$ and the discrete electric displacement $D \in Y_h$ 
that satisfy the following conditions: 
\begin{align}
    \label{hpos}
    &\mbox{Nonnegativity:} & & c_{s,i,j,k} \geq 0,  \quad s=1,\dots, M;  \ i,j,k=1,\dots, N;& \\
    \label{hcons}
    & \mbox{Discrete mass conservation:} & & h^3 \sum_{i,j,k=0}^{N-1} c_{s,i,j,k}=N_s,  \quad  s = 1, \dots, M; &\\
    \label{hGauss}
    & \mbox{Discrete Gauss' law:} & & \nabla_h \cdot D= \rho^h 
    + \sum_{s=1}^M q_s c_s \quad  \mbox{on } h\Z^3. &
\end{align}
We define (cf.\ \reff{Xrho} and \reff{tildeX0})
\begin{align}
\label{Xrhoh}
&X_{\rho, h}=\{(c,D)=(c_1,\dots, c_M; D)
 \in V_h^M \times Y_h:  \text{ (\ref{hpos})--(\ref{hGauss}) hold true}\},
\\
&
\label{X0h}
\tilde{X}_{0, h}=\{(\tilde{c},\tilde{D}) = (\tilde{c}_1, \dots, 
\tilde{c}_M; \tilde{D}) \in 
\mathring{V}_h^M \times \in Y_h:  \nabla_h \cdot \tilde{D} = 
\sum_{s=1}^M q_s \tilde{c}_s \mbox{ on } h\Z^3 \}.  
\end{align}


\begin{lemma}
    \label{l:Xrhohnotempty}
    If $\rho^h \in V_h$ satisfies the condition \reff{DiscreteNeutrality}, then
    $X_{\rho, h} \ne \emptyset.$
\end{lemma}

\begin{proof}
    Let $c_s  = N_s / L^3 > 0$ on all the grids and for $s = 1, \dots, M.$ 
    Define $\tilde{\rho}^h = \rho^h
    + \sum_{s=1}^M q_s c_s \in \mathring{V}_h.$ Then, 
    by Lemma~\ref{l:Srhoh} with $\tilde{\rho}^h$ replacing $\rho^h$, there exists
    $D \in Y_h$ such that $\nabla_h \cdot D = \tilde{\rho}^h$ on $h\Z^3.$ 
    Consequently, $(c_1, \dots, c_s; D) \in X_{\rho, h}$. 
\end{proof}

We define the discrete Poisson--Boltzmann (PB) energy
\begin{align}
\label{hatFh}
\hat{F}_h [c,D] =  \frac{1}{2} \| D \|^2_{1/\ve, h}
+ h^3 \sum_{s=1}^M \sum_{i,j,k=0}^{N-1} c_{s, i, j, k} \log c_{s, i, j, k}
\qquad \forall (c, D) \in X_{\rho, h}.
\end{align}
Let $\hat{\phi}^h_{\rm min} $ be the unique minimizer of the functional 
$\hat{I}_h: \mathring{V}_h \to \R$ as in Theorem~\ref{t:hCCPB}. Define
\begin{align}
\label{hatchmin}
    &\hat{c}^h_{{\rm min}, s} = \frac{N_s}{L^3\sA_h (e^{-q_s \hat{\phi}^h_{\rm min}})}
    e^{-q_s \hat{\phi}^h_{\rm min}}, \qquad s = 1, \dots, M, 
    \\
    \label{hatDhmin}
    &\hat{D}^h_{\rm min} = D^\ve_h[\hat{\phi}^h_{\rm min}]; 
\end{align}
cf.\ \reff{uvephi} for the definition of $D^\ve_h$.
Denote $\hat{c}^h_{\rm min} = (\hat{c}^h_{{\rm min}, 1}, 
\dots, \hat{c}^h_{{\rm min}, M}). $ 


\begin{lemma}
    \label{l:hminPB}
    Let $(c, D) = (\hat{c}^h_{\rm min}, \hat{D}^h_{\rm min})$ be defined as above. Then
    $(c, D) \in X_{\rho, h}$, $\nabla_h \times (D/\ve) = 0$ on $h (\Z+1/2)^3$. If 
    in addition $\sup_{h} \| \rho^h \|_\infty < \infty$, then
    there exist positive constants $\theta_1$ and $\theta_2$, independent of $h$, satisfying 
    \begin{equation}
    \label{UPB}
    \mbox{\rm  Uniform positive bounds:} \qquad 
    0 < \theta_1 \le c_s \le \theta_2  
    \qquad \mbox{on  } h\Z^3,  \  s = 1, \dots, M.
    \end{equation}
\end{lemma}

\begin{proof}
Direct calculations using \reff{AhDh} and \reff{discreteCCPBE} verify
that $(\hat{c}^h_{\rm min}, \hat{D}^h_{\rm min}) \in X_{\rho, h}$  and 
$\nabla_h \times (D/\ve) = 0$. 
The bounds \reff{UPB} follow from  Theorem~\ref{t:hCCPB}. 
\end{proof}


\begin{theorem}
\label{t:DiscretePBEnergy} 
The pair of concentrations and displacement
$(\hat{c}^h_{\rm min}, \hat{D}^h_{\rm min})$ defined in \reff{hatchmin} and \reff{hatDhmin} 
is the unique minimizer of $\hat{F}_h: X_{\rho, h} \to \R$. 
Moreover, if $(c, D)=(c_1, \dots, c_M; u, v, w) \in X_{\rho, h}$, then the following are equivalent: 
\begin{compactenum}
 \item[{\rm (1)}] 
 $(c, D) = (\hat{c}^h_{\rm min}, \hat{D}^h_{\rm min})$; 
 \item[{\rm (2)}]
 \begin{compactenum}
\item[{\rm (i)}]
{\em Positivity:} $c_s > 0$ on $h \Z^3$ for all $s = 1,\dots, M;$ and 
\item[{\rm (ii)}]
    {\em Global equilibrium:} 
     \begin{equation}
         \label{GlobalEquil}
   \langle D,\tilde{D} \rangle_{1/\varepsilon, h}+
    \sum_{s=1}^M \langle \tilde{c}_s, \log c_s \rangle_h = 0
    \qquad \forall (\tilde{c}, \tilde{D}) = ( \tilde{c}_1, 
    \dots, \tilde{c}_M; \tilde{D}) \in \tilde{X}_{0, h}; 
    \end{equation}
\end{compactenum}
\item[{\rm (3)}]
\begin{compactenum}
\item[{\rm (i)}]
   {\em Positivity:} $c_s > 0$ on $h \Z^3$ for all $s = 1,\dots, M;$ and 
\item[{\rm (ii)}]
    {\em Local equilibrium---finite-difference Boltzmann distributions:}

$(\nabla \log c_s)_{i, j, k} = h q_s (D/\ve)_{i+1/2,j+1/2,k+1/2}$, i.e.,  
     \begin{equation}
         \label{logloglog}
         \left\{
         \begin{aligned}
& \log \frac{c_{s, i+1, j, k}}{c_{s, i, j, k}} 
    = \frac{hq_s u_{i+1/2, j, k}}{\ve_{i+1/2, j,k}}, \
    \\
    & \log \frac{c_{s, i, j+1, k}}{c_{s, i, j, k}} = 
    \frac{hq_s v_{i, j+1/2, k}}{\ve_{i, j+1/2, k}}, \
   \\
   & \log \frac{c_{s, i, j, k+1}}{c_{s, i, j, k}} = 
    \frac{hq_s w_{i, j, k+1/2}}{\ve_{i, j, k+1/2}}, 
    \end{aligned}
    \right.
    \qquad \forall s \in \{ 1, \dots, M\} \ \forall i, j, k \in \Z.
    \end{equation}
\end{compactenum}
\end{compactenum}
\end{theorem}

\begin{proof}
Note that, with $h$ fixed, the functional $\hat{F}_h: X_{\rho, h} \to \R$ is defined on 
a compact subset of a finitely dimensional space. It is strictly convex
and bounded below, 
and $\hat{F}_h [c, D] \to \infty$ if $\| (c, D) \| \to +\infty$ with 
respect to any fixed norm on the underlying finitely dimensional 
space. Therefore, it has a unique minimizer. 


Denoting $(c, D) := (\hat{c}^h_{\rm min}, \hat{D}^h_{\rm min})$, we show it 
is the minimizer.
We first show that it satisfies the condition
of global equilibrium \reff{GlobalEquil}.
Let $(\tilde{c},\tilde{D})=(\tilde{c}_1, \dots, \tilde{c}_M; 
\tilde{D}) \in \tilde{X}_{0, h}.$ Then, $\nabla \cdot \tilde{D} = \sum_{s=1}^M 
q_s \tilde{c}_s.$ It follows from the definition of $D^\ve_h$ (cf.\ \reff{uvephi})
and summation 
by parts (cf.\ \reff{uminMore}) that 
\begin{equation}
    \label{hGlobal1}
\langle D, \tilde{D}\rangle_{1/\ve, h} = \langle \hat{\phi}^h_{\rm min}, \nabla_h \cdot\tilde{D}
\rangle_h = \sum_{s=1}^M q_s \langle \hat{\phi}^h_{\rm min}, \tilde{c}_s \rangle_h. 
\end{equation}
Noting that $\sA(\tilde{c}_s) = 0$ for all $s \in \{ 1, \dots, M\}, $ we get 
by \reff{hatchmin} that 
\begin{equation}
   \label{hGlobal2}
\sum_{s=1}^M \langle \tilde{c}_s, \log c_s \rangle_h = 
-\sum_{s=1}^M q_s \langle \tilde{c}_s, \hat{\phi}^h_{\rm min}\rangle_h.  
\end{equation}
Now \reff{hGlobal1} and \reff{hGlobal2} imply \reff{GlobalEquil}.

Denoting by $(c_{\rm m}, D_{\rm m})\in X_{\rho, h}$ 
the unique minimizer of $\hat{F}_h$ over $X_{\rho, h}$ and 
$(\tilde{c}, \tilde{D}) = (c_{\rm m}-c, D_{\rm m} - D)\in X_{0, h},$ we have
by the convexity of $x \mapsto x\log x$, the fact
that $\sum_{i,j,k=0}^{N-1} \tilde{c}_{s, i, j, k} = 0$ for all $s \in \{ 1, 
\dots, M \}$, and the global equilibrium property \reff{GlobalEquil} that
\begin{align}
\label{FFFhhh}
    & \hat{F}_h[c_{\rm m},D_{\rm m}]-\hat{F}_h[c,D]
   \nonumber \\
    & \qquad = \hat{F}_h [c+\tilde{c},D+\tilde{D}] - \hat{F}_h[c,D]
    \nonumber\\
    & \qquad \ge   \langle {D}, \tilde{D}  \rangle_{1/\ve, h} 
    + h^3 \sum_{s=1}^M \sum_{i,j,k=0}^{N-1}
    \left[ (c_{s, i, j, k}+\tilde{c}_{s, i, j,k} ) 
    \log (c_{s, i, j, k}+\tilde{c}_{s, i, j,k} ) 
   - c_{s, i, j, k}\log c_{s, i, j, k} \right] 
    \nonumber\\
     &\qquad \ge 
   \langle {D}, \tilde{D}  \rangle_{1/\ve, h} 
    + h^3 \sum_{s=1}^M \sum_{i,j,k=0}^{N-1} \tilde{c}_{s, i, j,k} 
    (1 + \log c_{s, i, j, k}) 
   \nonumber \\
    & \qquad 
     = \langle {D}, \tilde{D}  \rangle_{1/\ve, h} 
    + h^3 \sum_{s=1}^M \sum_{i,j,k=0}^{N-1}
    \tilde{c}_{s, i, j,k} \log c_{s, i, j, k}
    \nonumber\\
    & \qquad = 0.
\end{align}
Thus,  $(c,D)=(c_{\rm m}, D_{\rm m})$ is the  minimizer of $\hat{F}_h:
X_{\rho, h} \to \R. $ 

We now prove that all Part (1)--Part (3) are equivalent. 
First, we prove that Part (1) implies Part (2). Suppose Part (1) is true: 
$(c, D) = (\hat{c}^h_{\rm min}, \hat{D}^h_{\rm min}).$ The positivity (i)
of Part (2)
follows from Lemma~\ref{l:hminPB}.  The condition
of global equilibrium (ii) of Part (2) is proved above; cf.\ \reff{hGlobal1} and 
\reff{hGlobal2}. Thus, Part (2) is true.

The fact that Part (2) implies Part (1) is proved above; cf.\ \reff{FFFhhh}, 
where only the positivity of $c$ instead of the uniform positive boundedness is
needed. 

We now prove that Part (1) implies Part (3). 
Let $(c, D) = (\hat{\phi}^h_{\rm min}, \hat{D}^h_{\rm min}) \in 
X_{\rho, h}$ be the minimizer of 
$\hat{F}_h: X_{\rho, h} \to \R.$ We need only to prove
the local equilibrium property  \reff{logloglog}.
Let us fix $s \in \{ 1, \dots, M \}$ and a grid point $(i, j, k)$ 
with $0\le i, j, k \le N-1$.  Define 
$\hat{c}_s = c_s$ at all $(p, q, r) $ with $0 \le p, q, r\le N-1$
except $ \hat{c}_{s,i, j, k} = c_{s, i, j, k} + \delta $ and 
$ \hat{c}_{s,i+1,j,k}=c_{s,i+1,j,k}-\delta$,  
where $\delta \in \R $ is such that $-c_{s, i, j, k} < \delta <  c_{s, i+1, j, k}.$
Extend $\hat{c}_s$ periodically. For $s'\ne s$, we set $\hat{c}_{s'} = c_{s'}$. 
Let us also define $\hat{D} = (\hat{u}, \hat{v}, \hat{w}) \in Y_h$ by setting
$\hat{v} = v$ and $\hat{w} = w$ everywhere, and $\hat{u} = u$ everywhere except 
$ \hat{u}_{i+1/2,j,k}=u_{i+1/2,j,k}+h q_s \delta$ (extended periodically). 
We can verify that $(\hat{c}, \hat{D})
= (\hat{c}_1, \dots, \hat{c}_M; \hat{D}) \in X_{\rho, h}.$ Let 
\begin{align*}
g(\delta) & := \hat{F}_h[\hat{c}, \hat{D}] - \hat{F}_h [c, D]
\\
& = \frac{1}{2} h^3  \frac{ \left(u_{i+1/2,j,k}+h q_s \delta \right)^2- 
    u_{i+1/2,j,k}^2 }{ \ve_{i+1/2, j, k} }
    \\
    &\quad +h^3\left[ (c_{s, i, j, k}  + \delta) 
    \log (c_{s, i, j, k} + \delta) - c_{s, i, j, k} \log c_{s, i, j, k}
    \right.
    \\
    &\quad 
    \left. 
    + \left( c_{s,i+1,j,k}-\delta \right) 
    \log \left( c_{s,i+1,j,k}-\delta \right)-c_{s,i+1,j,k} 
    \log  c_{s,i+1,j,k} \right].
\end{align*}
If $\delta = 0$ then $(\hat{c}, \hat{D} ) = (c, D)$, which is the minimizer of 
$\hat{F}_h: X_{\rho, h} \to \R$. Thus,  $g'(\delta) = 0.$ 
With direct calculations, this leads to the first equation in \reff{logloglog}. 
The other two equations can be proved by the same argument. Hence, Part (3) is true.

Finally, we prove that Part (3) implies Part (2). Let $(c, D) \in X_{\rho, h}$ and assume
it satisfies (i) and (ii) of Part (3). We need only 
to prove the global equilibrium property \reff{GlobalEquil}. 
Let $(\tilde{c}, \tilde{D}) = (\tilde{c}_1, \dots, \tilde{c}_M; \tilde{u}, 
\tilde{v}, \tilde{w}) \in \tilde{X}_{0, h}. $ Fix $\sigma \in \{ 1, \dots, M \}$ 
and fix  $j, k \in \{ 0, \dots, N-1 \}$. By \reff{logloglog} and 
summation by parts, we have 
\begin{align*}
    \sum_{i=0}^{N-1}
  \frac{u_{i+1/2, j, k} \tilde{u}_{i+1/2,j,k}}{\ve_{i+1/2, j, k}}
  & = \frac{1}{hq_{\sigma}} \sum_{i=0}^{N-1}
  \left( \log c_{\sigma, i+1, j, k} - \log c_{\sigma, i, j, k}
  \right) \tilde{u}_{ i+1/2, j, k} 
  \\
  & = - \frac{1}{hq_{\sigma} } \sum_{i=0}^{N-1} 
  \left( \tilde{u}_{ i+1/2, j, k} - \tilde{u}_{ i-1/2, j, k} \right)
\log c_{\sigma, i, j, k}. 
\end{align*}
Similar identities for $\tilde{v}$ and $\tilde{w}$ hold true. Therefore, 
it follows from the definition of $\nabla_h \cdot \tilde{D}$ and the 
fact that $\nabla_h \cdot \tilde{D} = \sum_{s=1}^M q_s \tilde{c}_s$ as 
$(\tilde{c}, \tilde{D}) \in X_{0, h}$ that 
\begin{align*}
\langle D, \tilde{D}  \rangle_{1/\ve, h} &= 
- \frac{h^3}{q_{\sigma} } \sum_{i,j,k=0}^{N-1} 
(\nabla_h \cdot \tilde{D})_{i,j,k} \log c_{\sigma, i, j, k} 
= - \frac{h^3}{q_{\sigma} } \sum_{s=1}^M \sum_{i,j,k=0}^{N-1} q_s \tilde{c}_s
\log c_{\sigma, i, j, k}. 
\end{align*}
Consequently, 
\begin{align}
\label{beforefinal}
    &\langle D, \tilde{D}  \rangle_{1/\ve, h} 
    + h^3 \sum_{s=1}^M \sum_{i, j, k=0}^{N-1}\tilde{c}_{s, i, j, k}
    \log c_{s, i, j, k}
\nonumber 
    \\
    & \qquad = h^3 \sum_{s = 1}^M q_s \left[ \sum_{i, j, k = 0}^{N-1}
     \tilde{c}_{s, i, j, k} 
    \left( \frac{1}{q_s} \log c_{s, i, j, k} - \frac{1}{q_\sigma} 
    \log c_{\sigma, i, j, k}\right) \right].
\end{align}
For each $s$, we define $\phi_s \in {V}_h$ by 
$ \phi_{s, i, j, k} = - q_s^{-1} \log c_{s, i, j, k} + \xi_s $ for all
$i, j, k \in \Z,$ where 
$\xi_{s } = N^{-3}q_s^{-1} \sum_{p, q, r = 0}^{N-1} \log c_{s, p, q, r}.$
Clearly,  $\phi_s \in \mathring{V}_h.$ It follows from \reff{logloglog} that 
\begin{align*}
(\nabla_h \phi_s)_{i,j,k} = -\frac{1}{q_s} (\nabla_h \log c_s)_{i,j,k}
 = - h \left( \frac{u_{i+1/2, j, k}}{\ve_{i+1/2, j, k}}, 
\frac{v_{i, j+1/2, k}}{\ve_{i, j+1/2, k}}, 
\frac{w_{i, j, k+1/2}}{\ve_{i, j, k+1/2}} \right)
\quad \forall i, j, k \in \Z.
\end{align*}
The right-hand side is independent of $s$. So, 
if $s, s' \in \{ 1, \dots, M \} $, then
$\nabla_h ( \phi_s - \phi_{s'} ) = 0$ on $h\Z^3$,  
which implies $\phi_s = \phi_{s'}$, since $\sA_h(\phi_s - \phi_{s'}) = 0$.  Thus, 
\[
\frac{1}{q_s}  \log c_{s, i, j, k}- \frac{1}{q_\sigma} 
    \log c_{\sigma, i, j, k} = \xi_s - \xi_\sigma
\qquad \forall i, j, k \in \Z. 
\]
Since $\sA_h(\tilde{c}_s) = 0$ for each $s$, this and \reff{beforefinal}
 imply \reff{GlobalEquil}. 
\end{proof}

\section{Error Estimates}
\label{s:ErrorEstimates}


We shall denote by $C$ a generic positive constant that is independent of
the grid size $h.$ Sometimes
we denote by $C = C(a, b, \dots, c)$ to indicate that the constant $C $ can depend on 
the quantities $a, b, \dots, c$ but is still independent of $h$. A statement is true for 
all $h > 0$ means it is true for all $h = L/N$ with any $N \in \N.$
Let $f\in C_{\rm per}(\overline{\Omega})$. 
Define $\sQ_h f \in V_h$ (cf.\ \reff{Vh}) by 
\begin{equation}
\label{defineQhf}
\sQ_h f = f + \sA_\Omega(f) - \sA_h(f) \qquad \mbox{on } h\Z^3. 
\end{equation}

\begin{lemma}
  \label{l:rhoerror} 
If $f \in C^2_{\rm per}(\overline{\Omega})$, then 
there exists a constant $C = C(f, \Omega) > 0$, independent of $h$, such that 
\[
|\sQ_h f - f | = | \sA_\Omega(f) - \sA_h (f) | \le Ch^2  \qquad \forall i, j, k \in \Z.  
\]
\end{lemma}

\begin{proof}

Let $B$ be any grid box and denote by
$P=P(B)$  and $V_i  = V_i(B)$ $(i =1, \dots, 8)$ its center and $8$ vertices, respectively. 
Denote $x=(x_1, x_2, x_3)$. 
Note that $|B| = h^3$, $\sum_{p=1}^8 (V_p - P) = 0$, and the integral of 
$x - P$ over $x \in B$ vanishes. 
Since $f \in C^2_{\rm per}(\overline{\Omega})$, it follows from Taylor's expansion that 
\[
\left| \dashint_B f \, dx - \frac{1}{8} \sum_{p=1}^8 f(V_p) \right|
\le \left| \dashint_B \left[ f(x) - f(P)\right] \, dx \right|
+ \left| \frac{1}{8} \sum_{p=1}^8 \left[ f(V_p) - f(P) \right] \right| \le Ch^2. 
\]
There are a total of $N^3$ grid boxes and, due to the $\overline{\Omega}$-periodicity
of $f$, each grid point is a vertex of $8$  
grid boxes. Thus, denoting by $\sum_{B}$ the sum over
all the $N^3$ grid boxes $B$, we have 

\begin{align*}
| (\sQ_h f)_{i, j, k} - f(ih, jh, kh)| = | \sA_\Omega (f) - \sA_h(f) | 
 = \left| \frac{1}{N^3} \sum_{B} \left[ \dashint_B f  \, dx - 
\frac{1}{8} \sum_{p=1}^8 f(V_p(B)) \right] \right|\le Ch^2 
\end{align*}
for any $i, j, k \in \Z$, completing the proof. 
\end{proof}

Let $D = (u, v, w) \in C_{\rm per}(\overline{\Omega}, \R^3)$. We define $\mathscr{P}_h D\in Y_h$
(cf.\ \reff{Yh} for the notation $Y_h$) by 
\begin{align}
\label{definePhD}
&( \mathscr{P}_h D )_{i+1/2, j+1/2, k+1/2} 
\nonumber \\
& \quad = (u((i+1/2)h, jh, kh), 
v(ih, (j+1/2)h, kh), w(ih, jh, (k+1/2)h)) \quad \forall i, j, k \in \Z.
\end{align}
Recall that $D_h^\ve[\phi]$ and $A_h^\ve[\phi]$ 
are defined in \reff{uvephi} and \eqref{Ahvephi}, respectively. 

\begin{lemma}
    \label{l:expandD}
\begin{compactenum}
\item[{\rm (1)}]
 If $D \in C^3_{\rm per}(\overline{\Omega}, \R^3)$,
 then for each $h$ there exists $\sigma^h \in V_h$ such that 
\begin{equation}
\label{dhDhe} 
 \nabla_h \cdot \sP_h D = \nabla \cdot D + \sigma^h h^2
\quad \mbox{and} \quad | \sigma^h| \le C \qquad \mbox{on } h\Z^3. 
\end{equation}

\item[{\rm (2)}]
 If $\ve \in C^2_{\rm per}(\overline{\Omega}) $ satisfies \reff{epsilon}, 
$\phi \in C^3_{\rm per}(\overline{\Omega})$, and $D  = -\ve \nabla \phi
\in C^3_{\rm per}(\overline{\Omega}, \R^3)$, then for each $h$ there exists
$T^h \in Y_h$ such that
\begin{equation}
\label{Deh}
 \sP_h D = D_h^{\ve}[\phi] + h^2 T^h 
\quad \mbox{ and } \quad |T^h| \le C \qquad \mbox{on } h (\Z+1/2)^3.  
\end{equation}
\item[{\rm (3)}]
 If $\ve \in C^2_{\rm per}(\overline{\Omega}) $ satisfies \reff{epsilon}, 
$\phi \in C^4_{\rm per}(\overline{\Omega})$, and $D = -\ve \nabla \phi
\in C^3_{\rm per}(\overline{\Omega}, \R^3)$, then for each $h$ there exists 
$\tau^h \in V_h$ such that 
\begin{equation}
\label{Trunc}
\nabla \cdot \ve \nabla \phi = A_h^\ve[\phi] + h^2 \tau^h  
\quad \mbox{and} \quad |\tau^h| \le C \qquad \mbox{on } h\Z^3.
\end{equation}
\end{compactenum}
\end{lemma}

\begin{proof}
(1) 
Let $D = (u, v, w)$ and $i, j, k \in \Z.$   
By the definition of $\sP_h D$ and $\nabla_h \cdot \sP_h D$, and  
Taylor expanding $u((i+1/2)h, jh, kh)$ and $u((i-1/2)h, jh, kh)$ 
at $u(ih, jh, kh)$, similarly for the $v$ and $w$ components of
$D$, we obtain \reff{dhDhe} with 
\[
\sigma^h_{i, j, k} = \frac{1}{24} \left[ \partial_1^3 u (\alpha_{i, j, k})
+ \partial_2^3 v(\beta_{i,j,k})+\partial_3^3 w(\gamma_{i, j, k}) \right] 
\]
for some $\alpha_{i, j, k}, \beta_{i, j, k}, \gamma_{i, j, k} \in \R^3.$

(2) Note that $\ve_{i, j, k} = \ve(ih, jh, kh)$ 
and $\ve_{i+1/2, j, k} = (\ve_{i,j,h}+\ve_{i+1, j,k})/2$ for all $i, j, k$; 
cf.\ \reff{halfve}.  Let us write $\partial_j = \partial_{x_j}$ with $x = (x_1, x_2, x_3)$. 
It then follows from Taylor's expansion at the point
$((i+1/2)h, j h, kh)$ that 
\begin{align*}
&\ve((i+1/2)h,jh,kh) = \ve_{i+1/2, j,k} 
- \frac{1}{8}h^2 \partial_1^2 \ve(\xi_{i, j, k}),
\\
&\partial_{1} \phi((i+1/2)h, jh, kh) = \frac{1}{h} 
\left[ \phi((i+1)h, jh, kh) - \phi(ih, jh, kh)\right]
- \frac{1}{24} \partial_1^3 \phi(\eta_{i,j,k}) h^2, 
\end{align*}
where $\xi_{i, j, k}, \eta_{i, j, k} \in [(ih, jh, kh), ((i+1)h, jh, kh)].$ 
Consequently, 
with $D = (u, v, w)$, 
\begin{align*}
&u((i+1/2)h, jh, kh) 
\nonumber \\
&\qquad = - \ve((i+1/2)h, jh, kh) \partial_1 \phi((i+1/2)h, jh, kh)
\nonumber \\
& \qquad = -\ve_{i+1/2, j,k} \partial_1 \phi((i+1/2)h, jh, kh)
+ \frac{1}{8}h^2 \partial_1^2 \ve(\xi_{i, j, k})\partial_1 \phi((i+1/2)h, jh, kh)
\nonumber \\
&\qquad = - \frac{\ve_{i+1/2, j, k} }{h} \left[ \phi((i+1)h, jh, kh) 
- \phi(ih, jh, kh)\right] + T^h_{i+1/2, j, k} h^2,
\end{align*}
where 
\begin{equation}
    \label{Th}
T^h_{i+1/2, j, k} = \frac{1}{8} h \partial_1^2 \ve(\xi_{i j, k}) \partial_1 \phi((i+1/2)h, j, k))
+\frac{1}{24} \ve_{i+1/2, j, k} \partial_1^3 \phi(\eta_{i, j, k}).
\end{equation}
Similar expansions hold for $v(ih, (j+1/2)h, kh)$ and $w(ih, jh, (k+1/2)h)$, respectively. 
Setting $T^h = (T^h_{i+1/2, j, k}, T^h_{i,j+1/2, k}, T^h_{i,j,k+1/2})\in Y_h$, 
we then obtain \reff{Deh}.

(3)
It follows from \eqref{dhDhe}, \eqref{Deh}, and \eqref{AhDh} that 
\begin{align*}
    \nabla \cdot \ve \nabla \phi & = - \nabla \cdot D 
    = -\nabla_h \cdot \sP_h D + \sigma^h h^2 
    \\
    & =- \nabla_h \cdot D_h^\ve [\phi ] - h^2 \nabla_h \cdot T^h + \sigma^h h^2
    = A_h^\ve[\phi] + \tau^h h^2  \qquad \mbox{on } h\Z^3,
\end{align*}
where $\tau^h = \sigma^h - \nabla_h \cdot T^h.$ 
Note that $\eta_{i, j, k}$ in \reff{Th} satisfies that 
$| \eta_{i, j, k} - (ih, jh, kh) | \le h$. 
Since 
\begin{align*}
&\ve_{i+1/2, j, k} \partial_1^3 \phi(\eta_{i, j, k})
- \ve_{i-1/2, j, k} \partial_1^3 \phi(\eta_{i-1, j, k})
= \ve_{i,jk} \left[ \partial_1^3 \phi(\eta_{i, j, k})
- \partial_1^3 \phi(\eta_{i-1, j, k})\right]
\\
& \qquad 
+ \frac{\ve_{i+1, j, k} - \ve_{i, j, k}}{2} \partial_1^3 \phi(\eta_{i, j, k})
+ \frac{\ve_{i,j,k} - \ve_{i-1, j, k}}{2} \partial_1^3 \phi(\eta_{i-1, j, k}),
\end{align*}
and similar expansions hold true for $\ve_{i,j+1/2, k}\partial_2^3\phi$
and $\ve_{i,j,k+1/2} \partial_3^3 \phi$ at respective points, 
Taylor's expansion and \reff{Th} imply
$|\nabla_h \cdot T^h | \le C$, and hence $|\tau^h| \le C$ on $h\Z^3.$
\end{proof}

We now present the error estimate for the finite-difference approximation
of the Poisson energy.  Let $\ve \in C_{\rm per}(\overline{\Omega}) $ satisfy \reff{epsilon} 
and $\rho \in C_{\rm per}(\overline{\Omega})$. If  $\sA_\Omega(\rho) = 0$, then
$\rho^h := \sQ_h \rho = \rho - \sA_h(\rho): h\Z^3 \to \R$ can be readily computed. 
Clearly, $\rho^h \in \mathring{V}_h$; cf.\ \reff{defineQhf}.
If $D, H \in Y_h$ (cf.\ \reff{Yh}), 
we denote $\langle D, H \rangle_h = \langle D, H \rangle_{1/\ve, h}$
and $ \| D \|_h = \| D \|_{1/\ve, h} $ with $\ve = 1; $ 
cf.\ \reff{DtildeDveh} and \reff{DiscreteDnorm}. 
For any $D = (u, v, w) \in C_{\rm per}(\overline{\Omega}, \R^3)$, 
we define $\| D \|_h  = \| \sP_h D \|_h.$  

\begin{theorem}
    \label{t:L2Poisson}
Assume  $\ve \in C^2_{\rm per}(\overline{\Omega})$ satisfies \reff{epsilon}, 
$\rho\in C^2_{\rm per}(\overline{\Omega})$ satisfies $\sA_{\Omega}(\rho)  = 0$, 
and $\rho^h := \sQ_h \rho \in \mathring{V}_h$. 
Let ${\phi}_{\rm min} \in \mathring{H}^1_{\rm per}(\Omega)$, 
$\phi^h_{\rm min} \in \mathring{V}_h$, $D_{\rm min } \in  S_{\rho}$, and 
 $D^h_{\rm min} \in S_{\rho, h}$ be the unique minimizers 
 of the functionals $I: \mathring{H}^1_{\rm per}(\Omega)\to \R$, 
 $I_h: \mathring{V}_h \to \R$, 
$F: S_\rho \to \R$, and  $F_h: S_{\rho, h} \to \R$, respectively. 
Assume that $ {\phi}_{\rm min} \in C^3_{\rm per}(\overline{\Omega})$ and 
 ${D}_{\rm min} \in C^3_{\rm per}(\overline{\Omega}, \R^3)$, then
there exists a constant $C = C(\ve, \rho, \Omega) > 0$, independent of $h,$
such that  
\[
\| \sP_h D_{\rm min} - D^h_{\rm min} \|_h \le C h^2.
\]
 If in addition ${\phi}_{\rm min} \in C^4_{\rm per}(\overline{\Omega})$, then
\[
\| \sP_h D_{\rm min} - D^h_{\rm min} \|_\infty \le C h^2.
\]
\end{theorem}

\begin{proof}
Let us denote 
\begin{equation}
\label{notations4P}
D = D_{\rm min}, \quad  \phi = \phi_{\rm min}, \quad 
D^h = D^h_{\rm min}, \quad \phi^h = \phi^h_{\rm min}, 
\quad  e_h^D=\sP_h D -D^h \in Y_h.
\end{equation}
By Lemma~\ref{l:expandD}, $\sP_h D  = D_h^{\ve}[\phi] + h^2 T^h $
with $T^h\in Y_h$ satisfying $| T^h | \le C $ on $h (\Z+1/2)^3.$
For any $\tilde{D}\in S_{0,h}$, which means $\nabla_h \cdot \tilde{D} = 0$, 
we have by summation by parts that 
$\langle D^{\ve}_h [\phi],\tilde{D}  \rangle_{1/\ve,h}= 0.$
Thus, $ \langle \sP_h D,  \tilde{D} \rangle_{1/\ve, h} \le Ch^2 \| \tilde{D}\|_h.$
By Theorem~\ref{t:DiscreteEnergy}, 
$\langle  D^h,\tilde{D} \rangle_{1/\ve,h}=0$. Hence, 
\begin{equation}\label{e:errorP}
    \langle  e_h ^D,\tilde{D} \rangle_{1/\ve,h} \leq C h^2 \| \tilde{D}\|_h
    \qquad \forall \tilde{D}\in S_{0,h}.
\end{equation}

Since $D  \in C^3_{\rm per}
(\overline{\Omega}, \R^3)$ and $D \in S_\rho$ which means $\nabla \cdot D = \rho$, 
it follows from Lemma~\ref{l:expandD} that 
$\nabla_h \cdot \sP_h D = \rho + \sigma^h h^2$ on $h \Z^3,$
where $\sigma^h \in V_h$ satisfies $| \sigma^h | \le C $ on $h\Z^3.$
Since $D^h \in S_{\rho, h}$ which impiles $\nabla_h \cdot D^h = \rho^h$, 
it follows that 
\[
\nabla_h \cdot e^D_h = \nabla_h \cdot (\sP_h D - D^h) =  h^2 q^h,
\]
where $q^h := h^{-2} (\rho-\rho^h) + \sigma^h $ satisfies
$| q^h | \le C$ on $h \Z^3$ by  Lemma~\ref{l:rhoerror}.
Moreover, $q^h \in \mathring{V}_h$ as $e^D_h$ is periodic. 
Thus, by Lemma~\ref{l:Ihmin}, there exists $\psi^h \in \mathring{V}_h$ such that 
$\Delta_h \psi^h = -q^h$ with $| \psi^h |\le C$ on $h\Z^3.$ 
Let $G^h=-\nabla_h \psi^h \in Y_h$. Then $\nabla_h \cdot G^h=q^h$ on $h\Z^3$. 
Moreover, by summation by parts and the Cauchy--Schwarz inequality, 
\begin{equation}
\label{G}
\| G^h \|_h^2 = \langle G^h, -\nabla_h \psi^h \rangle_h 
= \langle \nabla_h \cdot G^h, \psi^h \rangle_h = \langle q^h, \psi^h \rangle_h
\le \| q^h \|_h \, \| \psi^h \|_h \le C.
\end{equation}
Setting now 
$\tilde{D} = e_h^D-h^2 G^h \in S_{0,h}$ in \reff{e:errorP}, 
one then obtains 
\begin{equation*} 
    \langle e_h^D,e_h^D-h^2 G^h  \rangle_{1/\ve,h} \leq C h^2 \| e_h^D-h^2 G^h\|_h  
\le C h^2  \| e_h^D\|_h  +  C h^4.
\end{equation*}
This, together with \reff{G} and the identity
\begin{equation*}
 \| e_h^D- h^2 G^h\|^2_{1/\ve,h} +\| e_h^D\|^2_{1/\ve,h}
= 2 \langle  e_h^D,e_h^D-h^2 G^h \rangle_{1/\ve,h} + h^4 \| G^h\|^2_{1/\ve,h},  
\end{equation*}
implies
\begin{align*}
\| e_h^D \|^2_{h} 
\le 2 \langle  e_h^D,e_h^D-h^2 G^h \rangle_{1/\ve,h} + h^4 \| G^h\|^2_{1/\ve,h}
 \le Ch^2  \| e_h^D \|_{h} + Ch^4 
\le \frac12 \| e_h^D \|_h^2 + Ch^4.
\end{align*}
Consequently, we obtain  $ \| \sP_h D - D^h\|_h = \| e^D_h \|_h \le C h^2.$

Assume now $\phi\in C^4_{\rm per}(\overline{\Omega})$ 
and denote the error $r_h^\phi:= \phi- \phi^h$. 
By Lemma~\ref{l:expandD} and Lemma~\ref{l:rhoerror}, 
$| \nabla \cdot \ve \nabla \phi  - A_h^\ve[\phi] | \le C h^2 $ and 
$|\rho - \rho^h| \le C h^2$ on $h\Z^3.$
Since $\nabla \cdot \ve \nabla  \phi=-\rho$ and  
$ A_h^\ve[\phi^h]=-\rho^h$,  it follows that 
$A_h^\ve[{r}_h^\phi] = h^2 \alpha^h $ on $h\Z^3$ for some $\alpha^h \in V_h$  
with $ \| \alpha^h \|_\infty\le C.$  Clearly, $\alpha^h \in \mathring{V}_h$. 
Moreover, letting $\bar{r}_h^\phi = r_h^\phi - \sA_h(r_h^\phi)
\in \mathring{V}_h$, we get  
$A_h^\ve[\bar{r}_h^\phi] =A_h^\ve[r_h^\phi] =  \alpha^h h^2 $. 
Since $A_h^\ve: \mathring{V}_h \to \mathring{V}_h$
is linear and invertible, we have $\bar{r}^\phi_h = - h^2 (-A_h^\ve)^{-1}[\alpha^h]$, and 
further $ \partial_m^h \bar{r}^\phi_h = - h^2 \partial_m^h (-A_h^\ve)^{-1}[\alpha^h]$
for $m = 1, 2, 3.$ It now follows from Lemma~\ref{l:Ihmin} that 
\[
\| \partial_m^h r_h^\phi \|_\infty= 
\|\partial_m^h \bar{r}_h^\phi \|_\infty 
\le h^2 \| \partial_m^h (A_h^\ve)^{-1} \|_\infty \| \alpha^h \|_\infty 
\le C h^2, \qquad m = 1, 2, 3.
\]
This, together with \reff{dhDhe} in Lemma~\ref{l:expandD} and 
the fact that $D^h = D_h^\ve [\phi^h]$ by Theorem~\ref{t:DiscreteEnergy}, implies  
\[
\| \sP_h D - D^h \|_{ \infty} = \| D_h^\ve[r_h^\phi] + h^2 T^h \|_\infty
\le C \| \nabla_h r_h^\phi \|_\infty + h^2 \| T^h \|_\infty \le C h^2,  
\]
where $T^h \in Y_h$ is the same as in \reff{dhDhe}. 
\end{proof}

For any $D = (u, v, w)\in Y_h$ (cf.\ \reff{Yh}), we define 
$m_h[D]: h\Z^3 \to \R^3$ by 
\begin{equation}
\label{mhD}
(m_h[D])_{i, j, k} = \left( \frac{u_{i+1/2, j, k} + u_{i-1/2, j, k}}{2}, 
\frac{v_{i, j+1/2, k}+v_{i,j-1/2, k}}{2}, 
\frac{w_{i, j, k+1/2} + w_{i, j, k-1/2}}{2} \right)
\end{equation}
for all $i, j, k \in \Z. $
The following corollary shows that a simple post process of the computed
$D^h_{\rm min}$ super-approximates the gradient $\nabla \phi_{\rm min} $ at all
the grid points $(i, j, k)$: 

\begin{corollary}
    \label{c:D4GradphiP}
With the same assumptions as in Theorem~\ref{t:DiscreteEnergy}, including
$\phi_{\rm min} \in C^4_{\rm per}(\overline{\Omega})$, there exists a constant $C > 0$, 
independent of $h$, such that 
\[
\left\| \frac{m_h [-D^h_{\rm min}]}{\ve} - \nabla \phi_{\rm min}
 \right\|_\infty \le C h^2. 
\]
\end{corollary}


\begin{proof}
Let us use the notations in \reff{notations4P}. Since $D = (u, v, w) 
= - \ve \nabla \phi$, 
Taylor expanding  $(\ve \partial_1 \phi) ((i+1/2)h, jh, kh)$ and 
$ (\ve\partial_1 \phi ((i-1/2)h, jh, kh)$ at $(\ve\partial_1 \phi)(ih, jh, kh)$ leads to
\[
 \left|\frac{u_{i+1/2, j, k} + u_{i-1/2, j, k}}{2} + (\ve \partial_1 \phi) (i, j, k)
\right| \le C h^2 \qquad \forall i, j, k \in \Z.
\]
Similar inequalities hold with respect to $\partial_2$
and $\partial_3$. Hence, $| m_h [ \sP_h D ] + \ve \nabla \phi| \le C h^2 $ on $h \Z^3.$
But $ | m_h [D^h] - m_h [\sP_h D] | \le C h^2$ on $h\Z^3$ by Theorem~\ref{t:L2Poisson}. 
Thus, the desired inequality follows.
\end{proof}

We now present the error estimate for the minimizer of the finite-difference 
approximation of the PB energy functional
that is the same as the finite-difference solution to the discrete charge-conserved 
PB equation (CCPBE). 
Let $\rho\in C_{\rm per}(\overline{\Omega})$ satisfy \reff{neutrality}. By \reff{defineQhf} 
and \reff{neutrality}, 
\begin{equation}
\label{definerhoh4PB}
\sQ_h \rho = \rho + \sA_\Omega (\rho) - \sA_h(\rho) = 
\rho - \frac{1}{L^3} \sum_{s=1}^M q_s N_s 
-\frac{1}{N^3} \sum_{l,m,n=0}^{N-1} \rho(lh, mh, nh).
\end{equation}
So, $\sQ_h \rho$ can be computed readily. 
For any $(c, D) = (c_1, \dots, c_s; u, v, w)\in X_{\rho, h}$, we denote 
$ \| c \|_h$ by $ \|c\|_h^2 = \sum_{s=1}^M \| c_s \|_h^2$, where $\| \cdot \|_h$ is 
the norm of $V_h.$

\begin{theorem}\label{t:pbL2error}
Let $\ve \in C^2_{\rm per}(\overline{\Omega}) $ satisfy \reff{epsilon},  
$\rho \in C^2_{\rm per}(\overline{\Omega}) $ satisfy \reff{neutrality}, and 
$\rho^h := \sQ_h \rho$ be given by \reff{definerhoh4PB}. 
Let $\hat{\phi}_{\rm min} \in \mathring{H}^1_{\rm per}(\Omega)$, $\hat{\phi}^h_{\rm min}
\in \mathring{V}_h$, $(\hat{c}_{\rm min}, \hat{D}_{\rm min} ) \in X_{\rho}$,
and $(\hat{c}^h_{\rm min}, \hat{D}^h_{\rm min})\in X_{\rho, h}$ be the 
unique minimizer of $\hat{I}: \mathring{H}^1_{\rm per}(\Omega) \to \R \cup \{ + \infty \}$, 
$\hat{I}_h: \mathring{V}_h \to \R$, $\hat{F}: X_{\rho} \to \R \cup \{ + \infty \}$,
and  $\hat{F}_h: X_{\rho, h} \to \R$, respectively.
 Assume that $ \hat{\phi}_{\rm min} \in C^3_{\rm per}(\overline{\Omega})$ and 
 $\hat{D}_{\rm min} \in C^3_{\rm per}(\overline{\Omega}, \R^3)$. Then
 there exists a constant $C=C(\Omega, \ve, \rho, q_1, \dots, q_s, N_1, \dots, N_M)> 0$, 
 independent of $h$, such that
 \begin{align}
 \label{PBL2}
&
 \|\hat{c}_{\rm min} - \hat{c}^h_{\rm min} \|_h +  
\| \sP_h \hat{D}_{\rm min} - \hat{D}^h_{\rm min} \|_h 
 \le Ch^2, 
 \\
 \label{CCPBL2}
 & \| \hat{\phi}_{\rm min} - \hat{\phi}_{\rm min}^h \|_h \le Ch^2.
 \end{align}
 If in addition $\hat{\phi}_{\rm min} \in C^4_{\rm per}(\overline{\Omega})$, then
 \begin{align}
 \label{PBLinfty}
 & \|\hat{c}_{\rm min} - \hat{c}^h_{\rm min} \|_{ \infty}
 +  \| \sP_h
\hat{D}_{\rm min} - \hat{D}^h_{\rm min} \|_{\infty} \le Ch^2.
 \end{align}
\end{theorem}


\begin{remark}
We need the $L^2$-estimate \reff{PBL2} to get the estimate \reff{CCPBL2}, which is 
needed for proving the $L^\infty$-estimate \reff{PBLinfty}. 
\end{remark}

\begin{proof}[Proof of Theorem~\ref{t:pbL2error}]
Let us denote 
\begin{equation}
\label{notation5}
\phi= \hat{\phi}_{\rm min}, \quad 
\phi^h = \hat{\phi}^h_{\rm min}, \quad 
c = \hat{c}_{\rm min}, \quad 
D = \hat{D}_{\rm min}, \quad  
c^h = \hat{c}^h, \quad 
D^h = \hat{D}^h_{\rm min}. 
\end{equation}
By Theorem~\ref{t:PBEnergy} and Theorem~\ref{t:CCPBenergy}, $(c, D )$ is 
given by \reff{definecmins} and \reff{defineDmin} through ${\phi}\in 
\mathring{H}^1_{\rm per}({\Omega})$ which is also
the unique weak solution to the CCPBE \reff{CCPBE}.
By Theorem~\ref{t:DiscretePBEnergy} and Theorem~\ref{t:hCCPB}, $({c}^h, {D}^h)$
is given by \reff{hatchmin} and \reff{hatDhmin} through ${\phi}^h
\in \mathring{V}_h$  which is also 
the unique solution to the discrete CCPBE \reff{discreteCCPBE}.  

It follows from Lemma~\ref{l:expandD} that 
$\sP_h D = D_h^\ve[\phi]+h^2 T^h$ with $ | T^h | \le C $  on $ h(\Z+1/2)^3. $
Let $(\tilde{c}, \tilde{D}) \in \tilde{X}_{0, h}.$
Summation by parts leads to 
\begin{equation} \label{DND}
   \langle \sP_h D,\tilde{D}  \rangle_{1/\varepsilon, h}
  \leq \langle \phi,\nabla_h \cdot \tilde{D} \rangle_h + C h^2 \| \tilde{D}\|_{h}.
\end{equation}
By \reff{definecmins} in Theorem~\ref{t:PBEnergy},  
$\log c_s =\xi_s - q_s \phi$ for each $s,$ where 
$ \xi_s = -\log (N^{-1}_s L^{3} \sA_h ( e^{-q_s \phi})) $. 
Since $(\tilde{c}, \tilde{D}) \in \tilde{X}_{0, h}$ (cf.\ \reff{X0h}), 
each  $\tilde{c}_s \in \mathring{V}_h$ (cf.\ \reff{Vh0}) and 
$\nabla_h \cdot \tilde{D} =\sum_{s=1}^M q_s\tilde{c}_{s}.$ Hence, 
\begin{equation} 
\label{cslogcN}
\sum_{s=1}^M \langle \tilde{c}_s, \log c_s \rangle_h
=\sum_{s=1}^M \langle \tilde{c}_s, \xi_s  - q_s \phi \rangle_h 
= -\langle \phi,\nabla_h \cdot \tilde{D} \rangle_h.
\end{equation}
The combination of \reff{DND} and \reff{cslogcN} leads to
\begin{equation}\label{e:LTE2}
    \langle \sP_h D,\tilde{D}  \rangle_{1/\varepsilon, h}  +\sum_{s=1}^M 
    \langle \tilde{c}_s, \log c_s \rangle_h 
 \leq C h^2 \| \tilde{D}\|_{h} \qquad \forall (\tilde{c}, \tilde{D}) \in \tilde{X}_{0, h}.
\end{equation}
Let $e^D_h =\sP_h D-D^h$.  By Theorem~\ref{t:DiscretePBEnergy}, 
$(c^h, D^h) \in {X}_{\rho, h}$ satisfies the global equilibrium condition
\eqref{GlobalEquil}: $\langle  D^h,\tilde{D} \rangle_{1/\varepsilon, h} +\sum_{s=1}^M
\langle \tilde{c}_{s},\log c^h_{s} \rangle_h =0$. 
This and \eqref{e:LTE2} imply
\begin{equation}
    \label{e:errorPB1}
    \langle  e^D_h,\tilde{D}  \rangle_{1/\varepsilon, h}
    + \sum_{s=1}^M \langle \tilde{c}_s, \log c_s - \log c^h_s \rangle_h
\leq C h^2 \| \tilde{D}\|_h \qquad \forall (\tilde{c}, \tilde{D}) \in \tilde{X}_{0, h}.
\end{equation}

Since $(c, D) \in X_\rho$ and $(c^h, D^h) \in X_{\rho, h}$, we have 
$\nabla \cdot D = \rho + \sum_{s=1}^M q_s c_s$ in $\R^3$ and 
$\nabla_h \cdot D^h = \rho^h + \sum_{s=1}^M q_s c_s^h$ on $h\Z^3.$
Moreover, by Lemma~\ref{l:expandD}, 
$\nabla_h \cdot \sP_h D = \nabla \cdot D + \sigma^h h^2$ on $h \Z^3$ for 
some $\sigma^h \in V_h$ such that $ | \sigma^h| \le C $ on $ h\Z^3.$ Therefore, 
\begin{equation}\label{De-Dh}
\nabla_h \cdot e^D_h = 
\nabla_h \cdot \left( \sP_h D- D^h\right)
= \sum_{s=1}^M q_s (c_s-c_s^h) + \rho - \rho^h + \sigma^h h^2
\qquad \mbox{on } h\Z^3. 
\end{equation}
Define 
\[
\tilde{c}_s = c_s - c_s^h + \sA_\Omega(c_s) - \sA_h (c_s), 
\qquad s = 1, \dots, M.
\]
Since $c\in X_{\rho}$ (cf.\ \reff{Xrho}) and $c^h \in X_{\rho, h}$
(cf.\ \reff{Xrhoh}), $\sA_\Omega(c_s) = \sA_h (c_s^h) = N_s L^{-3}$. Hence
$\tilde{c}_s \in \mathring{V}_h$. 
It then follows from \reff{De-Dh} that 
\begin{equation}
    \label{newDhe}
\nabla_h \cdot e^D_h 
= \sum_{s=1}^M q_s \tilde{c}_s  + h^2 \gamma^h,
\end{equation}
where 
\[
h^2 \gamma^h = -\sum_{s=1}^M q_s  \left[ \sA_\Omega(c_s) - \sA_h (c_s) \right]
+ \rho - \rho^h + \sigma^h h^2. 
\]
By Lemma~\ref{l:rhoerror}, $|\gamma^h| \le C$ on $h \Z^3. $
Moreover, $\gamma^h \in \mathring{V}_h$, since $e^D_h$ is periodic and
each $\tilde{c}_s \in \mathring{V}_h$.  Thus, by Lemma~\ref{l:Ihmin}, 
there exists ${\psi}^h \in \mathring{V}_h$ such that 
$\Delta_h \psi^h = -\gamma^h$ with $| \psi^h |\le C$ on $h\Z^3.$ 
Denoting ${G}^h = - \nabla_h {\psi}^h \in Y_h$ and 
$\tilde{D} = e^D_h - h^2 {G}^h \in Y_h$, we then have by \reff{newDhe} that
$ \nabla_h \cdot \tilde{D} = \sum_{s=1}^M q_s\tilde{c}_s.  $
Hence, setting $\tilde{c} = (\tilde{c}_s, \dots, \tilde{c}_M)$,
we have $ (\tilde{c}, \tilde{D})  \in \tilde{X}_{0, h}$. 

Now, plugging the newly constructed $(\tilde{c}, \tilde{D}) \in \tilde{X}_{0, h}$ in
\eqref{e:errorPB1}, we obtain
\[
 \langle  e^D_h, e^D_h - h^2 G^h \rangle_{1/\varepsilon, h}
 + \sum_{s=1}^M \langle c_s - c_s^h + \sA_\Omega(c_s) - \sA_h(c_s), 
 \log c_s - \log c^h_s \rangle_h \leq C h^2 \| e_h^D - h^2 G^h\|_h. 
\]
Consequently, since $| \sA_\Omega(c_s) - \sA_h (c_s) | \le C h^2$ for all $s$ 
by Lemma~\ref{l:rhoerror}, we have 
\begin{align}
\label{newconstruct}
& \langle  e^D_h, e^D_h - h^2 G^h \rangle_{1/\varepsilon, h}
 + \sum_{s=1}^M \langle c_s - c_s^h, \log c_s - \log c^h_s \rangle_h 
\nonumber \\
& \qquad \leq C h^2 \| e_h^D\|_h +  C h^4 \|G^h\|_h 
+ Ch^2 \| \log c_s - \log c_s^h \|_h.  
\end{align}
Since $0 < C_1 \le c_s, c_s^h \le C_2$ on $h\Z^3$ for all $h$ and $s$
(cf.\ Theorem~\ref{t:PBEnergy} and Theorem~\ref{t:DiscretePBEnergy}), 
we have by the Mean-Value Theorem that 
\begin{align}
\label{MVT1} 
&\langle c_s - c_s^h, \log c_s - \log c_s^h \rangle_h \ge \frac{1}{C_2}
\| c_s - c_s^h \|_h^2, 
\\
\label{MVT2}
& \|\log c_s - \log c_s^h \|_h  \le \frac{1}{C_1} \| c_s - c_s^h \|_h.
\end{align}
Moreover, by summation by parts and the Cauchy--Schwarz inequality, 
\begin{equation}
\label{Gagain}
\| G^h \|_h^2 = \langle G^h, -\nabla_h \psi^h \rangle_h 
= \langle \nabla_h \cdot G^h, \psi^h \rangle_h = \langle \gamma^h, \psi^h \rangle_h
\le \| \gamma^h \|_h \, \| \psi^h \|_h \le C.
\end{equation}
It now follows from \reff{newconstruct}--\reff{Gagain} 
and the equivalence of the norms $\| \cdot \|_{1/\ve, h}$ and $\| \cdot \|_h$ that 
\begin{align*}
&\| e^D_h  \|_{1/\ve, h}^2 + \frac{1}{C_2} \|c-c^h\|_h^2
\\
&\qquad \le \langle e^D_h,  e^D_h - h^2 G^h \rangle_{1/\ve, h} 
+ \langle e^D_h, h^2G^h \rangle_{1/\ve, h}  +  \sum_{s=1}^M 
\langle c_s - c_s^h, \log c_s - \log c_s^h \rangle_h 
\\
& \qquad \le  C h^2 \| e_h^D \|_h  + C  h^4 + C h^2 \| c_s - c_s^h \|_h
\\
&\qquad \le \frac12 \| e^D_h \|_h^2 + \frac{1}{2C_2} \| c_s - c_s^h \|_h^2 + Ch^4,  
\end{align*}
leading to \reff{PBL2}.




By Lemma~\ref{l:expandD} (cf.\ \reff{Deh}) and the fact that 
$D^h = D_h^\ve[\phi^h]$, we have 
\[
\| \nabla_h \phi - \nabla_h \phi^h \|_h \le C_3 \| D^\ve_h[\phi] - D_h^\ve[\phi^h] \| 
\le C_3 \| \sP_h D - D^h \|_h + C_3 h^2 \le C h^2.
\]
Since $\phi^h $ and $\sQ_h \phi$ are in $ \mathring{V}_h$
and $\phi - \sQ_h \phi$ is constant on $h\Z^3$, the 
discrete Poincar{\'e} inequality (cf.\ Lemma~\ref{l:FourierBasis}) then
implies that 
\[
\| \sQ_h \phi - \phi^h \|_h \le C \| \nabla_h \sQ_h \phi - \nabla_h \phi^h \|_h 
= C \| \nabla_h \phi - \nabla_h \phi^h \|_h \le C h^2.
\]
This and Lemma~\ref{l:rhoerror} then imply \eqref{CCPBL2}. 

Assume now $\phi\in C^4_{\rm per}(\overline{\Omega})$.
Since $\phi$ and $\phi^h$ are solutions to the CCPBE \reff{CCPBE} and 
the discrete CCPBE \reff{discreteCCPBE}, respectively, it follows that 
\begin{equation}
\label{diffphiphih}
\nabla \cdot \ve \nabla \phi - A_h^\ve[\phi^h]
+\sum_{s=1}^M \frac{q_s N_s}{L^3}   \left[ \frac{ e^{-q_s \phi}} {\sA_\Omega(e^{-q_s \phi})}
- \frac{e^{-q_s \phi^h}}{L^3 \sA_h(e^{-q_s \phi^h})} \right] 
= \rho^h -  \rho \quad \mbox{on } h\Z^3.
\end{equation}
By Lemma~\ref{l:rhoerror}, Lemma~\ref{l:expandD}, 
the definition $\rho^h = \sQ_h \rho$, and \reff{definerhoh4PB}, we have
\begin{equation}
\label{twomore}
| \nabla \cdot \ve \nabla \phi  - A_h^\ve[\phi] | \le C h^2 
\quad \mbox{and} \quad 
|\rho - \rho^h| \le C h^2 \qquad \mbox{on }  h\Z^3.
\end{equation} 
Clearly, $ \| \rho^h\|_\infty \le C.$ Thus, it follows from Theorem~\ref{t:hCCPB} that $\|\phi^h\|_\infty \le C$ and that all $ \|e^{-q_s \phi^h} \|_\infty$, 
$\sA_\Omega( e^{-q_s \phi^h} )$, and $ \sA_h (e^{-q_s \phi^h})$ are bounded below and above by positive constants independent of $h.$ 
Consequently, the Mean-Value Theorem, the Cauchy--Schwarz inequality, 
and \reff{CCPBL2} together imply that for each $s$
\begin{align*}
\left| \sA_h (e^{-q_s \phi} ) - \sA_h (e^{-q_s \phi^h}) \right|
& \le \frac{1}{N^3} \sum_{i,j,k=0}^{N-1} \left| e^{-q_s \phi_{i,j,k} }
- e^{-q_s \phi_{i,j,k}^h} \right|
 \le \frac{C}{N^3} \sum_{i,j,k=0}^{N-1} 
\left| \phi_{i, j, k} - \phi^h_{i, j, k} \right| 
\\
& \le C \| \phi - \phi^h\|_h
 \le C h^2. 
\end{align*}
This and Lemma~\ref{l:rhoerror} imply 
\begin{equation}
 \label{AhAh2}
    | \sA_\Omega (e^{-q_s \phi} ) - \sA_h (e^{-q_s \phi^h})|
    \le | \sA_\Omega (e^{-q_s \phi} ) - \sA_h (e^{-q_s \phi}) |
    + | \sA_h (e^{-q_s \phi}) - \sA_h (e^{-q_s \phi^h}) | \le Ch^2.
\end{equation}
Denote the error $r_h^\phi:= \phi- \phi^h$. By \reff{twomore} and \reff{AhAh2}, 
we can now rewrite \reff{diffphiphih} into
 \begin{equation*}
 A_h^\ve[r_h^\phi]+\sum_{s=1}^M \frac{q_s N_s }{L^3 \sA_\Omega(e^{-q_s \phi})}
 \left(  e^{-q_s \phi} - e^{-q_s \phi^h} \right)= h^2 \alpha^h \qquad 
 \mbox{on } h\Z^3, 
 \end{equation*}
 where $\alpha^h \in V_h$ satisfies $| \alpha^h | \le C $ on $h \Z^3.$  Since
 $e^{-q_s\phi} - e^{-q_s \phi^h} = - q_s e^{-q_s \psi_s^h } r_h^\phi$ for some
 $\psi_s^h \in V_h$ which lies in between $\phi$ and $\phi^h$ at each $(i, j, k),$
the above equation for the error $r_h^\phi$ becomes
 \begin{equation}
 \label{beforeMh}
     -A_h^\ve[r^\phi_h] + b^h r^\phi_h = - h^2 \alpha^h, 
 \end{equation}
 where 
 $ b^h = \sum_{s=1}^M q_s^2 N_s e^{-q_s \psi_s^h}/(L^3 \sA_\Omega(e^{-q_s \phi}))
 \in V_h$ and  $C_4 \le b^h \le C_5$ on $h\Z^3$ for some constants
 $C_4 > 0$ and $C_5 > 0$ independent of $h.$ 
 
 As $V_h$ is a vector space of dimension $N^3$, the 
linear operator $M_h: V_h \to V_h$ defined by 
\[
 M_h \xi_h = -A_h^\ve[\xi_h] + b^h \xi_h \qquad \forall \xi_h \in V_h
\]
can be represented by a matrix ${\bf M}_h:= {\bf B}_h - {\bf A}_h^\ve$, where 
${\bf B}_h$ is the diagonal matrix with diagonal entries $b^h_{i, j, k}$ $(0\le i,j, k \le N-1)$
and ${\bf A}_h^\ve$ is the matrix representing the difference operator $A_h^\ve.$ 
By \reff{Ahvephi} and \reff{halfve}, ${\bf B}_h - {\bf A}^\ve_h$ is strictly diagonally dominant. In fact, if 
$M_{h, (i, j, k), (l, m, n)}$ is the entry of ${\bf M}_h$ in the row and column corresponding
to $(i, j, k)$ and $(l, m, n)$, respectively, then we can verify that
\[
\min_{(i,j,k)} \biggl(  | M_{h, (i, j, k), (i, j, k)} |- \sum_{(l,m,n)\ne (i, j, k)}
| M_{h, (i, j, k), (l, m, n)} | \biggr) = \min_{(i,j, k)} b^h_{ i, j, k} \ge C_4 > 0. 
\]
Therefore, the matrix ${\bf M}_h$ is invertible and $ \| {\bf M}_h^{-1} \|_\infty 
\le 1/C_4$; cf.\ \cite{Varah75,Varga76}. Hence, $M_h: V_h \to V_h$ is invertible and 
$\|M_h^{-1}\|_\infty \le 1/C_4.$ Since $| \alpha^h | \le C$ on $h\Z^3$, we have by \reff{beforeMh} that 
\begin{equation}
    \label{rhveinfty}
\| r_h^\phi \|_{\infty} = h^2 \| M_h^{-1} \alpha^h \|_\infty 
\le h^2 \| M_h^{-1} \|_\infty \| \alpha^h \|_\infty \le C h^2.
\end{equation} 

By \reff{notation5}, Theorem~\ref{t:PBEnergy}, Theorem~\ref{t:DiscretePBEnergy}, 
\reff{AhAh2},  \reff{rhveinfty}, and the bound $\| \phi^h \|_\infty \le C$, we have 
\begin{equation}
\label{ccfinal}
\| c_s - c_s^h \|_\infty =\frac{N_s}{L^3} \left\| 
\frac{e^{-q_s \phi} }{ \sA_\Omega (e^{-q_s \phi})} 
- \frac{ e^{-q_s \phi^h}}{\sA_h (e^{-q_s \phi^h})}
 \right\|_\infty \leq C h^2, \qquad s = 1, \dots, M.
\end{equation}
If we denote $\bar{r}_h^\phi = r_h^\phi - \sA_h (r_h^\phi) \in \mathring{V}_h$ and
$\beta^h = h^2 \alpha^h + b^h r_h^\phi \in V_h$, then \reff{beforeMh} becomes
$A_h^\ve[\bar{r}_h^\phi] = \beta^h $ on $h\Z^3.$ This implies $\beta^h \in 
\mathring{V}_h$. Moreover, $\| \beta^h \|_\infty \le C h^2$ by \reff{rhveinfty}. 
Since $A_h^\ve: \mathring{V}_h \to \mathring{V}_h$ is invertible, we have 
$\bar{r}_h^\phi = (A_h^\ve)^{-1}\beta^h.$ 
It follows now from Lemma~\ref{l:Ihmin} that 
\[
\| \partial_m^h r_h^\phi \|_\infty= 
\|\partial_m^h \bar{r}_h^\phi \|_\infty 
\le \| \partial_m^h (A_h^\ve)^{-1} \|_\infty \| \beta^h \|_\infty \le C h^2, \qquad m = 1, 2, 3.
\]
This and Lemma~\ref{l:expandD} imply
\[
\| \sP_h D - D^h \|_{ \infty} \le \| \sP_h D - D_h^\ve[\phi] \|_\infty
+ \| D_h^\ve[r^\phi_h] \|_\infty \le C h^2,
\]
which together with \reff{ccfinal} imply \reff{PBLinfty}. 
\end{proof}

The proof of the following corollary is similar to that of Corollary~\ref{c:D4GradphiP}:
\begin{corollary}
    \label{c:D4GradphiPB}
With the same assumptions as in Theorem~\ref{t:DiscretePBEnergy}, including
$\hat{\phi}_{\rm min} \in C^4_{\rm per}(\overline{\Omega})$, there exists a constant $C > 0$, 
independent of $h$, such that 
\[
\left\| \frac{m_h [-\hat{D}^h_{\rm min}]}{\ve}
 - \nabla \hat{\phi}_{\rm min} \right\|_\infty \le C h^2. 
\tag*{\qed}
\]
\end{corollary}

\section{Local Algorithms and Their Convergence}
\label{s:LocalAlg}

\subsection{Minimizing the discrete Poisson energy}
\label{ss:LcoalAlg4Poisson}

Given $\ve \in V_h$ with $\ve > 0$ and $\rho^h \in \mathring{V}_h.$ 
The local algorithm \cite{MaggsRossetto_PRL02,M:JCP:2002} 
for minimizing the discrete Poisson energy
$F_h: S_{\rho,h}\to \R$ defined in \reff{Fh}
consists of two parts. One is the initialization of 
a displacement $D^{(0)} = (u^{(0)}, v^{(0)}, w^{(0)}) \in S_{\rho, h}$ such that 
$\sA_h(D^{(0)}) = 0$. The other is 
the local update of the displacement at each grid box. 
To construct a desired initial displacement, we first define
\cite{DuenwegPRE09}
\begin{align*}
& \forall i, j \in \{ 0, \dots, N-1\}:  \quad \hat{w}^{(0)}_{i,j,1/2}=0, \quad 
\hat{w}^{(0)}_{i,j,k+1/2}=\hat{w}^{(0)}_{i,j,k-1/2}+ 
h p_k, \quad k=1,\dots,N-1, 
\\
&\forall i, k \in \{ 0, \dots, N-1\}: \quad 
\hat{v}^{(0)}_{i,1/2,k}=0, \quad 
\hat{v}^{(0)}_{i,j+1/2,k}=\hat{v}^{(0)}_{i,j-1/2,k}+ 
h q_{j,k}, \quad j=1,\dots,N-1, 
\\
&\forall j, k \in \{ 0, \dots, N-1 \}: \quad
\hat{u}^{(0)}_{1/2,j,k}=0, \quad
\hat{u}^{(0)}_{i+1/2,j,k}=\hat{u}^{(0)}_{i-1/2,j,k}+ 
h (\rho^h_{i,j,k}- p_k- q_{j,k}), 
\\
&\qquad \qquad \qquad \qquad \qquad \qquad \qquad \qquad 
\qquad i=1,\dots,N-1, 
\end{align*}
where $ p_k =(1/N^2) \sum_{l,m=0}^{N-1} \rho^h_{ l,m,k}$ and 
$ q_{j,k} =(1/N) \sum_{l=0}^{N-1} \rho^h_{l,j,k}- p_k$
$(j, k = 0, \dots, N-1).$ We extend 
$\hat{D}^{(0)} = (\hat{u}^{(0)}, \hat{v}^{(0)}, \hat{w}^{(0)}) $ periodically, and then
define $D^{(0)} = \hat{D}^{(0)} - \sA_h (\hat{D}^{(0)}).$
It is readily verified that $D^{(0)} \in  S_{\rho, h}$ and $\sA_h ( D^{(0)} ) = 0.$

We now describe the local update. Let $D = (u, v, w) \in S_{\rho, h}.$
Fix $(i, j, k)$ with $0 \le i, j, k \le N-1$ and consider the grid 
box $B_{i, j, k}= (i, j, k)+[0, 1]^3;$
cf.\ Figure~\ref{fig:relaxation} (Left).
We update $D$ on the edges of the three faces of $B_{i, j, k}$ 
that share the vertex $(i, j, k)$, 
first the face on the plane
$x = ih$, then $y=jh$, and finally $z = kh.$ 

\begin{figure}[htp]
    \centering
    \includegraphics[width=0.67\linewidth]{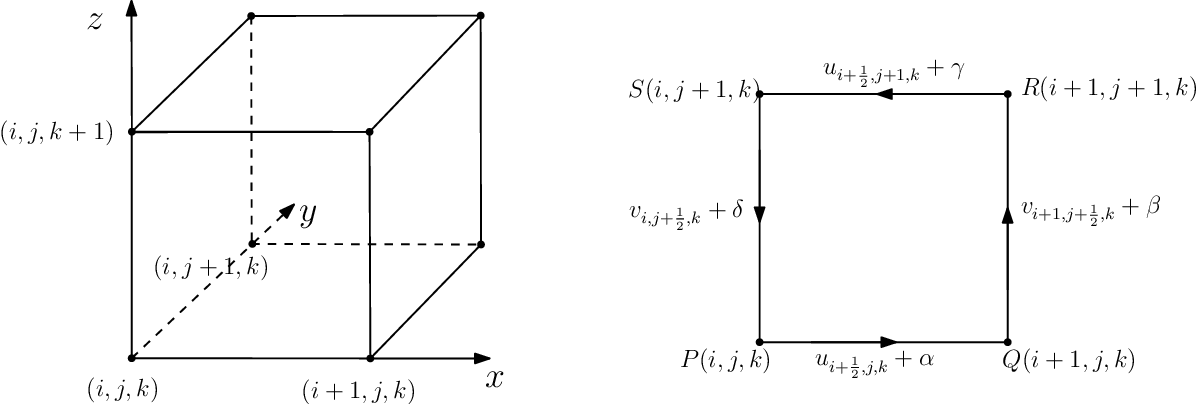}
    \caption{{\small { (Left) The grid box $B_{i, j, k} = (i, j, k)+[0,1]^3.$ 
(Right) The grid face of box $B_{i, j, k}$ with 
   vertices $P = (i, j, k)$, $Q = (i+1, j, k)$, $R = (i+1, j+1, k)$, 
and $S = (i, j+1, k).$ The perturbations $\alpha$, $\beta$, 
$\gamma$ and $\delta$ of $u$ and $v$ with subscript, 
the corresponding components of the displacement $D$, 
are to be determined.}}}
    \label{fig:relaxation}
\end{figure}

Consider the face on the plane $z = kh$, the square 
of vertices $P = (i, j, k)$, $Q = (i+1, j, k)$, $R = (i+1, j+1, k)$, 
and $S = (i, j+1, k);$ cf.\ Figure~\ref{fig:relaxation} (Right).
To update the $4$ values $u_{i+1/2, j, k}, $ $u_{i+1/2, j+1, k},$, $v_{i, j+1/2, k},$ 
and $v_{i+1, j+1/2, k}$ of $D$ on the $4$ edges of the face 
$PQRS$, we define a locally perturbed displacement
$\check{D} = (\check{u}, \check{v}, \check{w})  \in S_{\rho, h}$ 
by $\check{D} = D$ everywhere except
\begin{align*}
    &\check{u}_{i+1/2,j,k}=u_{i+1/2,j,k}+\alpha,
    \\
    & \check{v}_{i+1,j+1/2,k}=v_{i+1,j+1/2,k}+\beta, \\
    &\check{u}_{i+1/2,j+1,k}=u_{i+1/2,j+1,k}+\gamma, \\
    &
    \check{v}_{i,j+1/2,k}=v_{i,j+1/2,k}+\delta,
\end{align*}
where $\alpha, \beta, \gamma, \delta \in \R$ are to be determined. In order for
$\check{D} \in S_{\rho, h}$, the discrete Gauss' law $\nabla_h \cdot D = \rho^h$ 
at the $4$ vertices $P, Q, R, S$ should be satisfied. Consequently, 
$\alpha +\delta=0, $ $ -\alpha + \beta = 0, $ 
$-\beta - \gamma = 0, $ and $ \gamma-\delta = 0.$
Thus, $\alpha = \beta = - \gamma = - \delta =:\eta \in \R.$ 
The optimal value of $\eta$ is set to minimize the perturbed energy
$F_h[\check{D}]$, or equivalently, the energy change
\begin{align*}
    \Delta F(\eta) &:=F_h[\check{D}] - F_h[{D}]
    \\
    &=\frac{\ve_{z, i, j, k}h^3 }{2} \eta^2  
    +2\eta \left( \frac{u_{i+1/2,j,k}}{\ve_{i+1/2,j,k}}+ \frac{v_{i+1,j+1/2,k}}{\ve_{i+1,j+1/2,k}}-\frac{u_{i+1/2,j+1,k}}{\ve_{i+1/2,j+1,k}}-\frac{v_{i,j+1/2,k}}{\ve_{i,j+1/2,k}}\right) 
    \qquad \forall \eta \in \R,
\end{align*}
where 
\begin{equation*}
    \ve_{z, i,j,k}=\frac{1}{\ve_{i+1/2,j,k}} +\frac{1}{\ve_{i+1,j+1/2,k}} +\frac{1}{\ve_{i+1/2,j+1,k}} +\frac{1}{\ve_{i,j+1/2,k}}. 
\end{equation*}
This is minimized at a unique $\eta = \eta_{z, i, j, k}$ with the minimum energy change
$\Delta F_{z, i, j, k} :=\min_{\eta \in\R} \Delta F(\eta)$ given by 
\begin{align}
\label{etazijk}
&\eta_{z, i, j, k} = -\frac{1}{\ve_{z,i,j,k}}
\left( \frac{u_{i+1/2,j,k}}{\ve_{i+1/2,j,k}}+ \frac{v_{i+1,j+1/2,k}}{\ve_{i+1,j+1/2,k}}-\frac{u_{i+1/2,j+1,k}}{\ve_{i+1/2,j+1,k}}-\frac{v_{i,j+1/2,k}}{\ve_{i,j+1/2,k}} \right),
\\
\label{Fzijk}
&\Delta F_{z, i, j, k}  = - \frac{1}{2}  \ve_{z, i, j, k} h^3 \eta_{z, i, j, k}^2. 
\end{align}
Therefore, we update $D$ by 
\begin{align}
\label{Updatez1}
&{u}_{i+1/2,j,k} \leftarrow u_{i+1/2,j,k}+\eta_{z, i, j, k},
\\
\label{Updatez2}
   & {v}_{i+1,j+1/2,k} \leftarrow v_{i+1,j+1/2,k}+\eta_{z, i, j, k}, 
   \\
\label{Updatez3}
   &  {u}_{i+1/2,j+1,k} \leftarrow u_{i+1/2,j+1,k}-\eta_{z, i, j, k}, 
   \\
\label{Updatez4}
   &
    {v}_{i,j+1/2,k} \leftarrow v_{i,j+1/2,k}-\eta_{z, i, j, k}.
\end{align}
We denote by $D^z \in S_{\rho, h}$ this updated displacement. 

Similarly, we can update the $D$-values on the $4$ edges
of the face of the grid box $B_{i, j,k}$ on the plane $ y = jh $ and 
the plane $x  = i h$ to get the updated displacement $D^y\in S_{\rho, h}$ 
and $D^x \in S_{\rho, h},$ respectively, by
\begin{align}
\label{Updatey1}
& {w}_{i,j,k+1/2} \leftarrow w_{i,j,k+1/2}+\eta_{y, i, j, k}, \\
\label{Updatey2}
&  {u}_{i+1/2,j,k+1} \leftarrow u_{i+1/2,j,k+1}+\eta_{y, i, j, k}, \\
\label{Updatey3}
& {w}_{i+1,j,k+1/2} \leftarrow w_{i+1,j,k+1/2}-\eta_{y, i, j, k},
\\
\label{Updatey4}
& {u}_{i+1/2,j,k} \leftarrow u_{i+1/2,j,k}-\eta_{y, i, j, k},
\\
\label{Updatex1}
    &{v}_{i,j+1/2,k} \leftarrow v_{i,j+1/2,k}+\eta_{x, i, j, k},
    \\
\label{Updatex2}
  & {w}_{i,j+1,k+1/2} \leftarrow w_{i,j+1,k+1/2}+\eta_{x, i, j, k}, 
   \\
 \label{Updatex3}
   &  {v}_{i,j+1/2,k+1} \leftarrow v_{i,j+1/2,k+1}-\eta_{x, i, j, k}, 
   \\
 \label{Updatex4}
 &  {w}_{i,j,k+1/2} \leftarrow w_{i,j,k+1/2}-\eta_{x, i, j, k}. 
\end{align}
Note that the sign of each of the  perturbations 
$\eta_{x, i, j, k}$, $\eta_{y, i, j, k}$, and $\eta_{z, i, j, k}$ 
is defined by \reff{Updatex1}, \reff{Updatey1}, 
and \reff{Updatez1}, respectively. This follows from the right-hand rule for 
orientations, i.e., the grid faces used for defining these $\eta$-values are on the 
$xy$, $yz$, and $zx$ planes, and the convention of using counterclockwise 
directions for the sign of perturbation along each edge of a face; 
cf.\ Figure~\ref{fig:relaxation} (Right). 
The optimal perturbations $\eta_{y, i, j, k}$ and $\eta_{x, i, j, k}$ 
and the corresponding energy differences 
$\Delta F_{y, i, j, k}$ and $\Delta F_{x, i, j, k}$ are given by
\begin{align}
\label{etayijk}
&\eta_{y, i, j, k}= -\frac{1}{\ve_{y, i, j, k}} 
\left( \frac{w_{i,j,k+1/2}}{\ve_{i,j,k+1/2}}
+\frac{u_{i+1/2,j,k+1}}{\ve_{i+1/2,j,k+1}}
-\frac{w_{i+1,j,k+1/2}}{\ve_{i+1,j,k+1/2} } 
- \frac{u_{i+1/2,j,k}}{\ve_{i+1/2,j,k}} \right),
\\
\label{etaxijk}
   &\eta_{x, i, j, k} = -\frac{ 1}{\ve_{x, i,j,k}}
   \left( \frac{v_{i,j+1/2,k}}{\ve_{i,j+1/2,k}}
+\frac{w_{i,j+1,k+1/2}}{\ve_{i,j+1,k+1/2}}
-\frac{v_{i,j+1/2,k+1}}{\ve_{i,j+1/2,k+1}} 
-\frac{w_{i,j,k+1/2}}{\ve_{i,j,k+1/2}} \right), 
\\
\label{Fyijk}
   &\Delta F_{y, i, j, k}  = - \frac{1}{2} \ve_{y, i, j, k}h^3 \eta_{y, i,j,k}^2,
   \\
   \label{Fxijk}
   &\Delta F_{x, i, j, k}  = - \frac{1}{2} \ve_{x, i, j, k} h^3 \eta_{x, i, j, k}^2,
\end{align}
where 
\begin{align*}
&\ve_{y, i, j, k} = \frac{1}{\ve_{i,j,k+1/2}} + \frac{1}{\ve_{i+1/2,j,k+1} }
+ \frac{1}{\ve_{i+1,j,k+1/2}} +\frac{1}{\ve_{i+1/2,j,k}}, 
\\
&\ve_{x,i,j,k} = \frac{1}{\ve_{i,j+1/2,k}}
+ \frac{1}{\ve_{i,j+1,k+1/2}} 
+\frac{1}{\ve_{i,j+1/2,k+1}}
+\frac{1}{\ve_{i,j,k+1/2}}.
\end{align*}
Note that 
\[
h (\ve_{x, i, j, k}\eta_{x, i, j, k}, 
\ve_{y, i, j, k}\eta_{y, i, j, k}, \ve_{z, i, j, k} \eta_{z, i, j, k})
= - \left(\nabla_h \times \frac{D}{\ve}\right)_{i+1/2, j+1/2, k+1/2} \qquad \forall i, j, k \in \Z.
\]

We summarize these calculations in the following lemma:

\begin{lemma}
    \label{l:localalg4Poisson}
Let $\ve \in V_h$ with $\ve > 0$ on $h\Z^3,$ $\rho^h \in \mathring{V}_h$, 
and $D = (u, v, w) \in S_{\rho, h}$.
    \begin{compactenum}
        \item[{\rm (1)}] 
        Given $i, j, k \in \{ 0, \dots, N-1 \}$. Let 
        $D^x $, $D^y, $ and $D^z$ be updated from $D$ by \reff{Updatez1}--\reff{Updatex4}
 with $\eta_{x, i, j, k}$, $\eta_{y, i, j, k}$, $\eta_{z, i, j, k}$,
 $\Delta F_{x, i, j, k},$, $ \Delta F_{y, i, j, k}$, and $\Delta F_{z, i, j, k}$ 
 given in   \reff{etazijk}, \reff{Fzijk}, and \reff{etayijk}--\reff{Fxijk}, 
        respectively.  Then   $D^x, D^y, D^z \in S_{\rho, h}$, 
     $\sA_h(D^x) = \sA_h(D^y) = \sA_h(D^z) = \sA_h(D),$ and 
        \begin{align*}
\eta_{\sigma, i, j, k}^2  = \frac14 \| D^\sigma - D \|^2_h 
= -\frac{2}{\ve_{\sigma, i, j,k} h^3} \Delta F_{\sigma, i, j, k},  
\qquad \sigma \in \{ x, y, z \}. 
        \end{align*}
          \item[{\rm (2)}] 
        $D/\ve$ is curl-free, i.e.,   $\nabla_h \times (D/\ve )= 0$  on $h(\Z+1/2)^3$,
if and only if $\eta_{z, i, j, k} = \eta_{y, i, j, k} = \eta_{x, i, j, k} = 0$
        for all $i, j, k \in \{ 0, \dots, N-1\}.$ 
        \qed
    \end{compactenum}
\end{lemma}

Here is the local algorithm for a constant coefficient $\ve$. In this case, 
the expressions of all those subscripted $\eta$ and
$\Delta F$ can be simplified. 

\medskip

\noindent
{\bf Local algorithm for minimizing $F_h: S_{\rho, h} \to \R$.}


\noindent
$\qquad$ {Step 1.}
Initialize a displacement $D^{(0)} \in S_{\rho, h}$ with $\sA_h (D^{(0)}) = 0.$
 Set $m = 0.$

\noindent
$\qquad $ {Step 2.} Update $D:=D^{(m)}.$

\noindent
$\qquad\qquad \quad \ \, $
For $i, j, k = 0,  \dots, N-1$

\noindent $\qquad \qquad \qquad \quad  $ 
Update $D$ to get $D^x$ by \reff{Updatex1}--\reff{Updatex4} and $D \leftarrow D^x$, 

\noindent $\qquad \qquad \qquad \quad $ 
Update $D$ to get $D^y$ by \reff{Updatey1}--\reff{Updatey4} and $D \leftarrow D^y$, 

\noindent $\qquad \qquad \qquad \quad  $ 
Update $D$ to get $D^z$ by \reff{Updatez1}--\reff{Updatez4} and $D \leftarrow D^z$.

\noindent
$\qquad \qquad \quad \ \, $
End for

\noindent
$\qquad $
Step 3. 
If $\eta_{x, i, j, k} = \eta_{y, i, j, k} = \eta_{z, i, j, k} = 0$ for all
$i, j, k = 0, \dots, N-1$, then stop. 

$\qquad \qquad $
Otherwise, 
set $D^{(m+1)} = D$ and $m:= m+1$ and go to Step 2. 


\medskip

\begin{remark}
Suppose the local algorithm generates a sequence of displacements 
converging to some $D^{(\infty)} \in S_{\rho, h}.$ By Theorem~\ref{t:DiscreteEnergy},
$D^{(\infty)}$ is the minimizer of $F_h: S_{\rho, h} \to \R$ if and only 
$\nabla_h \times (D^{(\infty)}/\ve)  = 0$ and $\sA_h (D^{(\infty)}/\ve)=0.$ 
It is expected that $\nabla_h \times (D^{(\infty)}/\ve)  = 0$ which is equivalent to 
the vanishing of all perturbations, the subscripted $\eta$, in the update. 
Each update in the local algorithm does not change $\sA_h (D)$ but 
may likely change $\sA_h (D/\ve)$ if $\ve $ is not a constant. 
It is generally impossible to construct an initial
displacement so that at the end $\sA_h(D^{(\infty)}/\ve) = 0.$ Therefore, the above
algorithm only works for a constant $\ve$ in general.
\end{remark}

Before we present a new algorithm for a variable $\ve$, we prove the convergence of
the local algorithm for a constant dielectric coefficient. 

\begin{theorem}
\label{t:ConvP}
 Le $\ve \in V_h$ be a positive constant, 
  $\rho^h \in \mathring{V}_h$, and $ D^h_{\rm min} \in S_{\rho, h}$ 
be the unique minimizer of $F_h: S_{\rho, h} \to \R.$
Let $D^{(0)}\in S_{\rho,h}$ be such that $\sA_h(D^{(0)}) = 0$
and let $D^{(t)}\in S_{\rho, h} $ $(t = 0, 1,\dots)$ 
be the sequence  (finite or infinite) of 
    displacements generated by the local algorithm. 
    \begin{compactenum}
    \item[{\rm (1)}] If the sequence is finite ending at $D^{(m)}$, then 
    $D^{(m)}= D^h_{\rm min}$ and $F_h[D^{(m)}]  = F_h [D^h_{\rm min}].$
    \item[{\rm (2)}] If the sequence is infinite, then $D^{(t)}\to D^h_{\min}$ on 
    $h (\Z+1/2)^3$ and $F_h[D^{(t)}]\to F[D^h_{\min}]$.
    \end{compactenum}
\end{theorem}

\begin{proof}
 (1) Since $D^{(m)}$ is the terminate update, $\eta_{z, i, j, k} = 
 \eta_{y, i, j, k} = \eta_{x, i, j, k } = 0$ for all $i, j, k.$ Thus, by Lemma~\ref{l:localalg4Poisson}, 
 $D/\ve$ is curl free, and 
 $\sA_h( D^{(m)})= \sA_h (D^{(0)}) = 0$ which implies $\sA_h( D^{(m)} / \ve ) = 0$
 since $\ve $ is a constant.
 Therefore, by Theorem~\ref{t:DiscreteEnergy}, 
 $D^{(m)} = D^h_{\rm min}$ and $F_h [D^{(m)}] = F_h [D^h_{\rm min}].$

(2) Note that for each $t \in \N$, the iteration from $D^{(t)}$ to $D^{(t+1)}$ consists of 
 a cycle of $3 N^3$ local updates (with $1$  on each of the 
 $3$ faces of the grid box associated with
 each grid point and a total of $N^3$ grid points). 
Let us redefine the sequence of updates, still denoted
 $D^{(t)}$ $(t=1, 2, \dots)$, by a single-step local update, i.e., $D^{(t+1)}$ is obtained
 by updating $D^{(t)}$ on one of the $3N^3$ grid faces. 
The new $D^{(t+3N^3)}$ and $D^{(t)}$ are updates on the same grid face for each $t \ge 1$. 
 Clearly, the original sequence is a subsequence of the new one.
We prove that this new sequence
converges to $D^h_{\rm min}$, which will imply that 
the original sequence converges to $D^h_{\rm min}.$
 

By Lemma~\ref{l:localalg4Poisson},  $F_h[D^{(t)}]$ decreases as $t$ increases.
 Since $0 \leq F_h[D^{(t)}] \leq F_h[D^{(0)}]$ for all $t\ge 1,$  the limit
 $    F_{h, \infty}:=\lim_{t \to \infty} F_h [D^{(t)}]$ 
 exists and $F_{h, \infty} \ge 0.$ Denoting
\begin{equation}
    \label{detalt4Poisson}
\delta_t=F_h[D^{(t)}]-F_h[D^{(t+1)}] \geq 0 \qquad (t = 0, 1, \dots), 
\end{equation}
we have 
\[
0\le \sum_{t=0}^{\infty} \delta_t =
\lim_{T \to \infty} \sum_{t=0}^T \delta_t 
=\lim_{T \to \infty} \left( F_h [D^{(0)}]- F_h [D^{(T+1)}]\right) = F_h[D^{(0)}]- F_{h,\infty}
\leq   F_h[D^{(0)}].
\]
Hence, $ \lim_{t\to \infty} \delta_t = 0.$

To show $D^{(t)} \to  D^h_{\min}$, which implies immediately 
$F_h[D^{(t)}] \to F_h[D^h_{\min}]$, it suffices to show that the limit of any convergent 
subsequence of $\{ D^{(t)}\}_{t=1}^\infty$ 
is $D^h_{\min}$. Let $\{ D^{(t_r)} \}_{r=1}^{\infty}$ 
be such a subsequence and assume
$D^{(\infty)} =\lim _{r\to\infty} D^{(t_r)}. $ Since $D^{(t)} \in S_{\rho, h}$ 
and $\sA_h( D^{(t)} ) = 0$ for all $t \ge 1$
by Lemma~\ref{l:localalg4Poisson}, $D^{(\infty)} \in S_{\rho, h}$ and 
$\sA_h (D^{(\infty)}) = 0$.
By Theorem \ref{t:DiscreteEnergy} 
it suffices to show that $D^{(\infty)}$ is locally in equilibrium, 
i.e., $\nabla_h \times (D^{(\infty)}/\ve) = 0$
which is the same as $\nabla_h \times D^{(\infty)} = 0$ since $\ve$ is a constant. 


Since $\{ D^{(t_r)} \}_{r=1}^{\infty}$ is an infinite sequence and there are only 
finitely many grid faces, there exists a grid face with vertices, 
say, $(i+\delta_1, j+\delta_2, k)$ with $\delta_1, \delta_2 \in \{ 0, 1 \}$, 
on which $D^{(t_r)}$ is updated 
for infinitely many $r$'s. Therefore, there exists a subsequence 
of $\{ D^{(t_r)} \}_{r=1}^{\infty}$, not relabelled, such that 
for each $r\geq 1$, $D^{(t_r)}$ is updated on that same grid face.
Since $D^{(t_r)} \to D^{(\infty)}$, 
$\eta^{(t_r)}_{z,i,j,k} \to \eta^{(\infty)}_{z,i, j, k}$,
where $ \eta^{(t_r)}_{z,i, j, k}$ and $\eta^{(\infty)}_{z,i, j, k} $ are the 
$\eta_z$ values as defined in \reff{etazijk} with $D^{(t_r)}$ and $D^{(\infty)}$ 
replacing $D$, respectively. 
On the other hand, since $\delta_t \to 0$, 
Lemma~\ref{l:localalg4Poisson} implies that 
$[\eta^{(t_r)}_{z,i, j, k}]^2 \to 0$.
Hence, $\eta^{(\infty)}_{z,i, j, k}=0$.

Finally, fix any grid point $(l, m, n).$
We show $\eta^{(\infty)}_{z,l,m,n}=\eta^{(\infty)}_{y,l,m,n}=\eta^{(\infty)}_{x,l,m,n}=0$, 
where these $\eta$-values are defined as in \reff{etazijk}, \reff{etayijk}, and 
\reff{etaxijk} with $D^{(\infty)}$ and $(l, m, n)$ replacing $D$ and $(i, j, k)$, respectively. 
This will imply that $D^{(\infty)}$ is in local equilibrium, 
and complete the proof.
Note that in the local algorithm a cycle of $3N^3$ local updates 
are done for all the grid faces before next cycle starts. 
Thus, for each $r\geq 1$, there exists an integer $\tau_r$ 
such that $1\leq \tau_r \leq 3N^3$ and $D^{(t_r+\tau_r)}$ is updated, 
with the perturbation $\eta_{z, l, m, n}^{(t_r+\tau_r)}$, on the grid
face parallel to the $z$-plane of the grid box $B_{l, m, n}=(l, m, n)+[0,1]^3;$ 
cf.\ Figure~\ref{fig:relaxation} (Left).
(Since the order of grid points is fixed for local updates, 
the integer $\tau_r$ is independent of $r$.)
Since $\delta_t \to 0$, Lemma~\ref{l:localalg4Poisson} implies that 
$\| D^{(t+1)} - D^{(t)} \|_h \to 0$ as $t\to\infty$.  Thus, 
\begin{align*}
    \| D^{(t_r+\tau_r)} - D^{(t_r)}\|_h
    & \leq  \sum_{s=1}^{3 N^3} \| D^{(t_r+s)} - D^{(t_r+s-1)} \|_h
    \to 0 \qquad \mbox{as } r \to \infty.  
\end{align*}
This and the fact that $D^{(t_r)} \to D^{(\infty)}$ imply
$D^{(t_r+\tau_r)} \to D^{(\infty)}$. Consequently, by Lemma~\ref{l:localalg4Poisson},
$ \eta^{(\infty)}_{z,l,m,n}=\lim_{r\to \infty} \eta^{(t_r+\tau_r)}_{z,l,m,n} = 0.$ 
Similarly, $\eta^{(\infty)}_{x, l, m, n} = 0$ and $\eta^{(\infty)}_{y, l, m, n} = 0$.
\end{proof}

To treat the case of a variable coefficient $\ve$, we propose a new algorithm, a local 
algorithm with shift, 
by adding a step of shifting $D$ so that $\sA_h(D/\ve) = 0$. This is equivalent
to a global optimization as indicated by the following lemma whose proof 
is straightforward and thus omitted:  

\begin{lemma}
    \label{l:global} 
    Let $\ve\in V_h$ be such that $\ve > 0$, $\rho^h \in \mathring{V}_h$, 
    $D = (u, v, w) \in S_{\rho, h}$, and
    \[
    (\hat{a}, \hat{b}, \hat{c}) = - \sum_{i, j, k=0}^{N-1}
    \left(
    \frac{u_{i+1/2, j, k}/\ve_{i+1/2, j, k} }{\sum_{l, m, n=0}^{N-1}  1/\ve_{l+1/2, m, n} }, 
    \frac{v_{i, j+1/2, k}/\ve_{i, j+1/2,  k} }{\sum_{l, m, n=0}^{N-1}   1/\ve_{l, m+1/2,  n} },
    \frac{w_{i, j, k+1/2}/\ve_{i, j, k+1/2} }{\sum_{l, m, n=0}^{N-1}  1/\ve_{l, m, n+1/2} }
    \right).
    \]
Then $D + (a, b, c ) \in S_{\rho, h}$ for any $a, b, c \in \R,$ 
$(\hat{a}, \hat{b}, \hat{c})$ is the unique minimizer of 
$g(a, b, c) := F_h [D + (a, b, c)] - F_h[D]$ $(a, b, c \in \R),$ 
and the minimum  of $g: \R^3 \to \R$ is 
    \begin{align*}
    g( \hat{a}, \hat{b}, \hat{c}) 
    & = -\frac{h^3}{2} \left[ \left( \sum_{i,j,k=0}^{N-1}
    \frac{1}{\ve_{i+1/2, j, k}}\right) \hat{a}^2 + 
    \left( \sum_{i,j,k=0}^{N-1}\frac{1}{\ve_{i, j+1/2,  k}}\right) \hat{b}^2  
    + \left( \sum_{i,j,k=0}^{N-1}  \frac{1}{\ve_{i, j, k+1/2} }\right) \hat{c}^2 \right].
    \end{align*}
    Moreover, $\sA_h ( (D+(\hat{a}, \hat{b}, \hat{c}) ) / \ve) = 0.$
    \qed
\end{lemma}

In our local algorithm with shift
for minimizing the discrete Poisson energy with a variable coefficient $\ve,$ 
the initial $D^{(0)} $ is not necessary to satisfy $\sA_h (D^{(0)} )= 0.$ Moreover,
we introduce $N_{\rm local}\in \N$ to control the number of cycles of local updates
followed by one global shift.

\medskip



\noindent
{\bf A local algorithm with shift for minimizing $F_h: S_{\rho, h} \to \R$.}

\noindent
$\qquad$ {Step 1.}
Initialize a displacement $D^{(0)} \in S_{\rho, h}.$ 
 Set $m = 0.$

\noindent
$\qquad $ {Step 2.} Update locally $D:=D^{(m)}.$

\noindent
$\qquad\qquad \quad \ \, $
For $n=1, \dots N_{\rm local}$

\noindent $\qquad \qquad \qquad \quad $
For $i, j, k = 0,  \dots, N-1$

\noindent $\qquad \qquad \qquad \quad \quad  $ 
Update $D$ to get $D^x$ by \reff{Updatex1}--\reff{Updatex4} and $D \leftarrow D^x$, 

\noindent $\qquad \qquad \qquad \quad \quad $ 
Update $D$ to get $D^y$ by \reff{Updatey1}--\reff{Updatey4} and $D \leftarrow D^y$, 

\noindent $\qquad \qquad \qquad \quad \quad  $ 
Update $D$ to get $D^z$ by \reff{Updatez1}--\reff{Updatez4} and $D \leftarrow D^z$.

\noindent $\qquad \qquad \qquad \quad $
End for

\noindent
$\qquad \qquad \quad \ \, $
End for

\noindent
$\qquad $ {Step 3.} Shift $D:$
Compute $\hat{a}$, $\hat{b}, $ $\hat{c}$ and 
$D \leftarrow D + (\hat{a}, \hat{b}, \hat{c}))$.

\noindent
$\qquad $
Step 4. 
If $\eta_{x, i, j, k} = \eta_{y, i, j, k} = \eta_{z, i, j, k} = 0$ for all
$i, j, k = 0, \dots, N-1$ and $\hat{a} = \hat{b} = \hat{c} =0$, 

$\qquad \qquad $
then stop. 
Otherwise, 
set $D^{(m+1)} = D$ and $m:= m+1$. Go to Step 2. 


\medskip


\begin{theorem}
\label{t:ConvLocalGlobal}
Let $\ve \in V_h$ with $\ve > 0$, $\rho_h \in \mathring{V}_h$,  
and $ D^h_{\rm min} \in S_{\rho, h}$ 
be the unique minimizer of $F_h: S_{\rho, h} \to \R.$ 
Let $D^{(0)}\in S_{\rho,h}$ and  $D^{(t)}\in S_{\rho, h} $ 
$(t = 0, 1,  \dots)$ be the sequence  (finite or infinite)
 generated by the local algorithm with shift.
\begin{compactenum}
\item[{\rm (1)}] 
If the sequence is finite ending at $D^{(m)}$, then 
$D^{(m)}= D^h_{\rm min}$ and $F_h[D^{(m)}]  = F_h [D^h_{\rm min}].$
\item[{\rm (2)}] 
If the sequence is infinite, then $D^{(t)}\to D^h_{\min}$ on 
$h (\Z+1/2)^3$ and $F_h[D^{(t)}]\to F[D^h_{\min}]$.
\end{compactenum}
\end{theorem}

\begin{proof}
    (1) This is similar to the proof of Part (1) of the last theorem. 

    (2) For any $D = (u, v, w) \in S_{\rho, h}$, we define $\eta = \eta (D) = 
    (\eta_x, \eta_y, \eta_z)$ 
    by \eqref{etaxijk}, \eqref{etayijk}, and 
    \eqref{etazijk} at any $(i, j,k)$. We also define $G = G (D) = (\hat{a}, 
    \hat{b}, \hat{c}) \in \R^3$
  with $ \hat{a}$, $\hat{b}$, and $\hat{c}$ given in Lemma~\ref{l:global}. 
    Clearly, both $\eta(D)$ and $G(D)$ depend on $D$ linearly and hence continuously. 
    We claim that 
    \begin{equation}
        \label{DAlimit}
        \lim_{t \to \infty} \eta(D^{(t)}) = (0, 0, 0) \quad  \mbox{(at all the grid points)}
        \quad \mbox{and} \quad \lim_{t\to \infty} G (D^{(t)}) = (0, 0, 0).
    \end{equation}

Suppose \eqref{DAlimit} is true. We prove that $D^{(t)} \to D^h_{\rm min}$, which 
implies $F_h [D^{(t)}] \to F_h[D^h_{\rm min}]$.
It suffices to show the following:  assume that $D^{(t_r)}$ $(r = 1, 2, \dots)$ is a convergent
subsequence of $D^{(t)}$ $(t = 1, 2, \dots)$ and $D^{(t_r)} \to D^{(\infty)}$, 
then $D^{(\infty)} = D^h_{\rm min}.$ In fact, with such an assumption, 
$D^{(\infty)} \in S_{\rho, h}$, and 
$\eta(D^{(\infty)}) =(0, 0, 0)$ and $G(D^{(\infty)})=(0, 0, 0)$ by \reff{DAlimit}.
Hence, $\nabla_h \times (D^{(\infty)} / \ve ) = 0$ by Lemma~\ref{l:localalg4Poisson}
and $\sA_h (D^{(\infty)}/\ve) = 0$ by Lemma~\ref{l:global}.
Consequently, $D^{(\infty)} = D^h_{\rm min}$ by Theorem~\ref{t:DiscreteEnergy}.

We now proceed to prove \reff{DAlimit}.
Note that for each $t \in \N$, the iteration from $D^{(t)}$ to $D^{(t+1)}$ consists of 
 $N_{\rm local}$ cycles of local updates and one global shift. Each cycle consists
 of $3 N^3$ local updates  on $3$ grid faces associated with
 each grid point and with a total of $N^3$ grid points. 
 For convenience of proof, we redefine the sequence of updates, still denoted
 $D^{(t)}$ $(t=1, 2, \dots)$, by a single-step local or global update, 
i.e., $D^{(t+1)}$ is obtained
from $D^{(t)}$ either by a local update on one of the $3N^3$ grid faces or by
 a global update (i.e., global shift). 
The order of these local and global updates is kept the same as
 in the algorithm. Clearly, the original sequence is a subsequence of the new one. 
We shall prove \reff{DAlimit} for this new sequence.
 

By Lemma~\ref{l:localalg4Poisson} and Lemma~\ref{l:global},  $F_h[D^{(t)}]\ge 0$ 
decreases as $t$ increases. Thus, the limit 
$  F_{h, \infty}:=\lim_{t \to \infty} F_h [D^{(t)}]\ge 0$ exits. 
Denoting 
\[
\delta_t=F_h[D^{(t)}]-F_h[D^{(t+1)}] \geq 0 \qquad (t = 0, 1, \dots), 
\]
we have as before (cf.\ the proof of Theorem~\ref{t:ConvP})
$0 \le \sum_{t=1}^\infty \delta_t\le F_h [D^{(0)}]$ and hence 
\begin{equation}
  \label{deltato0LG}
   \lim_{t\to \infty} \delta_t = 0.
\end{equation}

Denote $\eta^{(t)} = (\eta^{(t)}_x, \eta^{(t)}_y, \eta^{(t)}_z) = \eta (D^{(t)}) $ 
and $G^{(t)} = G( D^{(t)}) = (\hat{a}^{(t)}, \hat{b}^{(t)}, \hat{c}^{(t)}) $ 
$(t  = 1, 2, \dots).$ We show that $\eta^{(t)}_z \to 0$ at all $i, j, k$ as $t \to \infty.$ 
Let us fix $ t \ge 1$ and also $i, j, k$. 
By \eqref{etazijk}, $\eta^{(t)}_{z, i, j, k}$ is a linear combination of 
$u^{(t)}_{i+1/2, j, k}$, $u^{(t)}_{i+1/2, j+1, k}$, $v^{(t)}_{i, j+1/2, k}$, 
and $v^{(t)}_{i+1, j+1/2, k}.$ Each of these values is obtained from some previous 
local updates or a global update. There are two cases: one is that the last update 
that determines all these values is local, and the other global. 

Consider the first case. Assume the last update that determines all 
$u^{(t)}_{i+1/2, j, k}$, $u^{(t)}_{i+1/2, j+1, k}$, $v^{(t)}_{i, j+1/2, k}$, 
and $v^{(t)}_{i+1, j+1/2, k}$ is a local update from $D^{(t'-1)}$ to $D^{(t')}$ with some 
$t'$ such that $ t'\le t < t'+3  N^3+1.$  
(This $1$ accounts for a possible global update.)
Note that some of the four $u^{(t)}$ and $v^{(t)}$-values 
might have been possibly updated
before this last update.
Assume also the perturbation associated with this last local update is 
$\eta^{(t'-1)}_{\theta, l, m, n}$ for
some $l, m, n$ with $\theta = x$ or $y$ or $z.$ 
All $l$, $m$, $n$, and $\theta$ depend on $t'$ and hence $t$, and 
$(l, m, n)$ may not be the same as $(i, j, k).$ 
By Lemma~\ref{l:localalg4Poisson},
\reff{deltato0LG}, and the fact that $t'\to \infty$ as $t \to \infty$, 
\begin{equation}
    \label{thetalmn}
    \lim_{t\to \infty} \eta^{(t'-1)}_{\theta, l, m, n} =0. 
\end{equation}
This, together with Lemma~\ref{l:localalg4Poisson} again, implies 
\begin{equation}
    \label{Dtprime}
 \| D^{(t')} - D^{(t'-1)} \|_h^2 = 4 [\eta^{(t'-1)}_{\theta, l, m, n}]^2 
\to 0 \quad \mbox{as } t \to \infty.
\end{equation}
Note that, after that last local update from $(t'-1)$ to $(t')$, all the values of 
$u^{(t)}_{i+1/2, j, k}$, $u^{(t)}_{i+1/2, j+1, k}$, $v^{(t)}_{i, j+1/2, k}$, 
and $v^{(t)}_{i+1, j+1/2, k}$ are not changed before the next update from 
$D^{(t)}$ to $D^{(t+1)}$. Thus, $u^{(t)}_{i+1/2, j, k}=u^{(t')}_{i+1/2, j, k}, $
$u^{(t)}_{i+1/2, j+1, k} = u^{(t')}_{i+1/2, j+1, k},$
$v^{(t)}_{i, j+1/2, k} = v^{(t')}_{i, j+1/2, k}, $ 
and $v^{(t)}_{i+1, j+1/2, k}=v^{(t')}_{i+1, j+1/2, k}.$
Consequently, $\eta^{(t)}_{z, i, j, k} = \eta^{(t')}_{z, i, j, k}.$
 By \reff{etazijk}, $\eta^{(t')}_{z, i, j, k}$ 
and $ \eta^{(t'-1)}_{\theta, l, m, n}$ depend linearly 
and hence continuously on $D^{(t')}$ and $D^{(t'-1)}$, 
respectively.  Hence, it follows from \reff{Dtprime} that 
$ \eta^{(t')}_{z, i, j, k} - \eta^{(t'-1)}_{\theta, l, m, n} \to 0$
as $t \to \infty. $ This and \reff{thetalmn} imply
$\eta^{(t)}_{z, i, j, k} = \eta^{(t')}_{z, i, j, k} \to 0$ as $t \to \infty. $
Similarly, $\eta^{(t)}_{x, i, j, k} \to 0$ and $\eta^{(t)}_{y, i, j, k} \to 0.$ 

Now consider the second case: the update from $D^{(t-1)}$ to $D^{(t)}$ is global, i.e., 
$D^{(t)} = D^{(t-1)} + (\hat{a}^{(t-1)}, \hat{b}^{(t-1)}, \hat{c}^{(t-1)}).$ By Lemma~\ref{l:global}
and \reff{deltato0LG}, all $\hat{a}^{(t)}$, $\hat{b}^{(t)}$, $\hat{c}^{(t)}$ converge to $0$. Therefore, 
since $\eta_{z, i, j, k} = \eta_{z, i, j, k}(D)$ depends on $D$ linearly, 
$\eta^{(t)}_{z, i, j, k}  - \eta^{(t-1)}_{z, i, j, k} \to 0$.
Note that $\eta^{(t-1)}_{z, i,j, k}$ is a linear combination of 
$u^{(t-1)}_{i+1/2, j, k}$, $u^{(t-1)}_{i+1/2, j, k}$, $v^{(t-1)}_{i, j+1/2, k}$, 
and $v^{(t-1)}_{i+1, j+1/2, k}$. 
Since the update from $D^{(t-1)}$ to $D^{(t)}$ is global, the last update
that determines  those four values of $D^{(t-1)}$ must be a local update. By case 1 above, 
we have $\eta^{(t-1)}_{z, i, j, k} \to 0$, and hence 
$\eta^{(t)}_{z, i, j, k} \to 0$.
Similarly,  $\eta^{(t)}_{x, i, j, k}\to 0$ and $\eta^{(t)}_{y, i, j, k} \to 0$.
The first limit in \reff{DAlimit} is proved. 

We now prove the second limit in \reff{DAlimit}. Let $t \ge 0.$ If 
the update from $D^{(t)}$ to $D^{(t+1)}$ is global, then $G(D^{(t)})\to (0, 0, 0)$ as
$t\to \infty$ by Lemma~\ref{l:global} and \reff{deltato0LG}. 
Suppose the update is local. Then, there exists an integer $m=m(t)$ such that
$1 \le m \le 3 N_{\rm local} N^3$, and with the notation $t_0 = t - m$, 
the update from $D^{(t_0)}$ to $D^{(t_0+1)}$ is global but all the updates 
from $D^{(t_0+n)}$ to $D^{(t_0+n+1)}$ $(n=1, \dots, m-1)$ are local. 
It follows from Lemma~\ref{l:global}, \eqref{deltato0LG}, and the fact
that $t_0\to \infty$ as $t \to \infty$ that 
\begin{equation}
    \label{globalto0}
    \| G(D^{(t_0)}) \|^2 \le C(\ve) h^{-3} \delta_{t_0} \to 0 \quad \mbox{as } t\to \infty, 
\end{equation}
where $C(\ve) > 0$ is a constant independent of $h$ and $t_0.$
By Lemma~\ref{l:localalg4Poisson}, Lemma~\ref{l:global},  and \eqref{deltato0LG}, 
$ \| D^{(t')} - D^{(t'-1)}\|_h \to 0$ as $t'\to \infty.$ 
Thus, 
$
 \| D^{(t)} - D^{(t_0)} \|_h \le \sum_{n=1}^m 
\| D^{(t_0 + n)} - D^{(t_0 + n-1)} \|_h \to 0.
$
This and \reff{globalto0}, together with the continuity of $G(D)$ on $D $ 
by Lemma~\ref{l:global}, imply that $G(D^{(t)}) \to (0, 0, 0)$. 
\end{proof}

\subsection{Minimizing the discrete Poisson--Boltzmann energy}
\label{ss:LocalAlg4PB}

Let $\ve\in V_h$ satisfy $\ve > 0$ on $h\Z^3$ and $\rho^h \in V_h$ 
satisfy \reff{DiscreteNeutrality}. The local algorithm for 
minimizing the discrete Poisson--Boltzmann (PB) energy functional 
$\hat{F}_h: X_{\rho, h} \to \R$ consists of two parts: initialization and local updates.
We initialize discrete concentrations $c^{(0)} = (c^{(0)}_1, \dots, c^{(0)}_M)$ by setting 
$c^{(0)}_{s, i, j, k} = L^{-3} N_s$ for all $i, j, k \in \Z$ and $s  = 1, \dots, M.$
Both the positivity condition \reff{hpos} and the conservation of mass \reff{hcons} are satisfied. 
We then initialize the displacement $D^{(0)}$ that satisfies the discrete Gauss' law in 
the same way as in the previous local algorithm for minimizing the discrete
Poisson energy functional, with the discrete total charge density
$\rho^h + \sum_{s=1}^M q_s c^{(0)}_s$ replacing $\rho^h $ there.
Thus $(c^{(0)}, D^{(0)})\in X_{\rho, h}.$ 

Let $(c, D) = (c_1,\dots, c_M; u, v, w) \in X_{\rho, h}$ be such that 
$c_{s, i, j, k} > 0$ for all $s \in \{ 1, \dots, M\}$ and 
let $i, j, k \in \{ 0, \dots, N-1\}.$
Fix $s $  and $(i, j, k)$.  Define $(\check{c}, \check{D})$ to 
be the same as $(c, D)$ except
\[
\check{c}_{s, i, j, k} := c_{s, i, j, k} - \zeta_s, \quad 
\check{c}_{s, i+1, j, k}:=c_{s, i+1, j, k} + \zeta_s, \quad 
\check{u}_{i+1/2, j, k} := u_{i+1/2, j, k} -hq_s\zeta_s, 
\]
and their corresponding periodic values, 
where $\zeta_s \in (-c_{s, i+1,j, k}, c_{s, i, j, k})$ is to be determined. 
One verifies that $(\check{c}, \check{D}) \in X_{\rho, h}$ and all 
the components of $\check{c}$ are still strictly positive. 
We choose $\zeta_s $ to minimize 
 the perturbed energy $\hat{F}_h[ (\check{c}, \check{D})]$, 
equivalently,  the energy change
\begin{align}
\label{DhFdeltas}
    \Delta \hat{F}_h (\zeta_s) &:= \hat{F}_h[\check{c}, \check{D}] - \hat{F}_h[c, D]
   \nonumber  \\
    & = h^3 \left[ \left({c}_{s,i,j,k}-\zeta_s \right) \log \left({c}_{s,i,j,k}-\zeta_s \right) + \left({c}_{s,i+1,j,k}+\zeta_s\right) \log \left({c}_{s,i+1,j,k}+\zeta_s\right) \right. 
    \nonumber \\ 
     &\left. \qquad -{c}_{s,i,j,k} \log {c}_{s,i,j,k} -{c}_{s,i+1,j,k} \log {c}_{s,i+1,j,k} \right]
    \nonumber  \\
     & \qquad 
     + \frac{h^3}{2} \left[ \frac{\left(u_{i+1/2,j,k}-hq_s\zeta_s\right)^2- u_{i+1/2,j,k}^2}{\varepsilon_{i+1/2,j,k}}  \right]
     \qquad \forall \zeta_s \in (-c_{s, i+1,j, k}, c_{s, i, j, k}).
\end{align}
We verify that $ (\Delta \hat{F}_h)''>0$, and hence $\Delta \hat{F}_h$ is strictly convex, in 
$(-c_{s, i+1, j, k}, c_{s, i, j, k})$. Thus, $\Delta \hat{F}_h$
attains its unique minimum at some 
$\zeta_s = \zeta_{s, i+1/2, j, k} \in (-c_{s, i+1,j, k}, c_{s, i, j, k}), $ which 
is determined by 
  $ (\Delta \hat{F}_h)'(\zeta_{s, i+1/2, j, k}) = 0$, i.e., 
\begin{align}
    \label{zetasi}
 & \log \left({c}_{s,i+1,j,k}+\zeta_{s, i+1/2, j,k}\right)  
 -\log \left({c}_{s,i,j,k}-\zeta_{s, i+1/2, j, k}\right) 
 \nonumber \\
 & \qquad 
 - \frac{h q_s}{\ve_{i+1/2,j,k}} \left(u_{i+1/2,j,k}-hq_s\zeta_{s, i+1/2, j, k}\right)=0.
\end{align}

With $\zeta:=\zeta_{s, i+1/2, j, k}$, $\alpha:= c_{s, i, j, k}$,  
$\beta := c_{s, i+1, j, k},$ $\gamma := u_{i+1/2, j, k},$
$a = h^2 q_s^2/\ve_{i+1/2, j, k}>0$, and 
$b = hq_s /\ve_{i+1/2, j, k}\in \R,$ 
\reff{zetasi} becomes $f(\alpha, \beta, \gamma, \zeta) = 0$, where 
\[
f(\alpha, \beta, \gamma, \zeta) = \log (\beta + \zeta) - \log (\alpha - \zeta) 
- b \gamma + a  \zeta, 
\]
and it is defined for  $ \alpha > 0, $ $\beta > 0,$ $-\infty < \gamma < \infty$, and
$ -\beta < \zeta < \alpha$. 
Clearly,  $f$ is a continuously differentiable function. Moreover, 
\[
\partial_{\zeta} f (\alpha, \beta, \gamma, \zeta) 
= \frac{1}{\beta+\zeta} + \frac{1}{\alpha - \zeta} + a> 0. 
\]
Since $f (\alpha, \beta, \gamma, \zeta) = 0$ has a unique solution 
$\zeta = \zeta (\alpha, \beta, \gamma)$ for $ \alpha > 0, 
$ $\beta > 0,$ and $-\infty < \gamma < \infty$,  
it follows from the Implicit Function Theorem that 
$\zeta=\zeta(\alpha, \beta, \gamma)$ depends on 
$(\alpha, \beta, \gamma)$ uniquely and continuously differentiably. 
Taking the partial derivative 
on both sides of $f(\alpha, \beta, \gamma, \zeta)=0,$ we obtain 
\[
 \partial_\alpha \zeta = \frac{\beta+\zeta}{q(\zeta)},
\qquad 
\partial_\beta \zeta = \frac{\zeta-\alpha}{q(\zeta)}, 
\qquad 
\partial_\gamma \zeta =  \frac{b(\alpha-\zeta)(\beta + \zeta)}{q(\zeta)}, 
\]
where $q(\zeta) = a (\alpha-\zeta)(\beta+\zeta)+\beta + \alpha$.
Therefore, $0 < \partial_\alpha \zeta < 1$, $-1 < \partial_\beta \zeta < 0,$ and 
$| \partial_\gamma \zeta | \le |b|/a,$ and hence 
$\zeta = \zeta(\alpha, \beta, \gamma)$ is Lipschitz-continuous for 
$ \alpha > 0, $ $\beta > 0,$ $-\infty < \gamma < \infty$, and $-\beta < \zeta < \alpha$. 

By \reff{DhFdeltas}, \reff{zetasi}, and the fact that 
$\log (1 + a ) \le a $ for any $a \in (-1, 1)$, we have 
    \begin{align*}
    \Delta \hat{F}_h (\zeta_{s, i+1/2, j, k})  
  &= h^3 \left( {c}_{s,i,j,k} \log 
  \frac{c_{s, i,j,k} - \zeta_{s, i+1/2, j, k}}{{c}_{s,i,j,k}}  + {c}_{s,i+1,j,k} \log \frac{{c}_{s,i+1,j,k}+\zeta_{s, i+1/2, j, k}}{{c}_{s,i+1,j,k}} 
    \right.
   \nonumber  \\
    & \qquad \left. -\zeta_{s, i+1/2, j, k} \log \frac{{c}_{s,i,j,k}-\zeta_{s, i+1/2, j, k}}{{c}_{s,i+1,j,k}+\zeta_{s, i+1/2, j, k}}\right) 
   \nonumber  \\
    &\qquad 
    +\frac{h^4 q_s \zeta_{s, i+1/2, j, k}}{2 \ve_{i+1/2, j, k}} 
    (hq_s \zeta_{s, i+1/2, j, k} - 2 u_{i+1/2, j, k})
   \nonumber  \\
     &=  h^3 \left[ {c}_{s,i,j,k} \log \left( 1 -  \frac{\zeta_{s, i+1/2, j, k}}{{c}_{s,i,j,k}} \right)+ {c}_{s,i+1,j,k} \log \left( 1 + \frac{\zeta_{s, i+1/2, j, k}} {{c}_{s,i+1,j,k}} 
     \right)\right]
    \nonumber  \\
    & \qquad  - \frac{h^5 q_s^2 \zeta_{s, i+1/2, j, k}^2 }{2\ve_{i+1/2,j,k}}
     \nonumber \\
    &\le - \frac{h^5 q_s^2 \zeta_{s, i+1/2, j, k}^2 }{2\ve_{i+1/2,j,k}}.
    \end{align*}
This indicates that the optimal perturbation is bounded by the related change of energy. 

To summarize, we update $c_{s, i, j, k}$, $c_{s, i+1, j, k}$, and $u_{i+1/2, j, k}$ to 
\begin{align}
\label{updatecsi}
&\check{c}_{s, i, j, k} = c_{s, i, j, k} - \zeta_{s, i+1/2, j, k} \quad
\mbox{and} \quad 
\check{c}_{s, i+1, j, k} = c_{s, i+1, j, k} + \zeta_{s, i+1/2, j, k}, 
\\
\label{updateui}
& 
\check{u}_{i+1/2, j, k} = u_{i+1/2, j, k} - h q_s \zeta_{s, i+1/2, j, k}, 
\end{align}
where $\zeta_{s, i+1/2, j, k}\in (-c_{s, i+1,j, k}, c_{s, i, j, k})$ is determined by 
\reff{zetasi}. 
Similarly, we update $c_{s, i, j, k}$, $c_{s, i, j+1, k}$, $v_{i, j+1/2, k}$, and 
$c_{s, i, j, k}$, $c_{s, i, j, k+1}$, $w_{i, j, k+1/2}$, respectively, by 
\begin{align}
    \label{updatecsj}
&\check{c}_{s, i, j, k} = c_{s, i, j, k} - \zeta_{s, i, j+1/2, k}  \quad
\mbox{and} \quad 
\check{c}_{s, i, j+1, k} = c_{s, i, j+1, k} + \zeta_{s, i, j+1/2, k}, 
\\
\label{updatevj}
&
\check{v}_{i, j+1/2, k} = v_{i, j+1/2, k} - h q_s \zeta_{s, i, j+1/2, k}, 
\\
\label{updatecsk}
&\check{c}_{s, i, j, k} = c_{s, i, j, k} - \zeta_{s, i, j, k+1/2} \quad 
\mbox{and} \quad 
\check{c}_{s, i, j, k+1} = c_{s, i, j, k+1} + \zeta_{s, i, j, k+1/2}, 
\\
\label{updatewk}
& \check{w}_{i, j, k+1/2} = w_{i, j, k+1/2} - h q_s \zeta_{s, i, j, k+1/2},
\end{align}
where $\zeta_{s, i, j+1/2, k} \in (-c_{s, i, j+1, k}, c_{s, i, j, k}) $ and 
$\zeta_{s, i, j, k+1/2}\in (-c_{s, i, j, k+1}, c_{s, i, j, k})$ are uniquely determined,
respectively, by
\begin{align}
    \label{zetasj}
 & \log \left({c}_{s,i,j+1,k}+\zeta_{s, i, j+1/2, ,k}\right)  
 -\log \left({c}_{s,i,j,k}-\zeta_{s, i, j+1/2, k}\right) 
 \nonumber \\
 & \qquad 
 - \frac{h q_s}{\ve_{i, j+1/2,k}} \left(v_{i,j+1/2,k}-hq_s\zeta_{s, i, j+1/2, k}\right)=0;
\\
    \label{zetask}
 & \log \left({c}_{s,i,j,k+1}+\zeta_{s, i, j, k+1/2}\right)  
 -\log \left({c}_{s,i,j,k}-\zeta_{s, i, j, k+1/2}\right) 
 \nonumber \\
 & \qquad 
 - \frac{h q_s}{\ve_{i, j, k+1/2}} \left(w_{i, j, k+1/2}-hq_s\zeta_{s, i, j, k+1/2}\right)=0.
\end{align}
    We solve \eqref{zetasi}, \eqref{zetasj}, and \eqref{zetask} using Newton's iteration with
a few steps. Note that $\zeta_{s, i+1/2, j, k} = \zeta_{s, i, j+1/2, k} = \zeta_{s, i, j, k+1/2}
        = 0$ for all $s, i, j, k$ is equivalent to the local equilibrium condition
        \reff{logloglog} in Theorem~\ref{t:DiscretePBEnergy}.

We summarize some of the properties of these local updates in the following: 

\begin{lemma}
    \label{l:localalg4PB} 
    Let $\ve \in V_h$ be such that $\ve > 0$ on $h\Z^3$
and let $\rho^h \in V_h$ satisfy \reff{DiscreteNeutrality}.  Let
    $(c, D) = (c_1, \dots, c_M, u, v, w)\in X_{\rho, h}$ satisfy $c_s > 0$ on 
    $h\Z^3$ for all $s = 1, \dots, M. $
    \begin{compactenum}
        \item[{\rm (1)}] 
        Let $0 \le i, j, k \le N-1$ and $ 1 \le s \le M$. 
        Update $(c, D)$ to $(\check{c}, \check{D}) \in  X_{\rho, h}$ by 
        \eqref{updatecsi}--\eqref{updatewk} with 
        $\zeta_{s, i+1/2, j, k}$, $\zeta_{s, i, j+1/2, k}$, 
            and $\zeta_{s, i, j, k+1/2}$ given 
            in \eqref{zetasi}, \eqref{zetasj}, and \eqref{zetask}, respectively. 
        \begin{compactenum}
        \item[{\rm (i)}]
        Each update keeps the components of $c$ to be still positive at all the grid points.
\item[{\rm (ii)}] The perturbations
$\zeta_{s, i+1/2, j, k} \in (-c_{s, i+1, j, k}, 
c_{s, i, j, k}), $ $ \zeta_{s, i, j+1/2, k} \in (-c_{s, i, j+1, k}, c_{s, i, j, k}), $
and $\zeta_{s, i, j, k+1/2} \in (-c_{s, i,j, k+1}, c_{s,i, j,k})$ 
are Lipschitz-continuous functions of
$c_{s, i, j, k},$ $ c_{s, i+1, j, k},$ and $  u_{i+1/2, j, k}$; 
$c_{s, i, j, k},$ $  c_{s, i, j+1, k}, $ and $  v_{i, j+1/2, k}$; 
and  $c_{s, i,j, k},$ $ c_{s,i, j,k+1},$ and $  w_{i, j, k+1/2}$, respectively.
            \item[{\rm (iii)}]
 The energy change $\Delta \hat{F}_h(\zeta) = 
\hat{F}_h[\check{c}, \check{D}]-\hat{F}_h[c, D]$ 
            associated with the three updates from $(c, D)$ to $(\check{c}, \check{D})$ for 
            given $s, i, j, k$ satisfy
            \[
       | \Delta \hat{F}_h (\zeta_{s,\sigma})|  
    \ge \frac{h^5 q_s^2 \zeta_{s, \sigma}^2 }{2\ve_{\sigma}}
    \quad \forall \sigma \in \{ (i+1/2, j, k), (i, j+1/2, k), (i,j, k+1/2)\}.
            \]
        \end{compactenum}
        \item[{\rm (2)}]
        The updates of $(c, D)$ at all the grid points do not further decrease the energy, 
        i.e., $\zeta_{s, i+1/2, j, k} = \zeta_{s, i, j+1/2, k} = \zeta_{s, i, j, k+1/2}
        = 0$ for all $s, i, j, k$,
        if and only if the local equilibrium conditions \reff{logloglog} are satisfied. 
        \qed
    \end{compactenum}
\end{lemma}

\medskip


\noindent
{\bf Local algorithm for minimizing $\hat{F}_h: X_{\rho, h} \to \R$}


\noindent
$\qquad$ {Step 1.}
Initialize $(c^{(0)}, D^{(0)}) \in X_{\rho, h}$ and set $m = 0.$ 

\noindent
$\qquad $ 
{Step 2.} Update $(c,D) = (c_1, \dots, c_M; u, v, w):=(c^{(m)},D^{(m)}).$

\noindent
$\qquad\qquad \quad\ \,    $
For $i, j, k = 0, \dots, N-1$

\noindent 
$\qquad \qquad \qquad \quad $
For $s= 1, \dots, M$

\noindent 
$\qquad \qquad \qquad \qquad\quad  $
Update $c_{s, i, j, k}$, $c_{s, i+1, j, k}$, and  $u_{i+1/2, j, k}$.

\noindent 
$\qquad \qquad \qquad \quad $
End for

\noindent 
$\qquad \qquad \qquad \quad $
For $s= 1, \dots, M$

\noindent 
$\qquad \qquad \qquad \qquad \quad $
Update $c_{s, i, j, k}$, $c_{s, i, j+1, k}$, and  $v_{i, j+1/2, k}$.

\noindent 
$\qquad \qquad \qquad \quad $
End for

\noindent 
$\qquad \qquad \qquad \quad $
For $s= 1, \dots, M$

\noindent 
$\qquad \qquad \qquad \qquad \quad $ 
Update $c_{s, i, j, k}$, $c_{s, i, j, k+1}$, and  $w_{i, j, k+1/2}$.

\noindent 
$\qquad \qquad \qquad \quad $
End for

\noindent
$\qquad \qquad \quad \ \, $
End for

\noindent
$\qquad \qquad \quad \ \, $
Set $ D^{(m+1)} = D$.

\noindent
$\qquad $ {Step 3.}
If the updates of $(c, D)$ at all the grid points do not
further decrease the energy, 

\noindent
$\qquad \qquad\quad \ \,  $
then stop.
Otherwise, set  $m: = m+1$ and go to Step 2. 

\medskip

    In practice, to speed up the convergence, one can add in Step 2 
    the local updates of
the displacement $D$ as in the local algorithm for minimizing the 
discrete Poisson energy (cf.\ section~\ref{ss:LcoalAlg4Poisson}). For instance, 
we can add the following at the end of the loop over $i, j, k = 0$ to 
$N-1$ in Step 2: 

\noindent $\qquad \qquad \qquad \quad  $ 
Update $D$ to get $D^x$ by \reff{Updatex1}--\reff{Updatex4} and $D \leftarrow D^x$, 

\noindent $\qquad \qquad \qquad \quad  $ 
Update $D$ to get $D^y$ by \reff{Updatey1}--\reff{Updatey4} and $D \leftarrow D^y$, 

\noindent $\qquad \qquad \qquad \quad   $ 
Update $D$ to get $D^z$ by \reff{Updatez1}--\reff{Updatez4} and $D \leftarrow D^z$.

\noindent
Note that adding updates of the displacement 
does not change the concentration and also keeps the 
discrete Gauss' law satisfied, and hence produces $(c, D) \in X_{\rho, h}.$ 



\begin{theorem}
\label{t:ConvPB}
  Let $ \ve \in V_h$ be such that $\ve > 0$ on $h\Z^3$
and $\rho^h\in V_h$ satisfy \reff{DiscreteNeutrality}. 
Let $(c^{(0)},D^{(0)})\in X_{\rho, h}$ with $c_s^{(0)}>0$ on $h\Z^3$ for all $s \in 
\{ 1, \dots, M \}$ and let $(c^{(t)},D^{(t)}) \in X_{\rho, h}$ $(t = 0, 1,  \dots)$ 
be the sequence (finite or infinite) 
generated by the local algorithm. Let
$(\hat{c}^h_{\rm min}, \hat{D}^h_{\rm min})\in X_{\rho, h}$ be the unique
minimizer of $\hat{F}_h: X_{\rho, h} \to \R.$
\begin{compactenum}
    \item[{\rm (1)}] If the sequence 
    $(c^{(t)},D^{(t)})$ $(t = 0, 1, \dots)$ is finite and the last one is 
    $(c^{(m)}, D^{(m)})$, then $(c^{(m)}, D^{(m)}) = (\hat{c}^h_{\rm min}, \hat{D}^h_{\rm min})$. 
    \item[{\rm (2)}] If the sequence $(c^{(t)},D^{(t)})$ $(t = 0, 1,  \dots)$ is infinite, then 
    \[
    \lim_{t\to \infty} (c^{(t)},D^{(t)}) = (\hat{c}^h_{\rm min},\hat{D}^h_{\rm min})
    \quad \mbox{and} \quad \lim_{t \to \infty}  
    \hat{F}_h [c^{(t)}, D^{(t)}] = \hat{F}_h [\hat{c}^h_{\rm min}, \hat{D}^h_{\rm min}]. 
    \]
\end{compactenum}
\end{theorem}

\begin{proof}
(1) This follows from  Lemma~\ref{l:localalg4PB} (Part (i) of (1) and (2)) and 
Theorem~\ref{t:DiscretePBEnergy}.

   
    (2)  We note that for each $t \ge 1$ the update from $(c^{(t)}, D^{(t)})$
to $(c^{(t+1)}, D^{(t+1)})$ consists of $ 3 M N^3$ local updates
(with a total $N^3$ grid points,  $3$ updates along the three edges for each grid, and $s = 1, \dots, M$).  
 For convenience, we redefine the sequence of iterates, still denoted
 $(c^{(t)}, D^{(t)})$ $(t=1, 2, \dots)$, by the sequence of single-step local update, i.e., 
 for each $t \ge 1$, 
 $(c^{(t+1)}, D^{(t+1)})$ is obtained
 from $(c^{(t)}, D^{(t)})$ by one of the $3 M$ updates associated to $M$ components 
 of $c^{(t)}$ and the three
 edges connected to one of the $N^3$ grid points. We keep the order of all these updates
 as in the local algorithm.  Note from the local algorithm that the new 
 $(c^{(t+ 3MN^3)}, D^{(t+3MN^3)})$ and 
 $D^{(t)}$ are updates on the same component of the concentration and the same edge of 
grid points. 
Clearly, the original sequence is a subsequence of the new one.
We shall prove the desired convergence for this new  sequence.
 This implies the convergence of the original sequence.
 

Since $\sigma \mapsto \sigma\log \sigma $ $(\sigma \ge 0)$ is bounded below, the discrete
    energy functional $\hat{F}_h: X_{\rho, h} \to \R $ is bounded below. 
    Since each update in the local algorithm decreases the energy, 
    the sequence $\hat{F}_h[c^{(t)},D^{(t)}]$ $(t = 0, 1, \dots)$ decreases monotonically and is 
    bounded below. Thus, $\hat{F}_{h, \infty}:=
    \lim_{t\to \infty} \hat{F}_h[c^{(t)}, D^{(t)}] \in \R$ exists. Denoting 
    \begin{equation}
        \label{definedeltat}
    \delta_t= \hat{F}_h[c^{(t)},D^{(t)}]-\hat{F}_h[c^{(t+1)},D^{(t+1)}],\qquad t=0, 1,\dots, 
    \end{equation}
    we have all $\delta_t \ge 0$ and 
    $
    0\le \sum_{t=0}^{\infty} \delta_t \leq \hat{F}_h[c^{(0)},D^{(0)}] -\hat{F}_{h, \infty} < \infty.
    $
    In particular, 
    \begin{equation}
        \label{drto0}
        \lim_{t\to \infty} \delta_t = 0.
    \end{equation}

Let us denote $(c^{(t)}, D^{(t)}) = (c^{(t)}_1, \dots, c^{(t)}_M; u^{(t)}, v^{(t)}, w^{(t)})$
$ (t = 0, 1,  \dots).$ For any $s, i, j, k \in \Z$ $(1 \le s \le M$ and $ 0 \le i, j, k \le N-1)$ 
 and any $t \ge 0$, we define $\zeta^{(t)}_{s, i+1/2, j, k}$
to be the unique solution to \eqref{zetasi} with $c^{(t)}_{s, i, j, k}$, 
$c^{(t)}_{s, i+1, j, k}$, and $u^{(t)}_{s, i+1/2, j, k}$ replacing 
those without the superscript $(t)$. Similarly, we define 
$\zeta^{(t)}_{s, i, j+1/2, k}$ and $\zeta^{(t)}_{s, i, j, k+1/2}$; 
cf.\ \eqref{zetasj} and \eqref{zetask}.
We claim that 
\begin{equation}
    \label{key}
\zeta^{(t)}_{s, i+1/2, j, k} \to 0, \quad 
\zeta^{(t)}_{s, i, j+1/2, k}\to 0, \quad 
\mbox{and} \quad 
\zeta^{(t)}_{s, i, j, k+1/2}\to 0 
\quad \mbox{as } t \to \infty.
\end{equation}
We shall prove the first convergence as the other two are similar. 

Fix $t,$ $s,$  $i,$ $j,$ and $ k $. The values of $c^{(t)}_{s, i, j, k}$, 
$c^{(t)}_{s, i+1, j, k}$, and $u^{(t)}_{s, i+1/2, j, k}$, which are the only components
of $c^{(t)}$ and $D^{(t)}$ used to define $ \zeta^{(t)}_{s, i+1/2, j, k} $ 
(cf.\ 
\eqref{zetasi}--\reff{updateui}),  are possibly obtained by 
several local updates (instead of just one single update) 
at grid points nearby and including $(i, j, k)$. 
Assume that the last local update that determines all 
$c^{(t)}_{s, i, j, k}$, $c^{(t)}_{s, i+1, j, k}$, and $u^{(t)}_{i+1/2, j, k}$ is 
from $({c}^{(t'-1)}, {D}^{(t'-1)})$ to $(c^{(t')}, D^{(t')} ),$ where 
$t'\le t < t'+3 M N^3.$ This means that 
$c^{(t)}_{s, i, j, k} = c^{(t')}_{s, i, j, k}$, 
$c^{(t)}_{s, i+1, j, k} = c^{(t')}_{s, i+1, j, k}$, 
and $u^{(t)}_{i+1/2, j, k} = u^{(t')}_{i+1/2, j, k},$ and 
hence $\zeta^{(t')}_{s, i+1/2, j, k} = \zeta^{(t)}_{s, i+1/2, j, k}. $
The update is given by 
\begin{equation*}
c^{(t')}_{s, i,j, k} = {c}^{(t'-1)}_{s, i,j, k} + {\delta}_i^{(t'-1)},\quad
 c^{(t')}_{s, i+1, j, k} = {c}^{(t'-1)}_{s, i+1, j, k}+ {\delta}_{i+1}^{(t'-1)}, \quad 
u^{(t')}_{s, i+1/2, j, k} =  {u}^{(t'-1)}_{s, i+1/2, j, k}+
{\delta}_{i+1/2}^{(t'-1)}.
\end{equation*}
Some of these perturbations ${\delta}_i^{(t'-1)},$ 
$ {\delta}_{i+1}^{(t'-1)}$, and ${\delta}_{i+1/2}^{(t'-1)}$ 
maybe $0$ but at least one of them is nonzero. Assume that this last local
update is associated with an edge connecting
some grid points $(l, m, n)$ and  $(l+1, m, n)$ or $(l, m+1, n)$ or
$(l, m,n+1)$ and with the species $s'$ that may be different from $s.$ If we denote the corresponding
optimal perturbation by ${\zeta}^{(t'-1)}_{s', l, m, n}$ 
(cf.\ \reff{zetasi}, \reff{zetasj}, and \reff{zetask}), then we can write
\begin{equation*}
{\delta}_i^{(t'-1)} = \sigma_i {\zeta}^{(t'-1)}_{s', l, m, n},  \quad
{\delta}_{i+1}^{(t'-1)} = \sigma_{i+1} {\zeta}^{(t'-1)}_{s', l, m, n}, \quad  
{\delta}_{i+1/2}^{(t'-1)} = - \sigma_{i+1/2} h q_{s'} {\zeta}^{(t'-1)}_{s', l, m, n},
\end{equation*}
where $\sigma_i, \sigma_{i+1}, \sigma_{i+1/2} \in \{ 0, 1, -1 \}$
and at least one of them is nonzero. 
By Lemma~\ref{l:localalg4PB} (Part (iii) of (1)), $({\zeta}^{(t'-1)}_{s', l, m,n})^2$
is bounded by the energy change resulting from this local update.
Consequently, it follows from
\reff{definedeltat}, \reff{drto0}, and the fact that $t'\to \infty$ if $t\to \infty$ that 
\begin{equation}
\label{tilde3}
\lim_{t\to \infty} {\zeta}^{(t'-1)}_{s', l, m,n}  =0.
\end{equation}
Therefore, by the formulas of local update (cf.\ \reff{updatecsi} and \reff{updateui}), 
\begin{equation}
    \label{tilde4}
\lim_{t\to \infty} \left[ (c^{(t')}, D^{(t')}) - ({c}^{(t'-1)}, {D}^{(t'-1)}) \right] = 0.
\end{equation}
By Lemma~\ref{l:localalg4PB} (Part (ii) of (1)), 
${\zeta}^{(t'-1)}_{s', l, m, n}$ and $\zeta^{(t')}_{s, i+1/2, j, k}$
depend respectively on $({c}^{(t'-1)}, {D}^{(t'-1)})$ and 
$(c^{(t')}, D^{(t')})$ Lipschitz-continuously. 
Therefore, it follows from \reff{tilde4} that 
$\zeta^{(t')}_{s, i+1/2, j, k} - {\zeta}^{(t'-1)}_{s', l, m, n}
\to 0$ as $t \to \infty.$ Consequently, by \reff{tilde3} again, 
$\zeta^{(t)}_{s, i+1/2, j, k} = \zeta^{(t')}_{s, i+1/2, j, k}  \to 0$ as $t \to \infty.$

We now prove 
   $  (c^{(t)},D^{(t)}) \to (\hat{c}^h_{\rm min},\hat{D}^h_{\rm min})$ 
   which implies  $ \hat{F}_h [c^{(t)}, D^{(t)}] \to 
   \hat{F}_h [\hat{c}^h_{\rm min}, \hat{D}^h_{\rm min}].$
Assume that 
    \begin{equation}
    \label{trinfty}
    \lim_{r\to \infty} (c^{(t_r)}, D^{(t_r)}) = (c^{(\infty)},D^{(\infty)})
    \end{equation}
    for a convergent subsequence $\{ (c^{(t_r)}, D^{(t_r)})\}_{r=1}^\infty$
    of $\{ (c^{(t)},D^{(t)})\}_{t=1}^\infty$ and some discrete and vector-valued
    functions $c^{(\infty)}$ and $D^{(\infty)}. $
    We show that  
    $(c^{(\infty)},D^{(\infty)}) =(\hat{c}^h_{\rm min}, \hat{D}^h_{\rm min}).$ 
    This will complete the proof.  
    Since clearly $(c^{(\infty)},D^{(\infty)}) \in X_{\rho, h}$, by 
    Theorem~\ref{t:DiscretePBEnergy}, we need only to show that 
    $c^{(\infty)}_{s, i, j, k} > 0$ for all $s, i, j, k$ and 
$(c^{(\infty)}, D^{(\infty)})$ is in local equilibrium, i.e., it satisfies
    \eqref{logloglog}.
    
If there exists $s \in \{ 1, \dots, M \}$ such that $c^{(\infty)}_s=0$ 
at some grid point, then by \reff{hcons} and the nonnegativity
of $c^{(\infty)}_s$, we may assume without loss of generality that 
$\alpha_\infty:=c^{(\infty)}_{s,l,m,n}>0$ but $c^{(\infty)}_{s,l+1,m,n}=0$ for some $(l, m, n).$ 
Let $c^{(\infty)} = (c^{(\infty)}_1, \dots, c^{(\infty)}_M)$ and $D^{(\infty)} = 
(u^{(\infty)}, v^{(\infty)}, w^{(\infty)})$. 
It follows from \reff{trinfty} that as $r\to \infty$, 
\[
\alpha_r:=c^{(t_r)}_{s, l, m, n} \to \alpha_\infty > 0, \quad 
\beta_r:=c^{(t_r)}_{s, l+1, m, n} \to 0, \quad 
\gamma_r: = u^{(t_r)}_{s, l+1/2, m, n} \to \gamma_\infty := u^{(\infty)}_{s, l+1/2, m, n}.
\]
By \reff{key}, $\zeta_r := \zeta^{(t_r)}_{s, l+1/2, m, n} \to 0$. 
On the other hand, by \eqref{zetasi}, $\zeta_r $ is uniquely determined by 
\[
\log (\beta_r + \zeta_r) - \log(\alpha_r - \zeta_r) +a \zeta_r - b\gamma_r = 0,
\]
where $a = h^2 q_s^2 / \ve_{l+1/2, m, n}$ and $ b = hq_s / \ve_{l+1/2, m, n}$
are independent of $r.$ As $r\to \infty$, the left-hand side of this equation 
diverges to $-\infty$, while the right-hand side remains $0$. This is a contradiction. 
Thus $c^{(\infty)}_{s, i,j, k} > 0$ for all $s, i, j, k.$

Fix $s, i, j, k$ and define $\zeta^{(\infty)}_{s, i+1/2, j, k}$ by \reff{zetasi}
with $c^{(\infty)}_{s, i, j, k}$, $c^{(\infty)}_{s, i+1, j, k}$, and 
$u^{(\infty)}_{i+1/2, j, k}$ replacing
$c_{s, i, j, k}$, $c_{s, i+1, j, k}$, and $u_{i+1/2, j, k},$ respectively. 
Then, by Part (ii) of (1) of Lemma~\ref{l:localalg4PB} and \eqref{trinfty}, 
 $ \zeta^{(t_r)}_{s, i+1/2, j, k} \to \zeta^{(\infty)}_{s, i+1/2, j, k}$ as $r \to \infty.$
But $ \zeta^{(t_r)}_{s, i+1/2, j, k} \to 0 $ by \reff{key}. 
Hence $\zeta^{(\infty)}_{s, i+1/2, j, k} = 0.$
Similarly, $\zeta^{(\infty)}_{s, i, j+1/2, k} = \zeta^{(\infty)}_{s, i, j, k+1/2}=0$. 
Since $s, i, j, k$ can be arbitrary, 
Part (2) of Lemma~\ref{l:localalg4PB} implies that 
$(c^{(\infty)}, D^{(\infty)})$ is in local equilibrium.
\end{proof}

\section{Numerical Tests}
\label{s:NumericalTests}


In this section, we conduct three
numerical tests to show the finite-difference approximation errors and 
demonstrate the convergence of the local algorithms. 
The computational box in all these tests is $[0,2]^3$ (i.e., $L = 2$).

\medskip

\noindent
{\it Test 1. The Poisson energy with a constant permittivity.}
We set 
\[
\ve = 1, \quad  
\phi (x_1, x_2, x_3) = - \cos (\pi x_1) \cos (\pi x_3) \cos (\pi x_3), 
\quad \mbox{and} \quad \rho = - \Delta \phi.
\]
Then $\phi \in \mathring{H}^1_{\rm per}(\Omega) $ is the unique solution
to Poisson's equation  $\Delta \phi = - \rho$
with the $[0, 2]^3$-periodic boundary condition, and $D := - \nabla \phi$ is the unique
minimizer of the Poisson energy functional $F: S_\rho \to \R.$ 
For a finite-difference grid with grid size $h = L/N$ for some $N \in \N$, we denote 
by $D_h \in S_{\rho, h}$ the finite-difference displacement that minimizes the
discrete energy $F_h: S_{\rho, h} \to \R$. We also denote by $D^{(k)}_h$ $(k = 0, 1, \dots)$
the iterates produced by the local algorithm. 
Figure~\ref{f:ConRs1} plots the discrete energy $F_h [D_h^{(k)}]$, $L^2$-error 
$\| \sP_h D - D_h^{(k)}\|_h$, and $L^\infty$-error 
$ \| \sP_h D - D_h^{(k)} \|_\infty$  vs.\ the
iteration step $k$ of local update with the grid size $h = L/N = 2/160 = 0.0125$.
We observe a fast decrease of the energy at the
beginning of iteration and then slow decrease of the energy afterwards. 
The errors converge to some values that are set by the grid size $h.$ 
In Figure~\ref{f:DConOrder1}, we plot in the log-log scale 
the $L^2$ and $L^\infty$-errors for the approximation $D_h$ of the exact minimizer $D$ and also
for the approximation $E_h:=m_h[D_h]/\ve$ of the electric field $-\nabla \phi$, respectively, against the 
finite-difference grid size $h.$ We observe the $O(h^2)$ convergence rates as predicted
by Theorem~\ref{t:L2Poisson} and Corollary~\ref{c:D4GradphiP}.

\vspace{-2 mm}

\begin{figure}[h] 
\begin{center}
  \includegraphics[scale=0.37]{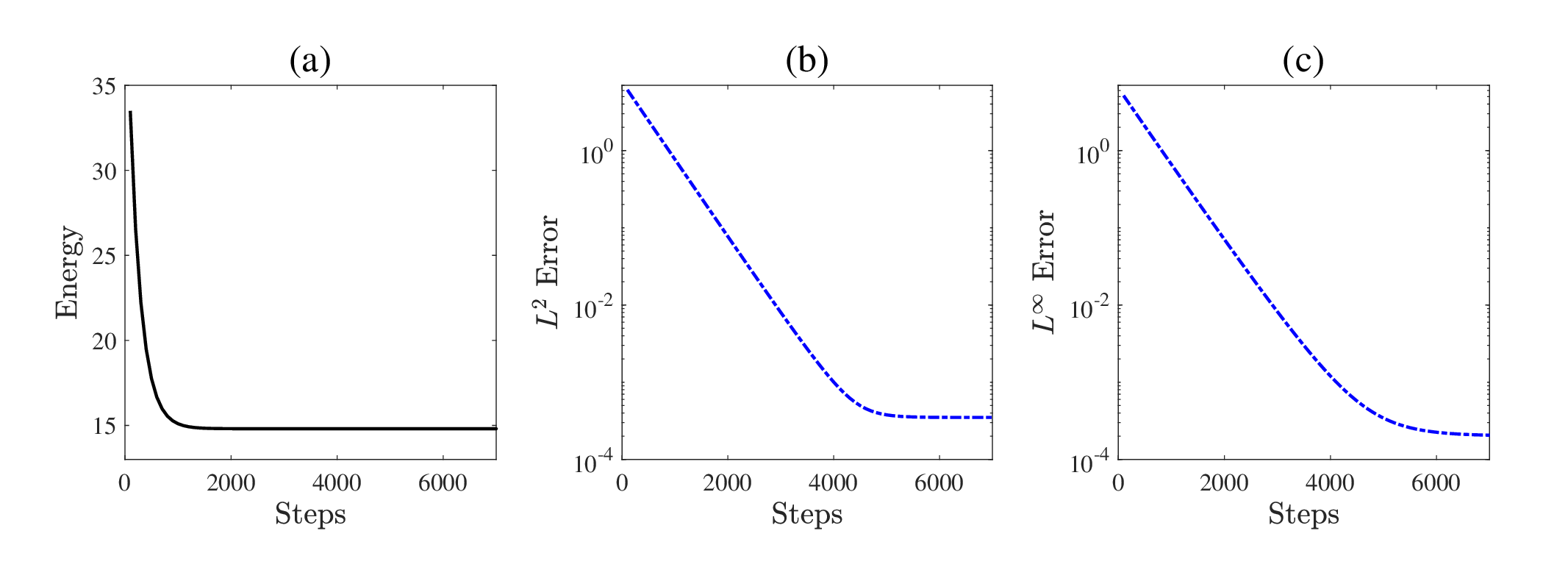}	
  \end{center}
  \vspace{-2em}  
  	\caption{
   {\small {The discrete energy~(a), $L^2$-error~(b), and $L^{\infty}$-error~(c)   
for the displacement $D^{(k)}_h$ vs.\ the 
iteration step $k$ in the local algorithm for Test 1.}}}
 \label{f:ConRs1}
\end{figure}

\vspace{-8 mm}

\begin{figure}[h]
 \centering
  	\includegraphics[width=3.2in,height=2in]{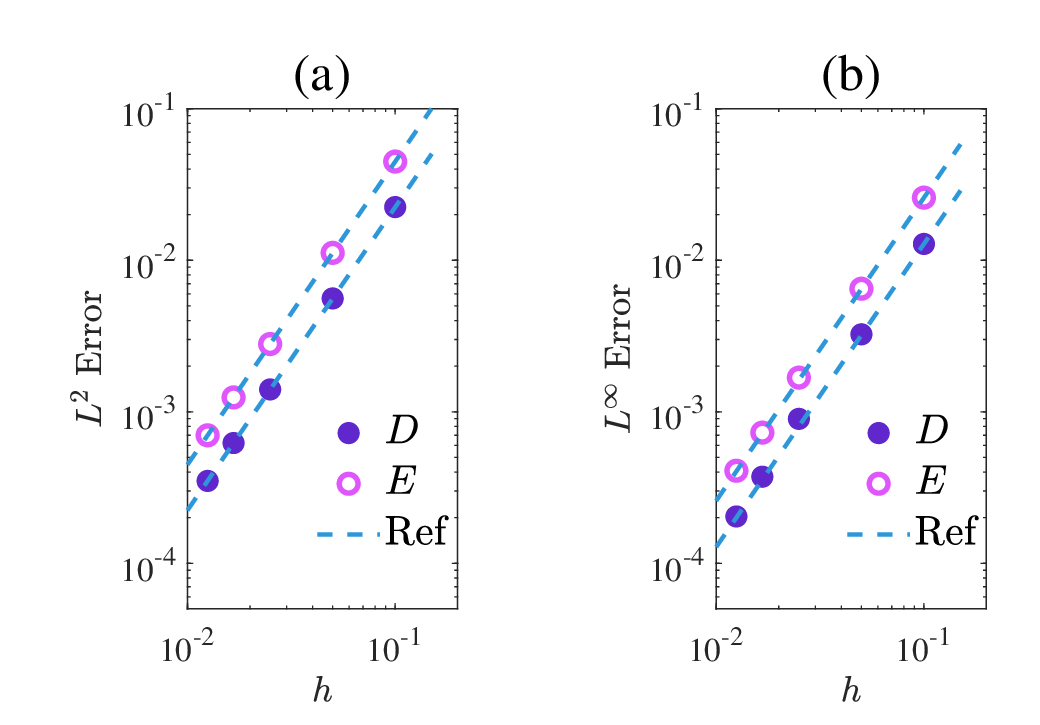}
  	\caption{{\small {Log-log plots of the $L^2$-error (a) 
and the $L^{\infty}$-error (b) for the approximation $D_h$
   of the displacement $D$ (indicated by $D$) and the 
reconstructed approximation $E_h:=m_h[D_h]/\ve$ of the
   electric field $E := -\nabla \phi$ (indicated by $E$) for Test 1. 
   The blue dashed lines are reference lines indicating
   the $O(h^2)$ convergence rate.}}}
   \label{f:DConOrder1} 
\end{figure}


\noindent
{\it Test 2. The Poisson energy with a variable permittivity.}
We set 
\begin{align*}
&\ve (x_1, x_2, x_3) = 3 - \cos (\pi x_1), 
\\
& \phi(x_1, x_2, x_3) =f(x_1) \cos (\pi x_2) \cos (\pi x_3),
\\
& f(x_1) = \left\{ 
\begin{aligned}
&e^{\frac{1}{  (x_1-1)^2-0.5^2} } & & \mbox{if } |x_1-1| < 0.5,
\\
&0  & & \mbox{if } 0 \le x_1 \le 0.5 \mbox{ or } 1.5 \le x_1 \le 2, 
\end{aligned}
\right.
\end{align*}
 first for $(x_1, x_2, x_3)\in [0, 2]^3$ and then extend them $[0, 2]^3$-periodically 
to $\R^3,$
Note that $f$ is a $C^{\infty}$-function. We then define  $ \rho = - \nabla \cdot \ve \nabla \phi$
and $D = -\ve \nabla \phi.$ So, $\phi$ is the periodic solution to Poisson's equation
$\nabla \cdot \ve \nabla \phi = -\rho$ and $D \in S_{\rho}$ is the minimizer of 
$F: S_\rho \to \R.$
As in Test 1, for a finite-difference grid with grid size $h = L/N$ for some $N \in \N$, we denote 
by $D_h \in S_{\rho, h}$ the finite-difference displacement that minimizes the
discrete energy $F_h: S_{\rho, h} \to \R$. We also denote by $D^{(k)}_h$ $(k = 0, 1, \dots)$
the iterates produced by the local algorithm with shift. 
Figure~\ref{f:ConRs2} plots the discrete energy $F_h [D_h^{(k)}]$, 
$L^2$-error $\| \sP_h D - D_h^{(k)}\|_h$,
and $L^\infty$-error $ \| \sP_h D - D_h^{(k)} \|_\infty$  vs.\ the iteration step
$k$ of local update with the grid size $h = L/N = 2/160 = 0.0125$.
We again observe a fast decrease of the energy at the
beginning of iteration and then slow decrease of the energy afterwards. The errors converge to 
some values that are set by the grid size $h.$ In Figure~\ref{f:DConOrder2}, we plot in the 
log-log scale the $L^2$ and $L^\infty$ errors for the 
approximation $D_h$ of the exact minimizer $D$ and also
for the approximation $E_h:=m_h[D_h]/\ve$ of the electric field $-\nabla \phi$, respectively, against the 
finite-difference grid size $h.$ We observe the $O(h^2)$ convergence rate as predicted
by Theorem~\ref{t:L2Poisson} and Corollary~\ref{c:D4GradphiP}.

\vspace{-2 mm}

\begin{figure}[h]
  	\centering
  \includegraphics[scale=0.37]{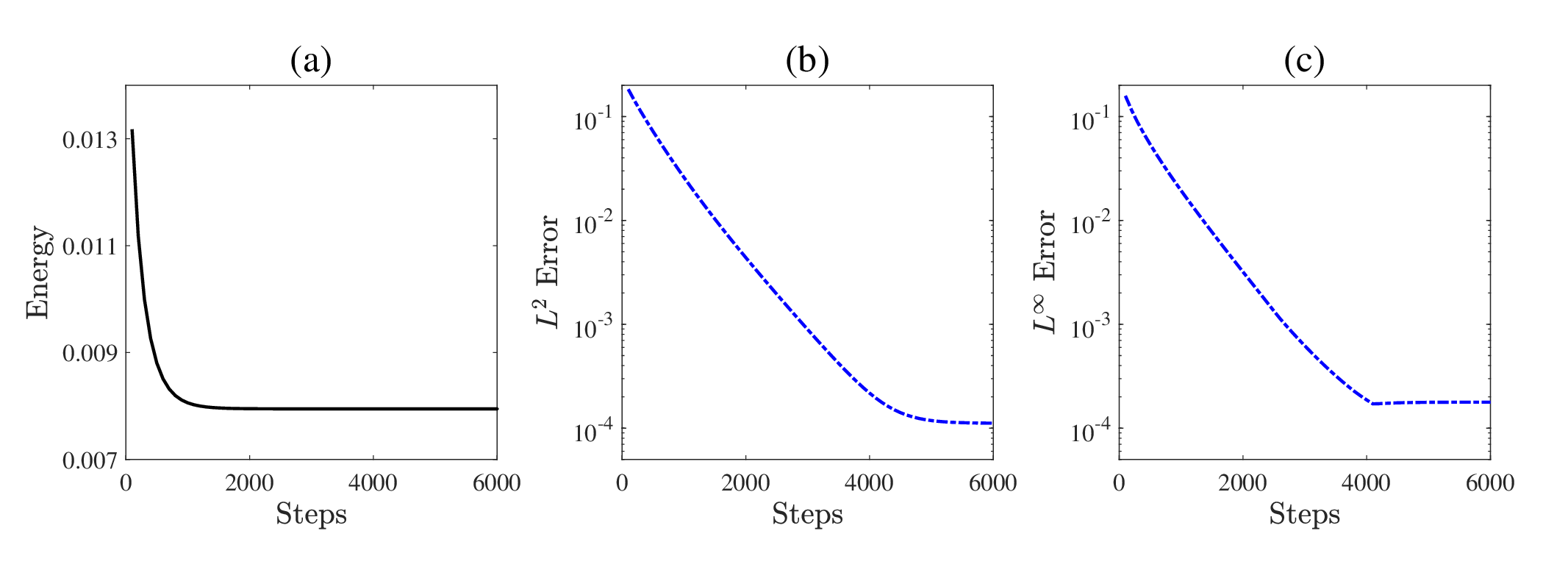}
  \vspace{-2 mm}
    	\caption{
   {\small {The discrete energy~(a), $L^2$-error~(b), and $L^{\infty}$-error~(c)   for 
 the displacement $D^{(k)}_h$ vs.\ the iteration step $k$ in the local algorithm 
with shift for Test 2.}}}
  	\label{f:ConRs2}
\end{figure}

\vspace{-3 mm}

\begin{figure}[h]
  	\centering
  	\includegraphics[width=3.2 in,height=2.1in]{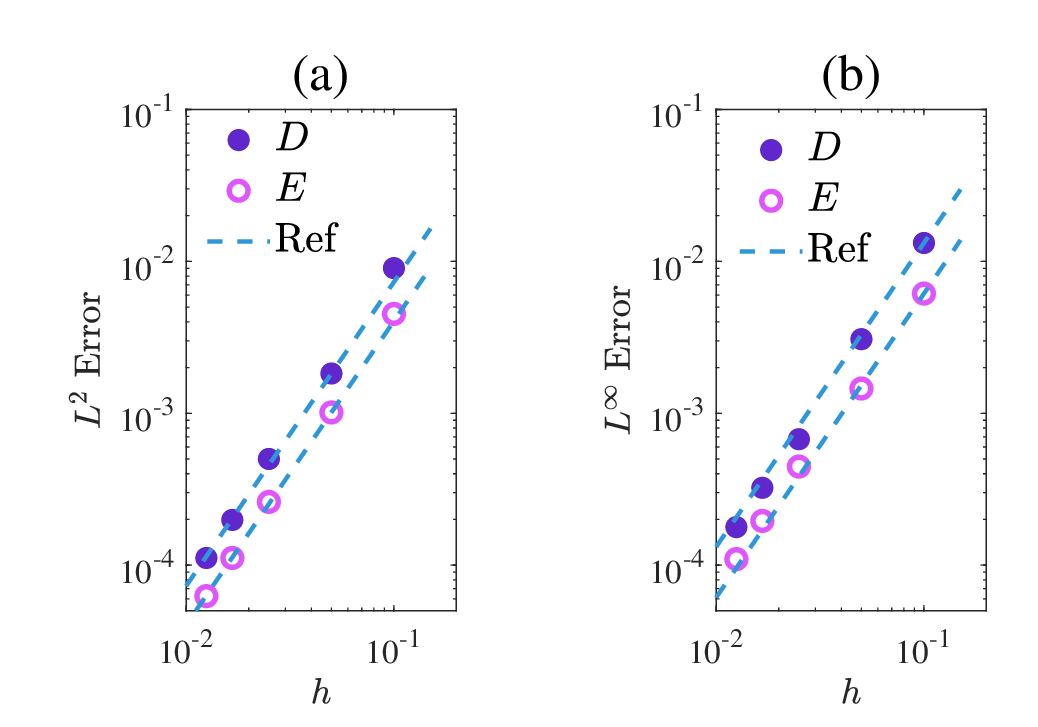}
     	\caption{{\small {Log-log plots of the $L^2$-error (a) and 
the $L^{\infty}$-error (b) for the approximation $D_h$
   of the displacement $D$ (marked $D$) and the reconstructed 
approximation $E_h:=m_h[D_h]/\ve$ of the
   electric field $E := -\nabla \phi$ (marked $E$) for Test 2. 
   The blue dashed lines (marked Ref) are reference lines indicating
  the $O(h^2)$ convergence rate.}}}
  	\label{f:DConOrder2} 
\end{figure}

\medskip

\noindent
{\it Test 3: The Poisson--Boltzmann (PB) energy with a variable permittivity.}
We define $M = 2$, $q_1 = - q_2 = 1$, and 
\begin{align*}
&\ve(x_1, x_2, x_3) = 3 - \cos (\pi x_1 \cos (\pi x_2) \cos (\pi x_3),
\\
&\phi(x_1, x_2, x_3) =-\cos (\pi x_1 ) \cos (\pi x_2)  \cos (\pi x_3), 
\\
& 
c_s = e^{-q_s \phi} \quad (s = 1, 2) \quad \mbox{and} \quad D = - \ve \nabla \phi, 
\\
& N_s= \int_\Omega e^{-q_s \phi} dx, \quad s = 1, 2, 
\\
& \rho (x) = - \nabla \cdot \ve \nabla \phi(x) -\sum_{s=1}^2
     N_s q_s \left( \int_\Omega e^{-q_s \phi(x} dx \right)^{-1} e^{-q_s \phi(x)}
\\
& \qquad =  - \nabla \cdot \ve \nabla \phi(x) -\sum_{s=1}^2 q_s e^{-q_s \phi(x)}, 
\end{align*}
where $x = (x_1, x_2, x_3).$
Note that we do not need to compute the integral that defines $N_s$.  
It can be verified that $\phi$ is the unique periodic solution
to the CCPBE \eqref{CCPBE}, Moreover,  
$(c, D) = (c_1, c_2; D) \in X_\rho$ is the unique minimizer of $\hat{F}: X_\rho \to \R \cup \{ + \infty \}.$
For a given finite-difference grid of size $h$, we denote by 
$(c_h, D_h) = (c_{1,h}, c_{2, h}; D_h)\in X_{\rho, h}$ the unique 
minimizer of the discrete PB energy functional $\hat{F}_h: X_{\rho, h}
\to \R$. We also denote by 
$(c_h^{(k)}, D_h^{(k)}) = (c_{1,h}^{(k)}, c_{1, h}^{(k)}; D_h^{(k)})$
$(k=0, 1, \dots)$ the iterates produced by the local algorithm. Figure~\ref{f:ConRs3} 
plots the discrete energy $\hat{F}_h [c^{(k)}_h, D_h^{(k)}]$, $L^2$-errors 
$\| c_{s} - c_{s, h} \|_h $ $(s = 1, 2)$ and $\| \sP_h D - D_h^{(k)}\|_h$, 
and $L^\infty$-errors 
$\| c_{s} - c_{s, h} \|_\infty $ $(s = 1, 2)$ and $\| \sP_h D - D_h^{(k)}\|_\infty$, 
 vs.\ the iteration step $k$  of local 
update with $h=L/N=2/160=0.0125$. We observe the monotonic decrease of all the
energy and errors. In fact, the errors converge to 
some values that are set by the grid size $h.$ 
In Figure~\ref{f:DConOrder3}, we plot in the log-log scale 
the $L^2$ and $L^\infty$ errors for the approximation $c_{s,h}$ of 
$c_s$ $(s = 1, 2)$ and $D_h$ of $D$, and also the approximation 
$E_h:=m_h[D_h]/\ve$ of the electric field $-\nabla \phi$, respectively, against the 
finite-difference grid size $h.$ We observe the $O(h^2)$ convergence rate as predicted
by Theorem~\ref{t:pbL2error} and Corollary~\ref{c:D4GradphiPB}.


\vspace{-2 mm}

\begin{figure}[h]
  	\centering
   \includegraphics[scale=0.37]{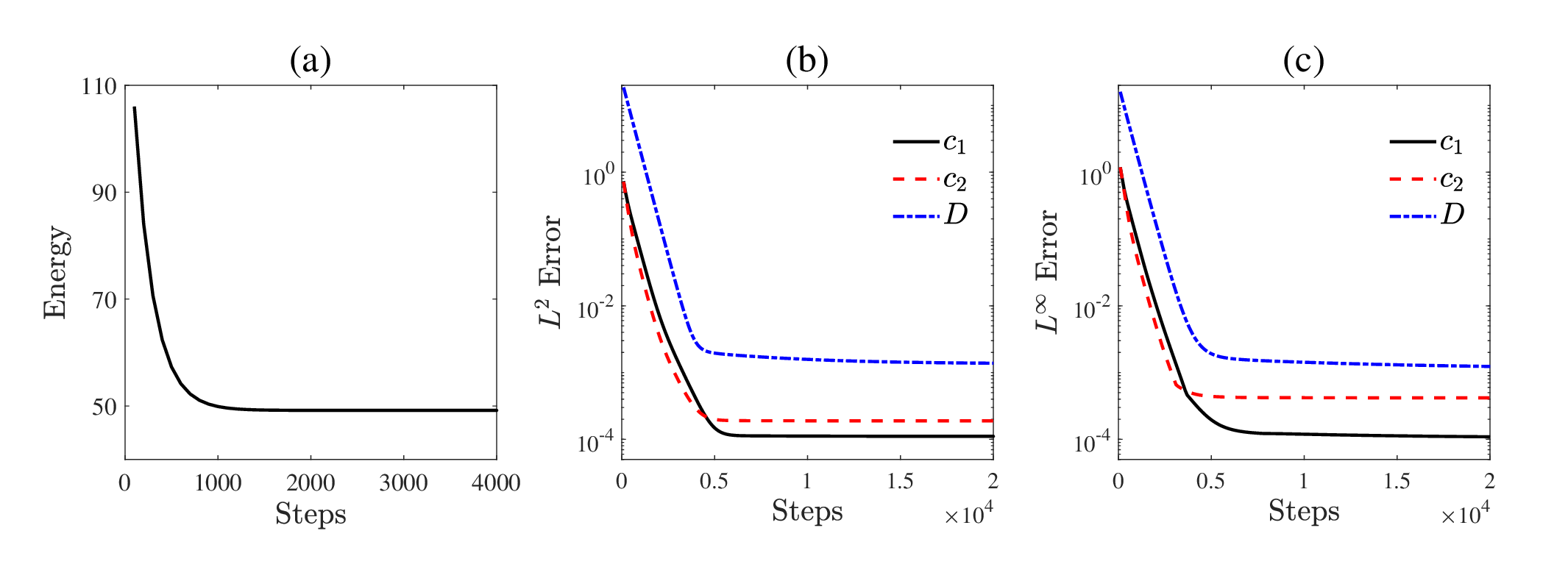}
   \vspace{-2mm}
  	\caption{\small{
   The discrete energy~(a), the $L^2$-error~(b), and the $L^{\infty}$-error~(c) for 
 the approximations $(c_h^{(k)}, D^{(k)}_h)$ vs.\ the iteration step $k$ 
in the local algorithm for Test 3.}}
  	\label{f:ConRs3}
\end{figure}

\vspace{-3 mm}

\begin{figure}[h]
  	\centering
  	\includegraphics[width=3.4in,height=2.3in]{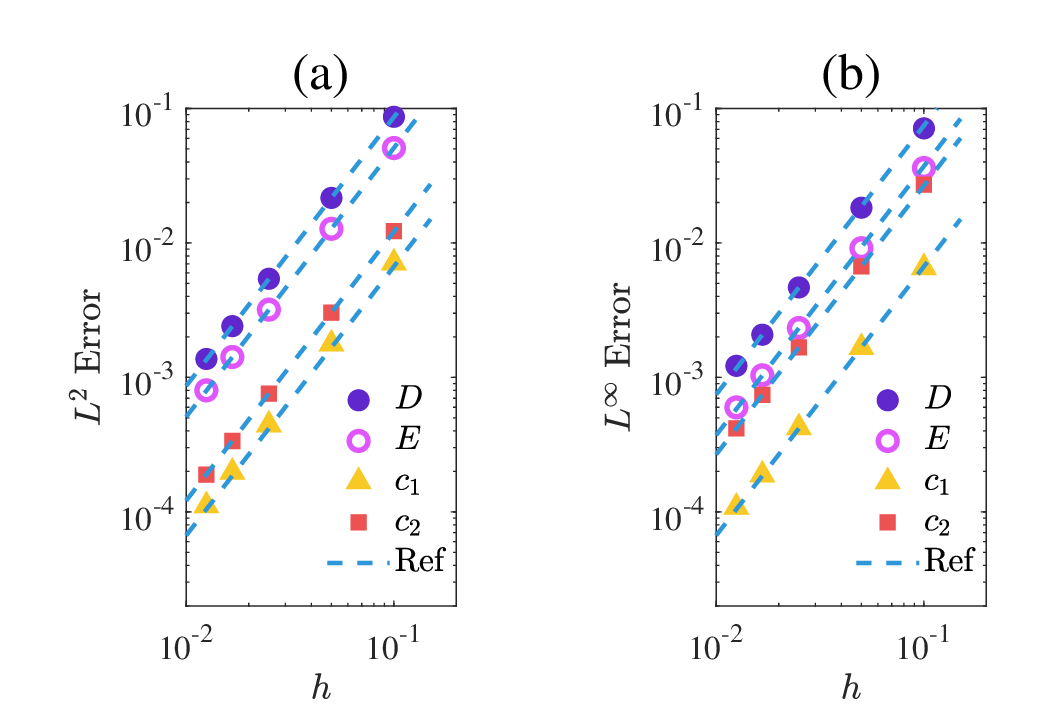}
   \vspace{-2mm}
  	\caption{\small{
   Log-log plots of the $L^2$-error (a) and the $L^{\infty}$-error (b) 
for the approximation 
   of $(c_h, D_h) = (c_{1,h}, c_{2, h}; D_h) $ of $(c, D) = (c_1, c_2, D)$  (marked
   $c_1$, $c_2$, and $D$), respectively, 
   and for the approximation $E_h:=m_h[D_h]/\ve$ of the
   electric field $E := -\nabla \phi$ (marked $E$) for Test 3. 
   The blue dashed lines (marked Ref) are reference lines indicating
   the $O(h^2)$ convergence rate.}}
  	\label{f:DConOrder3} 
\end{figure}

\addcontentsline{toc}{section}{Appendix}
\section*{Appendix}
\label{s:Appendix}
\renewcommand{\thesection}{A}
\setcounter{equation}{0}

\begin{proof}[Proof of Lemma~\ref{l:FourierBasis}]
The first discrete Green's identity  follows from an application of 
summation by parts and the periodicity. The second identity follows from the first one.

Let us use the symbol $\sqrt{-1}$ instead of $i$ to denote the imaginary unit: $\sqrt{-1}^2 = -1.$ 
For each grid point $(l, m, n)\in \Z^3,$ we define 
$\xi^{(l, m, n)}: h\Z^3 \to \C$ by 
\[
\xi^{(l,m,n)}_{i,j, k} = L^{-3/2} 
e^{\sqrt{-1}2\pi  l i/N}e^{\sqrt{-1}2\pi m j/N}
e^{\sqrt{-1} 2 \pi n k /N }\qquad \forall i, j, k \in \Z.
\]
The system
$\{ \xi^{(l, m, n)}: l,m,n=0, 1, \dots, N-1\}$ is an orthonormal basis for the space of
all complex-valued, $\overline{\Omega}$-periodic, grid functions with respect to 
the inner product $\langle \cdot, \cdot \rangle_h$ defined in \reff{phipsih}. 

Let $\phi: h\Z^3 \to \C$ be $\overline{\Omega}$-periodic 
and satisfy $\sA_\Omega(\phi) = 0.$ 
Since $\xi^{(0, 0, 0)}$ is a constant function and 
$\langle \phi, \xi^{(0, 0, 0)} \rangle_h = \sA_h(\phi) = 0$, we have 
\begin{align*}
&\phi_{i, j, k} = \sum_{l, m, n = 0}^{N-1} \langle \phi, 
\xi^{(l, m, n)} \rangle_h \xi^{(l, m, n)}_{i, j, k} 
= {\sum_{l, m, n}}' \langle \phi, \xi^{(l, m, n)} 
\rangle_h \xi^{(l, m, n)}_{i, j, k},  \quad 0 \le i, j, k \le N-1, 
\\
&\| \phi \|^2_h = \sum_{l,m,n=0}^{N-1} |\langle \phi, \xi^{(l,m,n)}  \rangle_h|^2
= {\sum_{l, m, n}}' |\langle \phi, \xi^{(l,m,n)}  \rangle_h|^2,
\end{align*}
where ${\sum'_{l, m, n}}$ denotes the sum over all $(l, m, n)$ such that
$0 \le l, m, n \le N-1$ and $(l, m, n) \ne (0, 0, 0).$
Hence, 
\[
\phi_{i+1, j, k}-\phi_{i, j, k} = {\sum_{l, m, n}}'
\langle \phi, \xi^{(l, m, n)} \rangle_h 
\xi^{(l, m, n)}_{i, j, k} \left( e^{\sqrt{-1} 2 \pi l /N} - 1\right).
\]
Consequently,  since $ \xi^{l,m,n} $ $(l,m,n = 0, \dots, N-1)$ 
are orthonormal, we have 
\begin{align*}
    & \sum_{i,j,k=0}^{N-1}( \phi_{i+1, j, k} - \phi_{i, j, k} )
    \overline{ ( \phi_{i+1, j, k} - \phi_{i, j, k})}
    \\
    & \quad =  \sum_{ i, j, k=0}^{N-1} 
    {\sum_{l,m,n}}'{\sum_{p, q, r}}' \langle \phi, \xi^{(l, m, n)}\rangle_h 
    \overline{\langle \phi, \xi^{(p, q, r)} \rangle_h }
    \xi^{(l, m, n)}_{i, j, k} \overline{ \xi^{(p, q, r)}_{i, j, k}} 
    \left( e^{\sqrt{-1} 2 \pi l /N} - 1\right)
    \overline{\left( e^{\sqrt{-1} 2 \pi p/N } - 1\right)}
    \\
    & \quad = 
       {\sum_{l,m,n}}'{\sum_{p, q, r}}' \langle \phi, \xi^{(l, m, n)}\rangle_h 
    \overline{\langle \phi, \xi^{(p, q, r)} \rangle_h }
    \left( e^{\sqrt{-1} 2 \pi l / N} - 1\right)
    \overline{\left( e^{\sqrt{-1} 2 \pi p/N } - 1\right)}
    \sum_{ i, j, k=0}^{N-1} 
       \xi^{(l, m, n)}_{i, j, k} \overline{ \xi^{(p, q, r)}_{i, j, k}} 
       \\
       & \quad = \frac{1}{h^3} 
       {\sum_{l,m,n}}' \left| \langle \phi, \xi^{(l, m, n)}\rangle_h \right|^2
    \left| e^{\sqrt{-1} 2 \pi l/N} - 1\right|^2
    \\
           & \quad = \frac{4}{h^3} 
       {\sum_{l,m,n}}' \left| \langle \phi, \xi^{(l, m, n)}\rangle_h \right|^2
    \sin^2 \left( \frac{\pi l}{N} \right),
\end{align*}
where we used the identity $ 1 - \cos ( 2 \pi l/N)=  2\sin^2 (\pi l/N ).$
Calculations for the differences $\phi_{i,j+1, k} - \phi_{i,j,k}$ and 
$\phi_{i,j,k+1} - \phi_{i,j,k}$ are similar. 

It now follows from 
\reff{dphidpsih} and the definition of $\nabla_h \phi$ that 
\[
\| \nabla_h \phi \|_h^2 = 
\frac{4}{h^2} {\sum_{l,m,n}}' \left| \langle \phi, \xi^{(l, m, n)}\rangle_h \right|^2
    \left[ \sin^2 \left( \frac{\pi l}{N} \right) + 
    \sin^2\left( \frac{\pi m}{N} \right) + \sin^2\left(
    \frac{\pi n}{N} \right) \right].
\]
Note that $\sin^2(\pi (N-1)/N)= \sin^2 (\pi / N)$ and that
$\sin x \ge (2/\pi) x $ if $x \in [0, \pi/2].$ Hence, if $1 \le l \le N-1$, then 
$ \sin^2({\pi l}/{N} ) \ge \sin^2 ( {\pi }/{N} ) \ge 
({2}/{N} )^2 = {4 h^2}/{L^2}.$
Finally, we have 
\[
\| \nabla_h \phi \|^2_h \ge  \frac{48}{L^2}  \sum_{l,m,n=0}^{N-1}
\left| \langle \phi, \xi^{(l, m, n)}\rangle_h \right|^2
= \left( \frac{4\sqrt{3}}{L} \right)^ 2 \| \phi \|^2_h,
\]
leading to the desired inequality. 
\end{proof}

\addcontentsline{toc}{section}{Acknowledgement}
\section*{Acknowledgment}

This work was supported in part by the US National Science Foundation
through the grant DMS-2208465 (BL), 
the National Natural Science Foundation 
of China through the grant 12171319 (SZ).
The authors thank Professor Burkhard D\"unweg for helpful discussions and thank
Professor Zhenli Xu for his interest in and support to this work. 
BL and QY thank Professor Zhonghua Qiao 
for hosting their visit to The Hong Kong Polytechnic University in the summer of 2023
where this work was initiated.


\addcontentsline{toc}{section}{References}


\bibliographystyle{plain}

\end{document}